\documentclass[12pt,article]{amsart}

\usepackage{amssymb}
\usepackage{graphicx}
\usepackage{amsfonts}
\usepackage{amsthm}
\usepackage{amsmath}
\usepackage{empheq}
\usepackage[usenames,svgnames,dvipsnames]{xcolor}
\usepackage{bbm}

\usepackage{subcaption}
\usepackage{caption}

\usepackage{array}
\usepackage{multirow}
\usepackage{subcaption}
\usepackage{epsfig}
\usepackage{epstopdf}

\usepackage{microtype}
	\usepackage{mathrsfs}
	\usepackage{amssymb}
	\usepackage{amsfonts}
	\usepackage{amsbsy}
	\usepackage{latexsym}
	\usepackage{amssymb,latexsym,amsmath,amsthm}
	\usepackage{xcolor}
	\usepackage{epsfig,psfrag,color}
	\usepackage[ruled,linesnumbered]{algorithm2e}
	\SetKwRepeat{Do}{do}{while}%
	\usepackage{graphicx}
\theoremstyle{plain}
\newtheorem{theorem}{Theorem}[section]
\newtheorem{lemma}[theorem]{Lemma}
\newtheorem{corollary}[theorem]{Corollary}
\theoremstyle{assumption}
\newtheorem{assumption}[theorem]{Assumption}
\theoremstyle{definition}

\newcommand{\BC}{ \color{blue} }

\newcommand{\Sp}{ \hspace{0.05cm} }

\newcommand{\x}{ {\bf x}}
\newcommand{\R}{\mathbb{R}}

\newcommand{\C}{\mathbb{C}}

\numberwithin{equation}{section}

\begin{document}

\title [Carleman-Fourier Linearization]
{\hspace{-1cm} Carleman-Fourier Linearization of Complex Dynamical Systems: Convergence and Explicit Error Bounds}

\author{Panpan Chen}
\address{Department of Mathematics, University of Central Florida, Orlando, Florida 32816}
\email{panpan.chen@ucf.edu}

\author{Nader Motee}
\address{Mechanical Engineering and Mechanics, Lehigh University,  Bethlehem, PA 18015}
\email{nam211@lehigh.edu}

\author{Qiyu Sun}
\address{Department of Mathematics, University of Central Florida, Orlando, Florida 32816}
\email{qiyu.sun@ucf.edu}


\begin{abstract}

 This paper presents a Carleman-Fourier linearization method for nonlinear dynamical systems with periodic vector fields involving multiple fundamental frequencies. By employing Fourier basis functions,  the  nonlinear dynamical system is transformed into a linear model on an infinite-dimensional space. The proposed approach yields accurate approximations over extended regions around equilibria and for longer time horizons,  compared to traditional Carleman linearization with monomials. Additionally, we develop a finite-section approximation for the resulting infinite-dimensional system and provide explicit error bounds that demonstrate exponential convergence to the original system's solution as the truncation length increases. For specific classes of dynamical systems, exponential convergence is achieved across the entire time horizon. The practical significance of these results lies in guiding the selection of suitable truncation lengths for applications such as model predictive control, safety verification through reachability analysis, and efficient quantum computing algorithms. The theoretical findings are validated through illustrative simulations.
  \end{abstract}

\maketitle

\section {Introduction}
\label{introduction.section}

Complex dynamical systems are characterized by their inherent nonlinearity, which leads to a wide spectrum of dynamical phenomena across various domains, including physical, biological, and engineering sciences. Despite significant advances, a comprehensive mathematical methodology for analyzing and designing nonlinear dynamical systems remains largely undeveloped. This gap makes the concept of lifting nonlinear systems to their linear counterparts particularly attractive, given the well-established and effective techniques available for analyzing and controlling linear systems, which are not easily adaptable to nonlinear contexts.

  Carleman linearization, originally formulated in 1932, is a powerful method for addressing the nonlinearities inherent in dynamical systems. It has emerged as a predominant technique for systematically converting nonlinear systems into linear forms \cite{Amini2022, amini2020approximate, amini2020quadratization, brockett2014early, Brunton2016, forets2017explicit, forets2021reachability, Korda2018, KordaMezic2018, kowalski1991nonlinear, minisini2007carleman, steeb1980non, surana2024, WangJungersOng2023, Wu2024}. The resurgence of interest in Carleman linearization is driven by significant advances in theoretical understanding, enhancements in numerical and algorithmic techniques, and increased access to large data sets \cite{Amini2022, Akiba2023, liu2021efficient}.

  The control systems community has witnessed several success stories stemming from Carleman linearization concepts \cite{amini2020approximate, amini2020quadratization, forets2021reachability, hashemian2015fast, krener1974linearization, krener1975bilinear, loparo1978estimating, pruekprasert2022moment, rauh2009carleman, rotondo2022towards}. For example, in \cite{rauh2009carleman}, Carleman linearization was used to design optimal control laws for nonlinear systems. In \cite{loparo1978estimating}, Carleman approximation was applied to establish a relationship between the lifted system and the domain of attraction of the original nonlinear system. Recent work \cite{hashemian2015fast} leveraged Carleman linearization for efficient implementation of model predictive control in nonlinear systems. Additionally, in \cite{minisini2007carleman}, Carleman linearization was employed for state estimation and feedback control law design. References \cite{amini2020approximate, amini2020quadratization} utilized the structure of the lifted system to develop a tractable approach for quadratizing and solving the Hamilton-Jacobi-Bellman equation using an exact iterative method. The effectiveness of Carleman linearization is largely attributed to its ability to transform complex nonlinear problems into linear ones, where well-established linear analysis tools can be applied. These transformations provide deeper insights into system behavior and enable the design of effective control strategies that would otherwise be challenging for inherently nonlinear dynamics. By leveraging linear control techniques, Carleman linearization opens up new possibilities for addressing a wide array of problems in the control of nonlinear dynamical systems.

  In this paper, we  adapt the cutting-edge tools and algorithms from traditional Carleman linearization to the Carleman-Fourier framework. This extension equips us with powerful tools for the analysis and control of nonlinear dynamical systems featuring periodic vector fields.
Our proposed Carleman-Fourier linearization is especially well-suited for lifting
 the following complex dynamical system
 \begin{subequations} \label{dynamicsystem}
\begin{equation} \label{complexdynamic.def}
\dot{\bf x}(t) = {\bf g}(t, {\bf x}), \ \ t\ge 0,
\end{equation}
for a state vector ${\bf x}(t) \in \C^d$ 
 evolving over time $t\ge 0$ from an initial  state ${\bf x}(0)= {\bf x}_0$,
 where the governing  vector field  ${\bf g}$ is periodic,
\begin{equation} \label{complexdynmaic.fdef}
{\bf g}(t, {\bf x}) \sim
\sum_{{\pmb \alpha}\in {\mathbb Z}^d} {\bf g}_{\pmb \alpha} (t) \hspace{0.05cm} e^{ i {\pmb \alpha} {\bf x}},
\end{equation}
\end{subequations}
 and  its Fourier coefficients
${\bf g}_{\pmb \alpha}(t)=[g_{1, {\pmb \alpha}}(t),\ldots, g_{d, {\pmb \alpha}} (t)]^T \in \C^d$ satisfies
\begin{assumption}\label{assump-1}
There exist positive constants $D_0$ and $R$ such that
\begin{equation}\label{assumption1}
    \sup_{t\geq 0}\sum_{|{\pmb \alpha}|=k}\sum_{j=1}^d |g_{j,   {\pmb \alpha}}(t)| \leq D_0 R^{-k}, \ k\ge 0.
\end{equation}
\end{assumption}
\noindent Here and thereafter, we  utilize notations ${\pmb \alpha}{\bf x}=\alpha_1 x_1+\dots +\alpha_d x_d$ and  $|{\pmb \alpha}|=|\alpha_1|+\dots+|\alpha_d|$ for   ${\bf x}=[x_1, \ldots, x_d]^{\BC T}\in {\mathbb C}^d$ and ${\pmb \alpha}=[\alpha_1, \dots, \alpha_d]\in \mathbb{Z}^d$.
With Assumption  \ref{assump-1} for the vector field ${\bf g}$,
its Fourier coefficients enjoy exponential decay when $R>1$, have exponential growth when $0<R<1$, and are bounded when $R=1$.
Throughout this paper, we refer to the following two illustrative examples of the nonlinear dynamical system \eqref{dynamicsystem} to demonstrate the applicability of our technical conditions and results. The first system is the scale-valued complex dynamical system
\begin{equation} \label{simpleexample2.eq1} \Dot{x} = a(1-b e^{ix}), \ t\ge 0 \end{equation}
with a governing vector field being a trigonometrical polynomial of degree one, where $a, b \in {\mathbb C}\backslash \{0\}$. The second example is the well-known first-order Kuramoto model, governed by
\begin{equation}\label{Kuramoto.def} \Dot\theta _{p}=\omega _{p}+{\frac {K}{d}}\sum _{q=1}^{d}\sin(\theta _{q}-\theta _{p}), \ 1\le p\le d, \end{equation}
where $\theta_p$ represents the phase of the $p$-th oscillator with natural frequency $\omega_p, 1\le p\le d$, and $K \ne 0$ signifies the coupling strength between the oscillators \cite{bronski2021, dietert2016, guo2021, heggli2019, ji2014, kuramoto1984}.

\smallskip

Let ${\mathbb Z}_{++}^d$ be the set of all nonzero \(d\)-tuples of nonnegative integers and set ${\bf x}^{\pmb\alpha}=x_1^{\alpha_1}\cdots x_d^{\alpha_d}$ for ${\bf x}=[x_1, \ldots, x_d]^{\BC T}\in {\mathbb C}^d$ and ${\pmb \alpha}=[\alpha_1, \dots, \alpha_d]^T\in {\mathbb Z}_{++}^d$. Traditional Carleman linearization utilizes state variable monomials ${\bf x}^{\pmb \alpha}$ for ${\pmb \alpha}\in {\mathbb Z}_{++}^d$ to lift finite-dimensional nonlinear dynamical systems to infinite-dimensional linear systems. While originally formulated for real dynamical systems, achieving sparse representations for complex dynamical systems remains challenging. In this paper, we begin by introducing Carleman linearization for the complex dynamical system \eqref{dynamicsystem} with the vector field ${\bf g}$ satisfying Assumption \ref{assump-1} for some $D_0>0$ and $R>1$; we refer to  \eqref{carlemannonhomo.eq3}.

We denote the standard $p$-norm of a vector ${\bf x} \in {\mathbb C}^d$ by $\|{\bf x}\|_p$ for $1 \le p \le \infty$. In Theorem \ref{maintheorem0}, we show that the first block ${\bf y}_{1, N}$ for $N \ge 1$ in the finite-section approximation \eqref{Athomogenous.def} of the Carleman linearization \eqref{carlemannonhomo.eq3} converges exponentially to the state vector ${\bf x}$ of the original dynamical system \eqref{dynamicsystem} within the time range $[0, T_C^*)$, provided that the complex dynamical system \eqref{dynamicsystem} has the origin ${\bf 0}$ as its equilibrium:
\begin{equation}\label{zeroequil.eq}
{\bf g}(t, {\bf 0}) = {\bf 0} \ \ {\rm for \ all} \ \ t \ge 0,
\end{equation}
and that the initial ${\bf x}_0$ is sufficiently close to the equilibrium:
\begin{equation} \label{maintheorem0.eq1} \|{\bf x}_0\|_{\infty} < e^{-1} \ln R, \end{equation}
where $R > 1$ and
\begin{equation}\label{Tstar0.def} T_C^* = \frac{(e - 1)(\ln R)^2}{(2e - 1) D_0 R} \ln\left(\frac{\ln R}{e \|{\bf x}_0\|_{\infty}}\right). \end{equation}
The exponential convergence rate is given by
\begin{equation}\label{CarlemanConvergencerate} r_{C}(t) = e^{D_0 R t/(\ln R)^2} \left(\frac{\|{\bf x}_0\|_{\infty} e}{\ln R}\right)^{(e - 1) / (2e - 1)} \end{equation}
for all $0 \le t < T_C^*$. The concept of equilibrium points is fundamental in the study of complex dynamical systems. However, many complex systems do not satisfy the equilibrium point requirement \eqref{zeroequil.eq} due to their intricate and often chaotic nature. For instance, condition \eqref{zeroequil.eq} is satisfied only if $b = 1$ in our illustrative dynamical system \eqref{simpleexample2.eq1} and if all intrinsic natural frequencies $\omega_p$ for $1 \le p \le d$ are zero in the Kuramoto model \eqref{Kuramoto.def}.

\smallskip

In this paper, we propose using the Fourier representation of periodic vector fields alongside traditional Carleman linearization techniques to capitalize on the periodicity of the governing vector field ${\bf g}$ in the complex dynamical system \eqref{dynamicsystem}. This approach leverages the inherent structure of periodic vector fields to enhance both the parsimony and interpretability of the resulting embedding. Specifically, we utilize Fourier basis functions $w_{\pmb \alpha} = e^{i{\pmb \alpha} {\bf x}}, {\pmb \alpha} \in {\mathbb Z}_{++}^d$, instead of monomials ${\bf x}^{\pmb \alpha}, {\pmb \alpha} \in {\mathbb Z}_{++}^d$ in Carleman linearization, to lift the complex dynamical system \eqref{dynamicsystem} with the periodic vector field ${\bf g}$, which satisfies Assumption \ref{assump-1} for some $D_0, R > 0$ and is analytic on the shifted upper half-plane, that is
\begin{equation}
\label{assumption0}
    {\bf g}_{{\pmb \alpha}}(t)={\bf 0}\  \ {\rm for \ all} \ {\pmb \alpha}\in {\mathbb Z}^d\backslash {\mathbb Z}^d_+ \ {\rm and }\  t\ge 0.
\end{equation}
Here, ${\mathbb Z}_{+}^d = {\mathbb Z}_{++}^d \cup \{{\bf 0}\}$ represents the set of all $d$-tuples of nonnegative integers. We refer to this Fourier-based lifting scheme as {\bf Carleman-Fourier linearization}; see \eqref{Carleman.eq5}.

Careful handling of the resulting infinite-dimensional linear system \eqref{Carleman.eq5} is essential, as the corresponding state matrix ${\bf B}(t)$ does {\bf not} represent a bounded operator on $\ell^2({\mathbb Z}_{++}^d)$, the Hilbert space of all square-summable sequences on ${\mathbb Z}_{++}^d$. This limitation prevents us from directly applying existing theory to analyze the linear system derived from Carleman-Fourier linearization and thereby gain insights into the original dynamical system. By observing the upper-triangular structure of the state matrix ${\bf B}(t)$ and following the approach in \cite{Amini2022}, we demonstrate that the logarithm of the first block of the finite-section approximation to the infinite-dimensional system \eqref{Carleman.eq5} converges exponentially to a multiple of the state vector ${\bf x}$ over the time interval $[0, T_{CF}^*)$, provided that the initial state vector ${\bf x}_0 := [x_{0,1}, \ldots, x_{0,d}]^T \in {\mathbb C}^d$ satisfies
  \begin{equation} \label{maintheoremanalytic.thm1.eq1}
  \|\exp(i{\bf x}_0)\|_\infty= \exp\big(-\min_{1\le j\le d}\Im {x}_{0, j}\big)<e^{-1} R,
 \end{equation}
 where $\exp(i{\bf x}_0)$ is the exponential of the initial state  ${\bf x}_0$, and
 \begin{equation} \label{maintheoremanalytic.thm1.eq4}
    T^*_{CF} = \frac{e-1}{D_0(2e-1)} \big(\ln R-\ln  \|\exp(i{\bf x}_0)\|_\infty  - 1 \big)>0.
\end{equation}
Here and throughout, we denote the imaginary and real parts of $x \in \mathbb{C}$ by $\Im x$ and $\Re x$, respectively. Furthermore, the exponential convergence rate is given by
 \begin{equation} \label{convergencerateCarlemanFourier}
r_{CF}(t) = e^{ D_0t}
 \left(\frac{e
 \|\exp(i{\bf x}_0)\|_\infty}{R }\right)^{(e-1)/(2e-1)}
\end{equation}
for all $0\le t< T_{CF}^*$; see Theorem \ref{maintheoremanalytic.thm1} and Corollary \ref{maintheoremanalytic.cor1}.

In this paper, we also establish the exponential convergence of the finite-section approximation of the Carleman-Fourier linearization \eqref{Carleman.eq5} over the entire time range $[0, \infty)$ with an exponential convergence rate given by
\begin{equation} \label{ConvergencerateCarlemanFourierWhole}
\tilde r_{CF}= \frac{(D_0+\mu_0)\|\exp(i{\bf x}_0)\|_2}{\mu_0 R}<1
\end{equation}
 under the assumption that
 the Fourier coefficient ${\bf g}_{\bf 0}(t)$
 of the vector field ${\bf g}$
  has strictly positive  imaginary parts, i.e., there exists a  positive constant $\mu_0$ such that
    \begin{equation}
\label{assumption2} \min_{1\le j\le d} \Im  g_{j, {\bf 0}}(t)\geq \mu_0 >0
  \end{equation}
for all $\ t\ge 0$,     and     the initial state vector ${\bf x}_0$  satisfies
     \begin{equation}\label{maintheoremanalytic.thm2.eq1}
\|\exp(i{\bf x}_0)\|_2< \frac{\mu_0 R} {D_0+\mu_0};
 \end{equation}
see Theorem \ref{maintheoremanalytic.thm2} and Corollary \ref{maintheoremanalytic.cor2}.
We remark that, under the assumptions \eqref{assumption2} and \eqref{maintheoremanalytic.thm2.eq1} on the governing vector field ${\bf g}$ and the initial state ${\bf x}_0$, the imaginary parts of all components of the state vector ${\bf x}(t)$ diverge to positive infinity as $t \to \infty$. Additionally, the dynamical system associated with the new state vector $e^{i{\bf x}}$ (the exponential of the original state vector ${\bf x}$) is stable; see Lemma \ref{maintheoremanalytic.thm2.lem1}.

\smallskip

A broad category of complex dynamical systems \eqref{dynamicsystem} does not meet the analyticity requirement \eqref{assumption0}. For example, the vector field corresponding to the Kuramoto model \eqref{Kuramoto.def} fails to satisfy this condition. To lift such nonlinear dynamical systems into the realm of infinite-dimensional linear dynamical systems, we propose a novel approach by introducing an augmented state vector that includes both the state vector ${\bf x}$ and its negative, $-{\bf x}$. This leads to the formation of the extended state vector $\tilde {\bf x}=[{\bf x}^T, -{\bf x}^T]^T$ and a corresponding dynamical system for the extended state vector; see \eqref{function.multiple.def1}. We demonstrate that the complex dynamical system \eqref{function.multiple.def1} associated with the extended state vector $\tilde {\bf x}$ has a governing vector field $\tilde {\bf g}$ with Fourier coefficients that decay exponentially at a uniform rate, fulfilling the criteria in \eqref{assumption0}; see \eqref{function.multiple.eq4} and \eqref{function.multiple.eq5}. Following the lifting scheme in Section \ref{subsec:Carleman-Fourier}, we introduce the Carleman-Fourier linearization of the complex dynamical system \eqref{dynamicsystem} with the corresponding vector field ${\bf g}$ satisfying Assumption \ref{assump-1} for some $D_0 > 0$ and $R > 1$; see \eqref{Carleman.multiple.def}. In this paper, we show that the finite-section approximation \eqref{Carleman.multiple.eq8} of the Carleman-Fourier linearization \eqref{Carleman.multiple.def} provides an exponential approximation to the solution ${\bf x}$ of the original complex dynamical system \eqref{dynamicsystem} over the time range $[0, \tilde T^*_{CF})$, provided that certain conditions are met.
  \begin{equation} \label{carlemanfourier.generalrequirement}
  e<R~~ {\rm and} ~~ \max \{ \|\exp(i{\bf x}_0)\|_\infty, \|\exp(-i{\bf x}_0)\|_\infty\} < R/e,
  \end{equation}
  where
  \begin{equation}\label{carlemanfourier.generaltimerequirement}
  \tilde T_{CF}^*=\frac{e-1}{2D_0(2e-1)} \big(\ln R-1-\ln \max \{\|\exp(i{\bf x}_0)\|_\infty,  \|\exp(-i{\bf x}_0)\|_\infty\}\big).
  \end{equation}
  The exponential  convergence rate is given by
  \begin{equation}\label{carlemanfourier.generalconvergencerate}
  \tilde r_{CF}(t)=e^{2D_0t} \left(\frac{e\max \big\{ \|\exp(i{\bf x}_0)\|_\infty, \|\exp(-i{\bf x}_0)\|_\infty\big\}}{R}\right)^{(e-1)/{(2e-1)}};
  \end{equation}
  see \eqref{maintheoremanalytic.cor3.eq2}, Theorem \ref{maintheoremanalytic.thm3}, and
 Corollary \ref{maintheoremanalytic.cor3}.

  \noindent {\bf Main contributions.} The primary contributions of this paper are summarized as follows.
\begin{itemize}
 \item[{(i)}] Originally developed for real dynamical systems, Carleman linearization is a powerful method for analyzing system behavior near the origin. In this paper, we extend Carleman linearization to complex dynamical systems described by \eqref{dynamicsystem} and demonstrate that the first block of the finite-section approximation of this linearization converges exponentially to the state vector of the original system within a specific time range. Our theoretical convergence results in Section \ref{carleman.subsection} and numerical demonstrations in Section \ref{demonstration.section} indicate that traditional Carleman linearization provides exceptionally accurate linearization for complex dynamical systems over a certain time interval, particularly when the initial state is close to the origin, where low-degree polynomial terms dominate the system's dynamics.
  \item[{(ii)}]  Given that the periodic vector field ${\bf g}$ of the complex dynamical system \eqref{dynamicsystem} can be well-approximated by trigonometric polynomials, a natural approach to linearization is to use exponentials $e^{i{\pmb \alpha} {\bf x}}$ for $  {\pmb \alpha}\in {\mathbb Z}^d\backslash \{0\}$, lifting the complex dynamical system \eqref{dynamicsystem} into an infinite-dimensional linear dynamical system. However, the state matrix of the resulting system does not represent a bounded operator on the Hilbert space of square-summable sequences and, unlike in traditional Carleman linearization, it lacks a block upper-diagonal structure \cite{moteesun2024}. We observe that, under the additional analyticity condition \eqref{assumption0} on the vector field ${\bf g}$, the dynamical system associated with the exponential ${\bf w}:=e^{i{\bf x}}$ of the state vector ${\bf x}$ in the complex dynamical system \eqref{dynamicsystem} has a governing vector field that is analytic with respect to the new variable ${\bf w}$ in a small neighborhood of the origin; see \eqref{complexdynamic.def2}. Based on this observation, we use the Fourier basis $e^{i{\pmb \alpha} {\bf x}}$ for ${\pmb \alpha}\in {\mathbb Z}_{++}^d$ to lift the complex dynamical system \eqref{dynamicsystem}, with its governing vector field satisfying the analyticity condition \eqref{assumption0}, into an infinite-dimensional linear dynamical system with a state matrix that has a block upper-diagonal structure. We refer to this lifting scheme \eqref{Carleman.eq5} as the Carleman-Fourier linearization; see Section \ref{subsec:Carleman-Fourier}.

  \item[{(iii)}]   Utilizing the Fourier system $e^{i{\pmb \alpha} {\bf x}}$ for ${\pmb \alpha} \in {\mathbb Z}_{++}^d$, instead of monomials ${\bf x}^{\pmb \alpha}$ for ${\pmb \alpha} \in {\mathbb Z}_{++}^d$, often yields a sparse representation of the periodic vector field ${\bf g}$. This approach leads to a ``sparse'' representation in the lifting of the dynamical system \eqref{dynamicsystem}. The effective use of Fourier basis functions proves highly capable of capturing both periodic and nonlinear behaviors in the complex dynamical system \eqref{dynamicsystem}. In this paper, we assess the accuracy of linearized models through finite-section approximations of Carleman-Fourier linearization. We demonstrate that truncating the infinite-dimensional linear systems obtained from Carleman-Fourier linearization at a suitably large length provides accurate approximations of the original complex dynamical system \eqref{dynamicsystem}. As shown in Theorems \ref{maintheoremanalytic.thm1} and \ref{maintheoremanalytic.thm2} and verified through numerical simulations in Section \ref{demonstration.section}, we establish that the state vector ${\bf x}$ of the complex dynamical system \eqref{dynamicsystem}, with a governing vector field satisfying \eqref{assumption0}, can be approximated exponentially using a finite-section approach to the Carleman-Fourier linearization \eqref{Carleman.eq5}.

\item[{(iv)}] The analytic requirement \eqref{assumption0} is not always satisfied for the governing vector field of a complex dynamical system. For instance, the first-order Kuramoto model \eqref{Kuramoto.def}, which has been widely used to analyze the dynamical behaviors of coupled oscillators, does not meet this criterion. We observe, however, that for the complex dynamical system \eqref{dynamicsystem}, the extended state vector $\tilde {\bf x} = [{\bf x}^T, -{\bf x}^T]$ satisfies a dynamical system whose governing periodic vector field meets the analytic condition \eqref{assumption0}. We then expand our approach in Section \ref{sec:finite-section}  to incorporate Carleman-Fourier linearization and its finite-section approximations for complex dynamical systems \eqref{dynamicsystem} with governing periodic vector fields that exhibit exponentially decaying Fourier coefficients; see Section \ref{multiple.section}. Consequently, we show that the finite-section approximation of the Carleman-Fourier linearization achieves exponential convergence when the initial state ${\bf x}_0$ is real-valued and the exponential decay rate $1/R$ for the Fourier coefficients of the periodic vector field ${\bf g}$ is strictly less than $1/e$. We note that a similar conclusion regarding exponential convergence is established in \cite{moteesun2024} when the initial state ${\bf x}_0$ and the periodic vector field are real-valued, and the exponential decay rate $1/R$ for the Fourier coefficients of the periodic vector field is strictly less than $1$.

\item [{(v)}]
The Carleman-Fourier linearization presented in Sections \ref{sec:finite-section} and \ref{multiple.section}  offers several key advantages. We observe that the requirements, time range, and convergence ratio for the exponential convergence of the finite-section approximation to the proposed Carleman-Fourier linearization depend solely on the imaginary parts $\Im {\bf x}_0$ of the initial state ${\bf x}_0$. This dependency arises because the governing vector field ${\bf g}$ can be well-approximated by trigonometric polynomials when the imaginary part of the state vector remains close to the origin, while the real parts of the state vector can be chosen freely. 
This capability is especially valuable for examining system behavior beyond the immediate vicinity of equilibrium. Carleman-Fourier linearization provides precise linearizations for systems with periodic vector fields over larger neighborhoods around the equilibrium point (the origin), surpassing traditional Carleman linearization in this regard. Specifically, for the complex dynamical system \eqref{dynamicsystem} with a governing vector field ${\bf g}$ satisfying Assumption \ref{assump-1} with $D_0 > 0$ and $R \ge e^{e/(e-1)} \approx 4.8646$, we have:
     $$\qquad \quad  \big\{{\bf x}_0\in {\mathbb C}^d\ \big|\  \|{\bf x}_0\|_\infty< e^{-1}\ln R\big\}
          \subset   \big\{{\bf x}_0\in {\mathbb C}^d\ \big|\  \max\{\|e^{i{\bf x}_0}\|_\infty, \|e^{-i{\bf x}_0}\|_\infty\}    <R/e\big\},
                $$
cf. \eqref{maintheorem0.eq1} and \eqref{carlemanfourier.generalrequirement}. Moreover, if the initial state ${\bf x}_0$ is far from the origin, such as $1 \le \|{\bf x}_0\|_\infty < e^{-1}\ln R$, the finite-section approximation of Carleman-Fourier linearization achieves greater accuracy over extended time intervals compared to Carleman linearization: \begin{equation}\label{comparison.eq00} T_C^* \le \tilde T_{CF}^*\ \ {\rm and} \ \ \tilde r_{CF}(t)\le r_C(t)
\end{equation}
for all $0\le t\le T_C^*$, where the exponential convergence ranges $T_C^*$ and $\tilde T_{CF}^*$ are defined in \eqref{Tstar0.def} and \eqref{carlemanfourier.generaltimerequirement}, and the convergence rates $r_C(t)$ and $\tilde r_{CF}(t)$ are given in \eqref{CarlemanConvergencerate} and \eqref{carlemanfourier.generalconvergencerate}, respectively; see Section \ref{comparison.eq00.pfsection} for a detailed proof. Consequently, the proposed Carleman-Fourier linearization provides a more effective approximation to the complex dynamical system \eqref{dynamicsystem} than traditional Carleman linearization, significantly enhancing system predictability over larger regions and longer time intervals; see the numerical simulation in Section \ref{kuramoto.section} for the Kuramoto model. This improvement is crucial for achieving a comprehensive understanding and management of dynamical systems governed by periodic vector fields, resulting in analyses that are more robust and reliable, especially for systems where accurate long-term predictions are essential.

\end{itemize}

\vspace{0.5cm}

\noindent {\bf Organization.} In Section \ref{carleman.subsection}, we examine the Carleman linearization of the complex dynamical system \eqref{dynamicsystem} with a periodic vector field that satisfies Assumption \ref{assump-1} for some $D_0 > 0$ and $R > 1$. We establish the exponential convergence of the finite-section approximation of the lifted infinite-dimensional linear dynamical system \eqref{carlemannonhomo.eq3}. In Section \ref{sec:finite-section}, we introduce Carleman-Fourier linearization for the complex dynamical system \eqref{dynamicsystem} when the periodic vector field ${\bf g}$ meets the analyticity condition \eqref{assumption0} and Assumption \ref{assump-1} for some $D_0 > 0$ and $R > 0$. We also prove the exponential convergence of the finite-section approximation of the Carleman-Fourier linearization and show that this convergence is effective over both short and extended time ranges. In Section \ref{multiple.section}, we extend the methodology developed in Section \ref{sec:finite-section}  to apply Carleman-Fourier linearization and its finite-section approximations to nonlinear dynamical systems with periodic vector fields that exhibit multiple fundamental frequencies and exponentially decaying Fourier coefficients. Section \ref{demonstration.section} provides examples demonstrating the efficacy of Carleman-Fourier linearization, Carleman linearization, and their finite-section approximations for our illustrative complex dynamical systems \eqref{simpleexample2.eq1} and \eqref{Kuramoto.def}. All proofs are gathered in Section \ref{proof.section}.

\section{Carleman Linearization using Monomials}
\label{carleman.subsection}

In this section,  we introduce  Carleman linearization  \eqref{carlemannonhomo.eq3} for the complex dynamical system
\eqref{dynamicsystem} with the vector field ${\bf g}$ satisfying Assumption \ref{assump-1} for some  $D_0>0$ and  $R>1$.
Under the above assumption on the vector field ${\bf g}$, its Fourier coefficients ${\bf g}_{\pmb \alpha}(t), {\pmb \alpha}\in {\mathbb Z}^d$,
enjoy exponential decay, and its  Fourier expansion \eqref{complexdynmaic.fdef} converges to the vector field ${\bf g}$
if  the  state vector ${\bf x}$ satisfies 
 \begin{equation}\label{assumption0-2}
 \max \big\{\|\exp(i{\bf x})\|_\infty, \|\exp(-i{\bf x})\|_\infty\big\}< R,
 \end{equation}
 cf.  \eqref{maintheorem0.eq1} and \eqref{carlemanfourier.generalrequirement} on the initial vector ${\bf x}_0$
 of the complex dynamical system \eqref{dynamicsystem}.

Carleman linearization   is predominantly suitable for systems with the corresponding vector fields being effectively approximated with low-degree polynomials.
This well-approximation property
allows for finite-section approximations of the resulting infinite-dimensional linear system with minimal errors
 \cite{Amini2022, amini2021error, forets2017explicit, forets2021reachability, liu2021efficient, minisini2007carleman, pruekprasert2022moment, rauh2009carleman, surana2024,  Wu2024}.  The main result of this section is  Theorem \ref{maintheorem0}, where we show that the first block  in the finite-section approximation
 \eqref{Athomogenous.def} of the Carleman linearization has exponential convergence
if the vector field ${\bf g}$ and the initial ${\bf x}_0$ satisfy
\eqref{zeroequil.eq}  and \eqref{maintheorem0.eq1} respectively.

\medskip

Using the  Maclaurin expansion of exponential function $e^{i{\pmb \alpha}{\bf x}}=\sum_{{\pmb \beta}\in {\mathbb Z}_+^d} i^{|{\pmb \beta}|}
 {\pmb \alpha}^{\pmb \beta}{\bf  x}^{\pmb \beta} /{\pmb \beta}!$, one can rewrite the dynamical system \eqref{complexdynamic.def} as
\begin{eqnarray} \label{polynomial.def1}
\dot{\x} = {\bf g}(t, {\bf x}) &\hskip-0.08in =  & \hskip-0.08in
    \sum_{{\pmb \beta}\in {\mathbb Z}_+^d} \frac{i^{|{\pmb \beta}|}}{{\pmb \beta}!}
\left(~\sum_{{\pmb \alpha}\in {\mathbb Z}^d} {\bf g}_{\pmb \alpha}(t) \Sp {\pmb \alpha}^{\pmb \beta}~\right) {\bf x}^{\pmb \beta}
=:\sum_{{\pmb \beta}\in {\mathbb Z}^d_+}  {\bf f}_{\pmb \beta}(t) \Sp {\bf x}^{\pmb \beta}
\end{eqnarray}
with initial condition ${\bf x}(0)={\bf x}_0\in {\mathbb C}^d$.
The Maclaurin expansion of the vector field ${\bf g}$ is well-defined
and
\begin{eqnarray*}
 \sum_{{\pmb \beta}\in {\mathbb Z}^d_+}  \|{\bf f}_{\pmb \beta}(t)\|_1 |{\bf x}^{\pmb \beta}|
& \hskip-0.08in \le & \hskip-0.08in  \sum_{{\pmb \beta}\in {\mathbb Z}_+^d} ({\pmb \beta}!)^{-1}
\sum_{{\pmb \alpha} \in {\mathbb Z}^d} \|{\bf g}_{\pmb \alpha}(t)\|_1   |{\pmb \alpha}^{\pmb \beta}| |{\bf  x}^{\pmb \beta}|\le
 \sum_{{\pmb \alpha} \in {\mathbb Z}^d}\|{\bf g}_{\pmb \alpha}(t)\|_1  e^{|{\pmb  \alpha}| \|{\bf x}\|_\infty}\\
& \hskip-0.08in \le & \hskip-0.08in  D_0 \sum_{k=0}^\infty  R^{-k} e^{k \|{\bf x}\|_\infty}=  \frac{D_0} {1- R^{-1} \exp(\|{\bf x}\|_\infty)} <\infty,
 \end{eqnarray*}
  provided that the state vector ${\bf x}$ 
  satisfies
\begin{equation} \label{assumption0+2}\|{\bf x}\|_\infty<\ln R,\end{equation}
 cf. the requirement  \eqref{maintheorem0.eq1}
on the initial state  ${\bf x}_0$
  of the  
  dynamical system
\eqref{dynamicsystem}.

 Under Assumption \ref{assump-1} for the vector field ${\bf g}$ with $D_0>0$ and $R>1$, one may verify that its
  Maclaurin coefficients  ${\bf f}_{\pmb \beta}(t), {\pmb \beta}\in {\mathbb Z}_+^d$ in \eqref{polynomial.def1} have the uniform exponential decay property, cf. \cite[Assumption 2.1]{Amini2022}. In particular, we have
\begin{subequations} \label{polynomial.def2}
   \begin{equation} \label{polynomial.def2b}
   \|{\bf f}_{{\bf 0}}(t)\|_1\le  \sum_{{\pmb \alpha}\in {\mathbb Z}^d} \|g_{ {\pmb \alpha}}(t)\|_1
 \le D_0\sum_{l=0}^\infty  R^{-l}= \frac{D_0 R}{R-1}\le  \frac{D_0 R}{\ln R}
 \end{equation}
 and
\begin{eqnarray} \label{polynomial.def2a}
   \sum_{{\pmb \beta}\in {\mathbb Z}_{+, k}^d}
\|{\bf f}_{ {\pmb \beta}}(t)\|_1 &\hskip-0.08in  \le & \hskip-0.08in
 \sum_{{\pmb \beta}\in {\mathbb Z}_{+, k}^d} \frac{1}{{\pmb \beta}!}
\sum_{l=1}^\infty \sum_{{\pmb \alpha}\in {\mathbb Z}_l^d}
\|{\bf g}_{{\pmb \alpha}}(t)\|_1 |{\pmb  \alpha}^{\pmb \beta}|\nonumber\\
& \hskip-0.08in \le  & \hskip-0.08in 
  D_0\sum_{l=1}^\infty \frac{l^k}{k!}  R^{-l}  
\le   \frac{D_0 R}{(\ln R)^{k+1}},\ \  k\ge 1.
\end{eqnarray}
 \end{subequations}

    For a given $k\ge 0$, let  ${\mathbb Z}_{ k}^d$ (resp.  $\mathbb{Z}_{+,k}^d:={\mathbb Z}_+^d\cap {\mathbb Z}_k^d $)  be the set of all (resp. nonnegative) integer indices ${\pmb \alpha}\in {\mathbb Z}^d$ of order $|{\pmb \alpha}|=k$.
Define  the new state variables ${\bf z}_k=[{\bf x}^{\pmb \alpha}]_{{\pmb \alpha}\in {\mathbb Z}_{+,k}^d}$, which contains all the monomials of order $k\ge 1$, and
 block matrices  ${\bf A}_{k, l}(t), k, l\ge 1$,  of size $ { {k+d-1}\choose d-1}\times { {l+d-1}\choose d-1}$ by
\begin{equation}  \label{carleman.eq2}
	{\bf A}_{k, l}(t)= \Big[\sum_{j=1}^d \alpha_j {\bf  f}_{j, \pmb\beta-\pmb\alpha+{\bf e}_j}(t)\Big]_{{\pmb \alpha}\in {\mathbb Z}_{+,k}^d,~ {\pmb \beta}\in {\mathbb Z}_{+,l}^d},
\end{equation}
where we set
${\bf f}_{\pmb \beta}(t)={\bf 0}$ for $ {\pmb \beta}\in {\mathbb Z}^d\backslash {\mathbb Z}_+^d$, write ${\bf f}_{\pmb \beta}(t)=[f_{1, {\pmb \beta}}(t), \ldots, f_{d, {\pmb \beta}}(t)]^T$ and ${\pmb \alpha}=[\alpha_1, \ldots, \alpha_d]^T\in {\mathbb Z}^d$.
With the uniform exponential decay property \eqref{polynomial.def2}
for the Maclaurin coefficients of the vector field ${\bf g}$,
 we propose the  {\bf Carleman linearization} of  the complex dynamical system \eqref{dynamicsystem}  as follows:
\begin{equation}  \label{carlemannonhomo.eq3}
	\dot{\bf z}(t)= {\bf A}(t) {\bf z}(t)+ {\bf b}(t), \ t\ge 0,
\end{equation}
 with the initial ${\bf z}(0)=[{\bf z}_k(0)]_{k\ge 1}$, where
${\bf z}(t)=[{\bf z}_1^T(t), {\bf z}_2^T(t), \ldots,  {\bf z}_N^T(t), \ldots]^T$ is
the infinite-dimensional state vector,
the  nonhomogeneous term ${\bf b}=\big[({\bf f}_{\bf 0}(t))^T, {\bf 0}, \ldots\big]^T$ is determined by the first
Maclaurin coefficient ${\bf f}_{\bf 0}(t)={\bf g}(t, {\bf 0})$,
and
\begin{equation}\label{carleman.matrixA}
{\bf A}(t)
= \left[\begin{array}{ccccccc}
		{\bf A}_{1, 1}(t) & {\bf A}_{1, 2}(t) & {\bf A}_{1,3}(t) & \cdots  &  {\bf A}_{1, N-1}(t)  & {\bf A}_{1, N}(t) & \cdots\\
			  {\bf A}_{2, 1}(t)
		& {\bf A}_{2, 2}(t)& {\bf A}_{2, 3}(t) & \cdots & {\bf A}_{2, N-1}(t) & {\bf A}_{2, N}(t) & \cdots\\
		& {\bf A}_{3, 2}(t)& {\bf A}_{3, 3}(t) & \cdots &  {\bf A}_{3, N-1}(t) & {\bf A}_{3, N}(t) & \cdots\\
		& &\ddots & \ddots & \vdots &  \vdots & \ddots\\
	& 	& &\ddots   & {\bf A}_{N-1, N-1}(t) & {\bf A}_{N-1, N}(t)  & \ddots\\
		& & & &  {\bf A}_{N, N-1}(t) &  {\bf A}_{N, N}(t) & \ddots\\
		& & & & &  \ddots & \ddots
	\end{array}\right]. \end{equation}

The  state matrix ${\bf A}(t)$  
of the infinite-dimensional linear dynamical system
\eqref{carlemannonhomo.eq3}
 is {\bf not}  a bounded operator on $\ell^2({\mathbb Z}_{++}^d)$, the Hilbert space of all square-summable sequences on
${\mathbb Z}_{++}^d$. This prevents us to apply existing theory on Hilbert space directly to analyze the
Carleman linearization.
An alternative approach to solve the infinite-dimensional  linear system \eqref{carlemannonhomo.eq3} is to consider its {\bf finite-section approximation} of order $N$, which is described as follws:
 \begin{eqnarray} \label{Athomogenous.def}
\hskip-0.02in	\left [\begin{array}{c}
		\dot{\bf y}_{1, N}\\
		\dot{\bf y}_{2, N}\\
		\vdots\\
\dot{\bf y}_{N-1, N}\\
		\dot{\bf y}_{N, N}
	\end{array}\right]
\hskip-0.08in  & = &  \hskip-0.08in \left[\begin{array}{cccccc}
		{\bf A}_{1, 1}(t) & {\bf A}_{1, 2}(t)  & \cdots  &  {\bf A}_{1, N-1}(t)  & {\bf A}_{1, N}(t) \\
			  {\bf A}_{2, 1}(t)
		& {\bf A}_{2, 2}(t)  & \cdots & {\bf A}_{2, N-1}(t) & {\bf A}_{2, N}(t) \\
		& \ddots 
& \ddots & \vdots & \vdots   \\
	 	& &\ddots   & {\bf A}_{N-1, N-1}(t) & {\bf A}_{N-1, N}(t)  \\
		& & &  {\bf A}_{N, N-1}(t) &  {\bf A}_{N, N}(t) \\
			\end{array}\right]
\nonumber \\
\hskip-0.08in  &  &  \hskip-0.08in
\times	\left[\begin{array}{c}
		{\bf y}_{1, N}\\
		{\bf y}_{2, N}\\
		\vdots\\
		{\bf y}_{N-1, N}\\
{\bf y}_{N, N}\\
	\end{array}\right]
+
	\left[\begin{array}{c}
		{\bf f}_{\bf 0}(t)\\
		{\bf 0}\\
		\vdots\\
		{\bf 0}\\
{\bf 0}\\
	\end{array}\right],
 \end{eqnarray}
where ${\bf y}_{k, N}:= {\bf y}_{k, N}(t)$ satisfies the initial condition ${\bf y}_{k, N}(0)={\bf z}_k(0), 1\le k\le N$.

With the assumption that
 the origin ${\bf 0}$ serves as an equilibrium for the complex dynamical system \eqref{complexdynamic.def}, i.e.,
 \eqref{zeroequil.eq}
 holds,
the state matrix ${\bf A}(t)$ in \eqref{carleman.matrixA}
is  a {\bf block    upper triangular} matrix.
 Utilizing the ``block upper-triangular'' structure of the state matrix ${\bf A}$, we can show
  that the first block  ${\bf y}_{1, N}, \ N\ge 1$, in the finite-section approximation converges exponentially to the state vector  ${\bf x}\in {\mathbb C}^d$
of the original nonlinear dynamical system within a certain time range.

\begin{theorem}\label{maintheorem0}
Suppose that ${\bf x}(t)$
 is the  solution of  the  dynamical  system \eqref{dynamicsystem} with the vector field ${\bf g}(t,{\bf x})$ satisfying \eqref{zeroequil.eq} and
 Assumption
 \ref{assump-1} for some $D_0>0$ and  $R>1$,
 and ${\bf y}_{1, N}(t), N\ge 1$, is the first block  of the solution of the finite-section approximation
 \eqref{Athomogenous.def}. If  the initial  ${\bf x}(0)={\bf x}_0$ of the dynamical system \eqref{dynamicsystem}
  satisfies \eqref{maintheorem0.eq1},
	then
	\begin{equation}\label{maintheorem0.eq2}
		\|{\bf y}_{1, N}(t)-{\bf x} (t)\|_\infty\le  \frac{R}{\sqrt{2\pi}(e-1)}
		N^{-3/2} e^{D_0 R Nt/(\ln R)^2} \left(\frac{\|{\bf x}_0\|_\infty e}{\ln R}\right)^{ (e-1)N/(2e-1)}
			\end{equation}
	hold for all $0\le t\le T_C^*$ and $N\ge 1$,
	where $T_C^*$ is given in \eqref{Tstar0.def}.
\end{theorem}

The exponential convergence for the  finite-section approximation  \eqref{Athomogenous.def} in Theorem \ref{maintheorem0}
is established for real dynamical systems with the corresponding vector field  is
 a time-independent polynomial \cite[Theorem 4.2]{forets2017explicit} and an analytic function around the origin \cite{Amini2022}.
We may  follow the argument used in   \cite{Amini2022} to establish the exponential convergence conclusion in Theorem \ref{maintheorem0} step by step,  and then we omit the detailed proof here.

The equilibrium point requirement \eqref{zeroequil.eq}
 is not always satisfied for the vector field ${\bf g}$ of the  dynamical system
 \eqref{dynamicsystem}, for instance, the
 complex dynamical system \eqref{simpleexample2.eq1} with $a\ne 0$ and $b\ne 1$.
However,  for the
complex dynamical system \eqref{simpleexample2.eq1} with $a(1-b)$ being proximate to zero, the shifted state vector $\tilde {\bf x}={\bf x}+\ln b$
satisfies the complex dynamical system \eqref{simpleexample2.eq1} with $b=1$ and hence 
the first block  in the finite-section approximation exhibits exponential convergence.
We conjecture that  the finite-section approximation
\eqref{Athomogenous.def} of Carleman linearization could still exhibit exponential convergence in general if the equilibrium point requirement \eqref{zeroequil.eq} is relaxed that
 ${\bf  g}(t,  {\bf 0})$ is proximate to the origin.

\section{Carleman-Fourier Linearization and Convergence of the Finite-Section Approximations}
\label{sec:finite-section}

In this section, we consider the complex dynamical system \eqref{dynamicsystem}
with the periodic vector field ${\bf g}$ satisfying \eqref{assumption0} and Assumption \ref{assump-1} for some $D_0, R>0$.
Under the above requirements on the vector field ${\bf g}$,
  the Fourier expansion
in  \eqref{complexdynmaic.fdef} converges to ${\bf g}(t, {\bf x})$ if
the  state vector ${\bf x}=[x_1, \ldots, x_d]^T\in {\mathbb C}^d$ satisfies
 \begin{equation}\label{assumption0-1}
 \min_{1\le j\le d} \Im x_j>-\ln R,\end{equation}
 and the vector field ${\bf g}$ is analytic on the
 shifted upper half plane $({\mathbb H}-i\ln R)^d$, where ${\mathbb H}$ is the upper half-plane.

In Section \ref{subsec:Carleman-Fourier}, we use Fourier basis functions
$w_{\pmb \alpha}=e^{i{\pmb \alpha} {\bf  x}}, {\pmb \alpha}\in {\mathbb Z}_{++}^d$,
to lift  the complex dynamical system \eqref{dynamicsystem} into an infinite-dimensional dynamical system \eqref{Carleman.eq5},  and we call the lifting scheme as  Carleman-Fourier linearization.
 The state matrix ${\bf B}(t)$ corresponding to the infinite-dimensional dynamical system \eqref{Carleman.eq5} is a block upper-triangular matrix, however it
does not act as a bounded operator on \(\ell^2({\mathbb Z}_{++}^d)\), see \eqref{Carlemanmatrix.property} and \eqref{Carleman.eq4b}.
A conventional approach is to employ
the finite-section approximation \eqref{Carleman.eq7} of the Carleman-Fourier linearization \eqref{Carleman.eq5}.
In Section \ref{carlemanfouriershortconvergence.subsection}, we demonstrate  that the first block ${\bf v}_{1, N}, N\ge 1$,  in the finite-section approximation \eqref{Carleman.eq7} exponentially converges to $e^{i {\bf x}}$, which represents the exponential of the state variable ${\bf x}$ in the complex dynamical system \eqref{dynamicsystem}, over a specified time range  when
the initial state vector ${\bf x}_0$ satisfies \eqref{maintheoremanalytic.thm1.eq1},
 cf. the requirement \eqref{assumption0-1} on the state vector ${\bf x}$ for the convergence of Fourier expansion of the vector field ${\bf g}$.
 The detailed conclusions on the exponential convergence of the finite-section approximation in a time range are stated in Theorem \ref{maintheoremanalytic.thm1} and Corollary \ref{maintheoremanalytic.cor1}.
In Section \ref{carlemanfourierentireconvergence.subsection},  we  establish
 the exponential convergence of ${\bf v}_{1, N}, N\ge 1$, in the finite-section approximation \eqref{Carleman.eq7}
  over the entire time range $[0, \infty)$ when
  the vector field ${\bf g}$ and the initial state vector ${\bf x}_0$ satisfies \eqref{assumption2} and \eqref{maintheoremanalytic.thm2.eq1}
  respectively.  This is elaborated further in Theorem \ref{maintheoremanalytic.thm2} and Corollary \ref{maintheoremanalytic.cor2}.

\subsection{Carleman-Fourier linearization using Fourier basis functions}\label{subsec:Carleman-Fourier}
Let ${\bf w}_k=[w_{\pmb \alpha}]_{{\pmb \alpha}\in \mathbb{Z}_{+, k}^d}$ contain all  exponentials
${w}_{\pmb \alpha}$ with ${\pmb \alpha}\in {\mathbb Z}_{+, k}^d, k\ge 1$.
By \eqref{complexdynamic.def} and \eqref{assumption0}, we have
\begin{equation}\label{complexdynamic.def2}
    \Dot{w}_{\pmb \alpha}=i w_{\pmb \alpha}
    \sum_{{\pmb \gamma} \in \mathbb{Z}^d_+} {\pmb \alpha}^T {\bf g}_{{\pmb\gamma}}(t) {w}_{{\pmb \gamma}}
    =\sum_{{\pmb\beta} \in \mathbb{Z}^d_{++}} \Big( i{\pmb \alpha}^T {\bf g}_{{\pmb\beta}-{\pmb \alpha}}(t)\Big) {w}_{{\pmb \beta}}, \ \ {\pmb \alpha}\in {\mathbb Z}^d_{++}.
\end{equation}
Regrouping all exponentials $w_{\pmb \alpha}$ of order $k\ge 1$ together, we obtain
\begin{subequations} \label{Carleman.eq1}
\begin{equation}\label{Carleman.eq1a}
    \Dot{\bf w}_k(t)=\sum_{l=1}^{\infty}{\bf B}_{k,l}(t) {\bf w}_l(t), \  t\ge 0,
\end{equation}
with the initial
\begin{equation} \label{Carleman.eq1b} {\bf w}_k(0)=\big[e^{i {\pmb \alpha} {\bf  x}_0}\big]_{{\pmb \alpha}\in \mathbb{Z}_{+,k}^d}, \ \ k\ge 1,
\end{equation}
\end{subequations}
 where for every $k, l\ge 1$,
\begin{equation}\label{Carleman.eq2}
    {\bf B}_{k,l}(t)=\Big[ \Sp \sum_{j=1}^d i \hspace{0.05cm} {\pmb \alpha}^T{\bf  g}_{{\pmb\beta
    }-{\pmb \alpha}}(t) \Sp \Big]_{{\pmb \alpha}\in \mathbb{Z}_{+,k}^d, \Sp {\pmb \beta}\in \mathbb{Z}_{+,l}^d}
\end{equation}
is a matrix of size $\binom{k+d-1}{d-1}\times \binom{l+d-1}{d-1}$.
 By  \eqref{assumption0}, one may verify that
${\bf  B}_{k,l}, 1\le l<k$, are zero matrices and ${\bf  B}_{k,k}, k\ge 1$, are diagonal matrices,
i.e.,
\begin{subequations}\label{Carlemanmatrix.property}
\begin{equation} \label{Carleman.eq3}
{\bf  B}_{k,l}(t)={\bf  0} \quad {\rm if } \quad  1\le l<k,
\end{equation}
and
\begin{equation} \label{Carleman.eq4}
{\bf  B}_{k,k}(t)={\rm diag} \big( i {\pmb \alpha}^T {\bf  g}_{ {\bf  0}}(t)\big)_{{\pmb \alpha}\in {\mathbb Z}_{+, k}^d} \quad {\rm if} \quad k\ge 1.
\end{equation}
    \end{subequations}

Define
  $  {\bf w}=[{\bf w}_1^T, {\bf w}_2^T, \ldots, {\bf w}_N^T, \ldots]^{T}$,
and
\begin{equation} \label{Carleman.eq6}
{\bf B}(t)=
\begin{bmatrix}
    {\bf B}_{1,1}(t) & {\bf B}_{1,2}(t) & \dots & {\bf B}_{1, N}(t)  & \cdots\\
      &{\bf B}_{2,2}(t) & \cdots & {\bf B}_{2, N}(t)  & \cdots\\
      & & \ddots & \vdots & \ddots\\
      & & & {\bf B}_{N,N}(t) & \cdots \\
      & & & & \ddots
\end{bmatrix}.
\end{equation}
Then we can reformulate \eqref{Carleman.eq1} in the following matrix form,
\begin{equation}\label{Carleman.eq5}
    \Dot{\bf w}(t)={\bf B}(t){\bf w}(t), \ t\ge 0,
\end{equation}
 with initial condition ${\bf w}(0)=[{\bf w}_k(0)]_{k=1}^\infty$. We call the  infinite-dimensional  dynamical system \eqref{Carleman.eq5} as {\bf Carleman-Fourier linearization} of the  finite-dimensional nonlinear dynamical system \eqref{complexdynamic.def} when
the periodic vector field ${\bf g}$
 satisfies
 \eqref{assumption0} and Assumption \ref{assump-1}.

\subsection{Exponential convergence of finite-section approximation in a time range}
\label{carlemanfouriershortconvergence.subsection}

Given two countable index sets \(X\) and \(Y\), we let \({\mathcal S}(X,Y)\) be the Banach space of matrices \(\mathbf{C}=[c(i,j)]_{i\in X, j\in Y}\) equipped with the finite Schur norm, which is defined by
\begin{align}\label{Schurnorm.def}
   \|C\|_{{\mathcal S}} :=\Sp \max \Big\{\Sp \sup_{i \in X} \sum_{j\in Y}|c(i,j)|, \Sp \sup_{j \in Y} \sum_{i\in X}|c(i,j)|\Sp \Big\}.
\end{align}
In alignment with the reasoning presented in \cite{Amini2022}, we have
\begin{equation} \label{Carleman.eq4b}
    \sup_{t\geq 0}\| {\bf  B}_{k,l}(t)\|_{{\mathcal S}}\leq D_0k R^{k-l}, \ 1\le k\le l.
\end{equation}
 From the above estimate, the state matrix \({\bf  B}(t)\) of the infinite-dimensional dynamical system \eqref{Carleman.eq5} does not act as a bounded operator on \(\ell^2({\mathbb Z}_{++}^d)\).
This makes it challenging to directly apply extant Hilbert space theories when analyzing the infinite-dimensional linear system \eqref{Carleman.eq5}.
Thus  we propose employing the conventional {\bf finite-section approximation} of the  Carleman-Fourier linearization  \eqref{Carleman.eq5}, which is given by
\begin{equation}
\label{Carleman.eq7}
\begin{bmatrix}
    \Dot{{\bf v}}_{1,N}(t) \\
    \Dot{{\bf v}}_{2,N}(t) \\
    \vdots \\
    \Dot{{\bf v}}_{N,N}(t)
\end{bmatrix}
=
\begin{bmatrix}
    {\bf B}_{1,1} (t) & {\bf B}_{1,2}(t) & \dots & {\bf B}_{1, N} (t) \\
      &{\bf B}_{2,2} (t) & \cdots & {\bf B}_{2, N}(t) \\
      & & \ddots & \vdots \\
      & & & {\bf B}_{N,N}(t)
\end{bmatrix}
\begin{bmatrix}
    {\bf v}_{1,N}(t) \\
    {\bf v}_{2,N}(t) \\
    \vdots \\
    {\bf v}_{N,N}(t)
\end{bmatrix}
\end{equation}
using the initial condition \({\bf v}_{k,N}(0)={\bf w}_k(0)\) for \(k=1, \dots, N\). The subsequent  theorem  establishes the exponential convergence of ${\bf  v}_{1, N}(t)$ for $N\ge 1$ to ${\bf w}_1(t)$ within a certain time interval. In the following theorems and corollaries, the state vector of the system \eqref{complexdynamic.def} is denoted by ${\bf x}(t)=[x_1(t), \ldots, x_d(t)]^T \in \mathbb{C}^d$ and ${\bf v}_{1, N}(t)=[v_{1, N}(t), \ldots, v_{d, N}(t)]^T \in \mathbb{C}^d$ represents the leading block in the finite-section approximation \eqref{Carleman.eq7} for \( N \geq 1 \).

\begin{theorem}\label{maintheoremanalytic.thm1}
Let us consider a nonlinear dynamical system described by \eqref{complexdynamic.def} and governed by the periodic analytic vector field ${\bf g}: \R  \times  \C^d \rightarrow  \C^d$ satisfying \eqref{assumption0} and Assumption \ref{assump-1}. If the initial state vector ${\bf x}_0=[x_{0, 1}, \cdots, x_{0, d}]^T$ conforms to \eqref{maintheoremanalytic.thm1.eq1}, then
\begin{equation} \label{maintheoremanalytic.thm1.eq2}
 \max_{1\le j\le d} \hspace{0.05cm} \big| {v}_{j, N}(t) \hspace{0.05cm} e^{-ix_j(t)}-1\big| \leq C_0 N^{-3/2} \hspace{0.05cm} e^{ D_0tN}
 \Big(\frac{e
 \|\exp(i{\bf x}_0)\|_\infty}{R }\Big)^{(e-1)N/(2e-1)},  \ 0\le t\le T_{CF}^*,
\end{equation}
where the constants \( D_0 \) and \( R \) are delineated in \eqref{assumption1}, and \( T^*_{CF} \)  is given in \eqref{maintheoremanalytic.thm1.eq4}, and
 and \( C_0 \) is  defined by
\begin{equation}
C_0 = \frac{1}{\sqrt{2\pi} (e-1)} \exp \left( \frac{3e-1}{2e-1} 
\max_{1\le j\le d} \Im x_{0, j} \Sp + \frac{3e-1}{2e-1} \Sp \ln R- \frac{e}{2e-1} \right).
\end{equation}
\end{theorem}


 For a 
 proof of  Theorem \ref{maintheoremanalytic.thm1}, we refer to Section \ref{maintheoremanalytic.thm1.pfsection}.

\smallskip

Take $T^{**}\le T_{CF}^*$ and select  a sufficiently large order \(N\) in the finite-section approximation \eqref{Carleman.eq7}, such that \(C_0 \Sp N^{-3/2} \Sp e^{ D_0(T^{**}-T_{CF}^*)N}<1\). Then  it  follows directly from \eqref{maintheoremanalytic.thm1.eq2} that $v_{j,N}(t)\ne 0, 1\le j\le d$. Consequently, we can express
\begin{equation} \label{maintheoremanalytic.cor1.eq0}
v_{j, N}(t) = e^{i\xi_{j, N}(t)}\ \ {\rm with} \ \ \xi_{j, N}(0) = x_{0, j}, \ 1\le j\le d.
\end{equation}
Invoking Theorem \ref{maintheoremanalytic.thm1} and noting that
$ |z \ {\rm mod} \ 2\pi|\le  4 \epsilon$ for all $z\in {\mathbb C}$ with $|e^{iz}-1|\le \epsilon\le 1/2$, 
we deduce that \(\xi_{j, N}(t), 1\le j\le d\) offers an accurate approximation of the state vector for the complex dynamical system \eqref{complexdynamic.def}.

\begin{corollary}\label{maintheoremanalytic.cor1}
Given the assumptions of Theorem \ref{maintheoremanalytic.thm1} and the initial condition ${\bf x}_0$ for the system \eqref{complexdynamic.def}, let us consider the time range \(T_{CF}^*\) as described in Theorem \ref{maintheoremanalytic.thm1} and a chosen \(T^{**} \le  T_{CF}^*\). If the order \(N\) of the finite-section approximation \eqref{Carleman.eq7} satisfies
\begin{equation}\label{maintheoremanalytic.cor1.eq1}
C_0 \Sp N^{-3/2} \Sp e^{ - D_0(T_{CF}^{*}-T^{**})N} \Sp \leq \Sp \frac{1}{2},
\end{equation}
where \(C_0\) and \(D_0\) are constants from Theorem \ref{maintheoremanalytic.thm1}, then
\begin{equation} \label{maintheoremanalytic.cor1.eq2}
\sup_{1\le j\le d} \Sp |\xi_{j, N}(t) - x_j(t)| \Sp \leq \Sp 4 \Sp
 C_0 \Sp N^{-3/2} \Sp e^{ D_0tN}
 \Big(\frac{e
 \|\exp(i{\bf x}_0)\|_\infty}{R }\Big)^{(e-1)N/(2e-1)}
\end{equation}
holds for all \(0 \leq t \leq T^{**}\) with \(\xi_{j, N}(t), 1\le j\le d\),  given in \eqref{maintheoremanalytic.cor1.eq0}.
\end{corollary}

\subsection{Exponential convergence of finite-section approximation in the entire time range}
\label{carlemanfourierentireconvergence.subsection}
In this subsection, we consider the exponential convergence of the first block in the finite-section approximation \eqref{Carleman.eq7} over the entire time range \( [0, \infty) \). This examination is under the conditions where the vector field \({\bf g}\)  adheres to  \eqref{assumption0}, \eqref{assumption2} and Assumption \ref{assump-1},
 and the initial state vector \({\bf x}_0=[x_{0, 1}, \ldots, x_{0, d}]^T\) of the nonlinear dynamical system \eqref{complexdynamic.def} satisfies the criterion
\eqref{maintheoremanalytic.thm2.eq1}, or equivalently
     \begin{equation}\label{maintheoremanalytic.thm2.eq1+}
\Big(\sum_{j=1}^d e^{-2\Im x_{0, j}}\Big)^{1/2}< \frac{\mu_0 R} {D_0+\mu_0}.
 \end{equation}

\begin{theorem}\label{maintheoremanalytic.thm2}
Let us consider the nonlinear dynamical system described by \eqref{complexdynamic.def} with a periodic vector field ${\bf g}: \R \times \C^d \rightarrow \C^d$ that satisfies conditions
\eqref{assumption0} and \eqref{assumption2} and Assumption \ref{assump-1}, and its initial state vector ${\bf x}_0$ meets the criteria in \eqref{maintheoremanalytic.thm2.eq1}.
Then,
\begin{equation} \label{maintheoremanalytic.thm2.eq02}
\|{\bf v}_{1, N}(t)-e^{i{\bf x}(t)}\|_\infty
\Sp \leq \Sp
\frac{D_0 R }{\mu_0} \Sp \left(\frac{(D_0+\mu_0)\|\exp(i{\bf x}_0)\|_2}{\mu_0 R}\right)^{N} \ \ {\rm for \ all}\ \ t\ge 0,
\end{equation}
 where  ${\bf x}(t)$ is
the state vector of the system \eqref{complexdynamic.def},  ${\bf v}_{1, N}(t)$ represents the leading block in the finite-section approximation \eqref{Carleman.eq7}
for \( N \geq 1 \), and
the constants $D_0, R$ and $\mu_0$ are defined in \eqref{assumption1}  and \eqref{assumption2}.  
\end{theorem}

 For a  
  proof of  Theorem \ref{maintheoremanalytic.thm2}, we refer to Section \ref{maintheoremanalytic.thm2.pfsection}.

For the state vector ${\bf x}(t)=[x_1(t), \ldots, x_d(t)]^T$, let us consider the function
\begin{equation}\label{ut.def}
u(t)=\|\exp(i{\bf x}(t))\|_2 = \Big(\Sp \sum_{j=1}^d \Sp e^{-2\Im x_j(t)}\Big)^{1/2}.
\end{equation}
From Lemma \ref{maintheoremanalytic.thm2.lem1} in Section \ref{maintheoremanalytic.thm2.pfsection}, it follows that:
\begin{equation} \label{ut.eq1}
u(t) \le u(0) = \|\exp(i{\bf x}_0)\|_2 \quad \textrm{for all} \ \  t\ge 0, \ \ {\rm and} \ \lim_{t\to \infty} u(t)=0.
\end{equation}
Given the relations in \eqref{complexdynamic.def}, \eqref{assumption1}, \eqref{assumption0}, \eqref{maintheoremanalytic.thm2.eq1}, and \eqref{ut.eq1}, we deduce
\begin{equation*}
\big|x_j(t)-x_j(0)\big| = \Bigg|\int_{0}^t \sum_{{\pmb \alpha}\in {\mathbb Z}_+^d} g_{j, {\pmb \alpha}}(s) e^{ i {\pmb \alpha} {\bf x}(s)} ds\Bigg| \le (D_0+\mu_0) t, \ \ 1\le j\le d.
\end{equation*}
 Using the above estimate for the state vector ${\bf x}$ and incorporating \eqref{maintheoremanalytic.thm2.eq02}, we derive
\begin{equation}
\hspace{0cm} \sup_{1\le j\le d} \Sp \left|v_{j, N}(t)e^{-ix_j(t)} -1\right| \Sp \le \Sp \frac{\mu_0 R \|\exp(-i{\bf x}_0)\|_\infty  }{ D_0} \left(\frac{(D_0+\mu_0)\|\exp(i{\bf x}_0)\|_2}{\mu_0 R}\right)^{N} \Sp e^{(D_0+\mu_0)t} \label{maintheoremanalytic.cor2.eq1}
\end{equation}
for all $t\ge 0$. By mirroring the reasoning from the proof of Corollary \ref{maintheoremanalytic.cor1}, it can be shown that the terms $\xi_{j, N}(t), 1\le j\le d$, as mentioned in \eqref{maintheoremanalytic.cor1.eq0}, offer a precise approximation to the state variables $x_j(t), 1\le j\le d$, of the nonlinear dynamical system \eqref{complexdynamic.def} over any time interval $[0, T]$ when \( N \) is sufficiently large.

\begin{corollary} \label{maintheoremanalytic.cor2}
Let the initial condition of the nonlinear dynamical system \eqref{complexdynamic.def} be ${\bf x}_0=[x_{0, 1}, \ldots, x_{0, d}]^T$ as specified in Theorem \ref{maintheoremanalytic.thm2}. For a given $T>0$, if the order $N$ of the finite-section approximation \eqref{Carleman.eq7} meets the criterion
\begin{equation}\label{maintheoremanalytic.cor2.eq1}
\frac{\mu_0 R \|\exp(-i{\bf x}_0)\|_\infty  }{ D_0} \left(\frac{(D_0+\mu_0)\|\exp(i{\bf x}_0)\|_2}{\mu_0 R}\right)^{N} \Sp e^{(D_0+\mu_0)T}
 \Sp \le \Sp \frac{1}{2},
\end{equation}
where $D_0, R$ and $\mu_0$ are constants from \eqref{assumption1} and \eqref{assumption2}, then $v_{j, N}, 1\le j\le d$, can be expressed as in  \eqref{maintheoremanalytic.cor1.eq0} with $\xi_{j, N}(t),  0\le t\le T$,  satisfying
\begin{equation} \label{maintheoremanalytic.cor2.eq2}
\sup_{1\le j\le d}|\xi_{j,  N}(t)-x_j(t)| \Sp \le \Sp
\frac{4\mu_0 R \|\exp(-i{\bf x}_0)\|_\infty  }{ D_0} \left(\frac{(D_0+\mu_0)\|\exp(i{\bf x}_0)\|_2}{\mu_0 R}\right)^{N} \Sp e^{(D_0+\mu_0)t}.
\end{equation}
\end{corollary}

\section{Carleman-Fourier Linearization of Complex Dynamical Systems with Multiple Fundamental Frequencies}
\label{multiple.section}

In this section, we consider complex  dynamical system \eqref{complexdynamic.def}
with the periodic vector field ${\bf g}$ having multiple non-zero fundamental frequencies $\omega_l\in {\mathbb C}$ for $1\le l\le L$, which is represented by the expression
\begin{equation}\label{function.multiple}
{\bf   g}(t, {\bf   x}) =\sum_{{\pmb \alpha}_1,\cdots, {\pmb \alpha}_L\in {\mathbb Z}^{d}} {\bf g}_{{\pmb \alpha}_1, \ldots, {\pmb \alpha}_L}(t)~
 e^{i\sum_{l=1}^L \omega_l {\pmb \alpha}_l {\bf   x}}
\end{equation}
and its Fourier coefficients  ${\bf g}_{{\pmb \alpha}_1, \ldots, {\pmb \alpha}_L}(t)=
[g_{1, {\pmb \alpha}_1,\cdots, {\pmb \alpha}_L}(t), \ldots, g_{d,  {\pmb \alpha}_1,\cdots, {\pmb \alpha}_L}(t)]^T$
of the vector field ${\bf g}$  exhibit the uniform exponential decay property, i.e.,
\begin{equation} \label{function.multiple.eq1}
\sup_{t\ge 0}~
\sum_{j'=1}^d ~\sum_{ |{\pmb \alpha}_1|+\cdots+|{\pmb \alpha}_L|=k} |g_{j', {\pmb \alpha}_1,\cdots, {\pmb \alpha}_L}(t)|
\leq\frac{D_1}{ 2\left( \sum_{l=1}^L|\omega_l|\right) R^{k}}, \ \ k\ge 0,
\end{equation}
 where $D_1>0$ and $R>1$ are positive constants. One can verify that the multiple-frequency
Fourier expansion in \eqref{function.multiple} converges when the state vector ${\bf x}=[x_1, \ldots, x_d]^T\in {\mathbb C}^d$ satisfies
\begin{equation} \label{maintheoremanalytic.thm3.eq1-0}
 \max_{1\le l\le L} \Sp \max_{1\le j\le d} \Sp \big|\Im (\omega_l {x}_{j})\big| \Sp < \Sp \ln R, \end{equation}
 cf. the requirement \eqref{maintheoremanalytic.thm3.eq1} on the initial ${\bf x}_0$ for the convergence of finite-section appproximation of Carleman-Fourier linearization.
Clearly, 
 the periodic vector field ${\bf g}$ in \eqref{complexdynmaic.fdef} and \eqref{assumption1} satisfies the conditions \eqref{function.multiple} and \eqref{function.multiple.eq1} with the single fundamental frequency $\omega_1=1$, the same radius $R$, and a doubled constant $D_1=2D_0$.

In  Section \ref{CarlemanFourierMultifrequency.subsection},
 we introduce  Carleman-Fourier linearization \eqref{Carleman.multiple.def} for the nonlinear dynamical system \eqref{complexdynamic.def}
with the periodic vector field ${\bf g}: \R \times \C^d \rightarrow \C^d$ exhibiting multiple fundamental frequencies and having its Fourier coefficients with exponential decay, as specified by \eqref{function.multiple} and \eqref{function.multiple.eq1}.
 The lifting scheme, highlighted in \eqref{Carleman.multiple.def}, stems from noting that the extended state vector
\begin{equation} \label{function.multiple.eq2}
\tilde {\bf x}=[\omega_1 {\bf x}^T, \ldots, \omega_L {\bf x}^T,
-\omega_1 {\bf x}^T, \ldots, -\omega_L {\bf x}^T]^T\in {\mathbb C}^{2dL}
\end{equation}
obeys the nonlinear dynamical system \eqref{function.multiple.def1} with  Fourier coefficients of its governing periodic vector field $\tilde {\bf g}$
satisfying \eqref{assumption0} and Assumption \ref{assump-1}, as demonstrated by \eqref{function.multiple.eq4} and \eqref{function.multiple.eq5}.

In  Section \ref{multiplefrequence.shortrange.section},  we focus on the finite-section approximation \eqref{Carleman.multiple.eq8} of the  Carleman-Fourier linearization \eqref{Carleman.multiple.def}.
  We  establish that the primary block, $\tilde {\bf v}_{1, N}$ for $N\ge 1$, within the finite-section approximation \eqref{Carleman.multiple.eq8} can offer an exponential approximation to the state vector ${\bf x}$ of the complex dynamical system \eqref{complexdynamic.def} within a certain time span. This is further elucidated in \eqref{maintheoremanalytic.cor3.eq2}, Corollary \ref{maintheoremanalytic.cor3} and Theorem \ref{maintheoremanalytic.thm3}.

We observe
 the imaginary components of the periodic vector field $\tilde {\bf g}$ in \eqref{function.multiple.def1} don't generally fulfill the positivity requirement \eqref{assumption2}, as indicated in \eqref{positive.tildeg}.
 This inspires us to consider complex dynamical systems
 with the periodic vector field ${\bf g}$ having positive multiple frequencies only,
see \eqref{function.positivemultiple} and \eqref{function.positivemultiple.eq21}.
 In Section \ref{multifrequency.entirerange.section}, we
 delve into the Carleman-Fourier linearization \eqref{Carleman.positivemultiple.def} of such nonlinear dynamical system  and
   address the exponential convergence of its primary block in the finite-section approximation throughout the entire time range $[0, \infty)$, as detailed in Theorem \ref{maintheoremanalytic.positive.thm2}.

\subsection{Carleman-Fourier linearization for dynamical systems with multiple fundamental frequencies}
\label{CarlemanFourierMultifrequency.subsection}

For a given ${\pmb \alpha}=[\alpha_1, \ldots, \alpha_d]\in {\mathbb Z}^d$, let us denote ${\pmb \alpha}_+=[\max(\alpha_1, 0), \ldots, \max(\alpha_d, 0)]^T\in {\mathbb Z}_+^d$ and ${\pmb \alpha}_-={\pmb \alpha}_+-{\pmb \alpha}\in {\mathbb Z}_+^d$. We  define
\begin{equation}\label{tildeg.def0}
\tilde {\bf   g}_{\pmb \gamma}(t)=\big[\Sp \tilde  g_{1, {\pmb \gamma}}(t), \Sp \ldots \Sp, \tilde  g_{2dL,  {\pmb \gamma}}(t) \Sp \big]^T,  \ {\pmb \gamma}\in {\mathbb Z}^{2dL},
\end{equation}
 where $\tilde  g_{j,  {\pmb \gamma}}(t)=0$ except that $ \tilde  g_{j, {\pmb \gamma}}(t)= (-1)^m \Sp \omega_l \Sp g_{j', {\pmb \alpha}_1, \ldots, {\pmb \alpha}_L}(t)$ if $j= m Ld+ (l-1) d+j'$ for some  $m\in \{0, 1\}, 1\le l\le L, 1\le j'\le d$, and
 $ {\pmb \gamma}=\big[({\pmb \alpha}_{1})_+^T, \Sp \ldots \Sp, ({\pmb \alpha}_{L})_+^T,
({\pmb \alpha}_{1})_-^T, \Sp  \ldots \Sp, ({\pmb \alpha}_{L})_-^T \big]^T$
for some ${\pmb \alpha}_1, \ldots, {\pmb \alpha}_L\in {\mathbb Z}^d$. Using \eqref{function.multiple.eq1}, we observe that
\begin{equation} \label{function.multiple.eq4}
\tilde {\bf   g}_{\pmb \gamma}(t)={\bf   0} \quad {\rm for \ all} \quad {\pmb \gamma}\in {\mathbb Z}^{2dL}\backslash {\mathbb Z}^{2dL}_+,\  t\ge 0,
\end{equation}
 and
\begin{equation} \label{function.multiple.eq5}
 \sum_{ {\pmb \gamma}\in {\mathbb Z}_{+, k}^{2dL}} \|\tilde {\bf g}_{{\pmb \gamma}}(t)\|_1
  =   \sum_{m=0}^1  \sum_{l=1}^L \sum_{j'=1}^d \Sp \sum_{|{\pmb \alpha}_1|+\ldots+|{\pmb \alpha}_L|=k}
 \big|(-1)^m \Sp \omega_l \Sp g_{j', {\pmb \alpha}_1, \ldots, {\pmb \alpha}_L}(t)\big|
\leq   D_1 R^{-k}, \  t\ge 0,
\end{equation}
for integers $k\ge 0$,
 c.f. \eqref{assumption0} and \eqref{assumption1}. Moreover, the extended state vector $\tilde {\bf x}$  in \eqref{function.multiple.eq2} satisfies
 the following complex dynamical system,
\begin{equation} \label{function.multiple.def1}
{\Dot{\tilde {\bf   x}}}= { \tilde {\bf g}}(t, \tilde {\bf x})=\sum_{{\pmb \gamma}\in {\mathbb Z}^{2dL}_+} \Sp {\tilde {\bf   g}}_{\pmb \gamma}(t) \Sp e^{i {\pmb \gamma} \tilde  {\bf   x}},
\end{equation}
with initial condition  $\tilde {\bf x}(0)=\tilde {\bf x}_0:= [\omega_1 {\bf   x}_0^T,  \ldots, \omega_L {\bf   x}_0^T,
-\omega_1 {\bf   x}_0^T, \ldots, -\omega_L {\bf   x}_0^T]^T$.

Set $\tilde {\bf   w}_k=[\exp(i{\pmb \gamma}{\tilde {\bf x}})]_{{\pmb \gamma}\in {\mathbb Z}_{+, k}^{2dL}}$ with initial
 $\tilde {\bf   w}_k(0)=[\exp(i{\pmb \gamma}{\tilde {\bf x}_0})]_{{\pmb \gamma}\in {\mathbb Z}_{+, k}^{2dL}} , k\ge 1$, and
for indices $1\le k\le l$, define
\begin{equation} \label{function.multiple.eq5+}
\tilde {\bf   B}_{k,l}(t)=\Big[\Sp i\sum_{j=1}^{2dL} \Sp \gamma_j \Sp \tilde g_{j, {\pmb \delta}-{\pmb \gamma}}(t) \Sp \Big]_{{\pmb \gamma}\in {\mathbb Z}_{+,k}^{2dL}, \Sp {\pmb \delta}\in {\mathbb Z}_{+,l}^{2dL}}
\end{equation}
with ${\pmb \gamma}=\big[\Sp \gamma_1, \Sp \ldots \Sp, \gamma_{2dL} \Sp \big]\in {\mathbb Z}^{2dL}_{+,k}$.
 In accordance with the lifting scheme from Section \ref{sec:finite-section}, we define the {\bf Carleman-Fourier linearization} of the  complex  dynamical system \eqref{complexdynamic.def}
with the governing vector field ${\bf g}$  satisfying \eqref{function.multiple} and \eqref{function.multiple.eq1}
as follows:
\begin{equation}\label{Carleman.multiple.def}
{\Dot {\tilde {\bf   w}}(t)}=\tilde {\bf   B} (t) {\tilde {\bf   w}}(t),
\end{equation}
with initial condition ${\tilde {\bf   w}}(0)=[{\tilde {\bf   w}}_k(0)]_{k=1}^\infty$,
where $\tilde {\bf   w}=\big[\Sp  \tilde {\bf   w}_1^T, \tilde {\bf   w}_2^T, \Sp \dots \Sp, \tilde {\bf   w}_N^T, \Sp \dots \Sp \big]^{T}$ and
\begin{equation} \label{Carleman.multiple.eq7}
\tilde {\bf   B}(t)=
\begin{bmatrix}
    \tilde {\bf  B}_{1,1}(t) & \tilde {\bf   B}_{1,2}(t) & \dots  & \tilde {\bf   B}_{1, N}(t)  & \cdots\\
      &\tilde {\bf B}_{2,2}(t) & \cdots & \tilde{\bf   B}_{2, N}(t)  & \cdots\\
      & & \ddots & \vdots & \ddots\\
      & & & \tilde{\bf   B}_{N,N}(t) & \cdots \\
      & & & & \ddots
\end{bmatrix}.
\end{equation}

\subsection{Convergence of finite-section approximation in a time range}
\label{multiplefrequence.shortrange.section}

We define the corresponding  finite-section approximation of the infinite-dimensional dynamical system \eqref{Carleman.multiple.def} by
\begin{equation}
\label{Carleman.multiple.eq8}
\begin{bmatrix}
    {\Dot{\tilde {\bf   v}}}_{1,N}(t) \\ {\Dot{\tilde {\bf   v}}}_{2,N}(t) \\ \vdots \\ {\Dot{\tilde{\bf   v}}}_{N,N}(t)
\end{bmatrix}
    =\begin{bmatrix}
    \tilde {\bf   B}_{1,1}(t) &  \tilde {\bf   B}_{1,2}(t)  & \ldots & \tilde {\bf   B}_{1, N}(t)  \\
      &\tilde{\bf  B}_{2,2}(t) & \ldots & \tilde{\bf  B}_{2, N}(t) \\
      & & \ddots & \vdots \\
      & & & \tilde{\bf  B}_{N,N}(t)
      \end{bmatrix}
      \begin{bmatrix}
    \tilde{\bf  v}_{1,N}(t) \\ \tilde {\bf  v}_{2,N}(t) \\ \vdots \\ \tilde{\bf  v}_{N,N}(t)
\end{bmatrix}
\end{equation}
with the initial $\tilde {\bf  v}_{k,N}(0)=\tilde{\bf  w}_k(0)$ for $k=1, \dots, N$. Following a previously established argument in the proof of Theorem \ref{maintheoremanalytic.thm1} and applying \eqref{function.multiple.eq4} and \eqref{function.multiple.eq5}, the first block $\tilde {\bf v}_{1, N}$ in the approximation \eqref{Carleman.multiple.eq8} converges exponentially to the solution of the original nonlinear system \eqref{complexdynamic.def}.

\begin{theorem}\label{maintheoremanalytic.thm3}
Suppose that the periodic vector field ${\bf g}: \R \times \C^d \rightarrow \C^d$ in the nonlinear dynamical system  \eqref{complexdynamic.def} meets the conditions of \eqref{function.multiple} and \eqref{function.multiple.eq1}.
Let $\tilde {\bf  v}_{1, N}(t)=[\tilde{v}_{1, N}(t), \ldots, \tilde{v}_{2dL, N}(t)]^T, N\ge 1$,  be  the first block  in
the finite-section approximation \eqref{Carleman.multiple.eq8}.
 If the initial state ${\bf x}_0 = [x_{0, 1}, \ldots, x_{0, d}]^T \in {\mathbb C}^d$ satisfies
\begin{equation} \label{maintheoremanalytic.thm3.eq1}
 \max_{1\le l\le L} \Sp \ln \max ( \|\exp(\omega_l{\bf x}_0)\|_\infty,  \|\exp(-\omega_l{\bf x}_0)\|_\infty\big)=
 \max_{1\le l\le L} \Sp \max_{1\le j\le d} \Sp \big|\Im (\omega_l {x}_{0, j})\big| \Sp < \Sp \ln R - 1,
\end{equation}
then
\begin{equation} \label{maintheoremanalytic.thm3.eq2}
\left| \tilde{v}_{j'+d(l-1)+mdL,N}(t) \Sp e^{-i(-1)^m\omega_lx_{j'}(t)}-1\right| \Sp \leq \Sp C_1 \Sp  N^{-3/2} \Sp e^{D_1(t-\tilde T_{CF})N}
\end{equation}
for all $m\in \{0,1\}$, $1\le l\le L$, $1\le j'\le d$, and $0 \le t \le \tilde{T}_{CF}$. Here, the constants $D_1$ and $R$ are defined in \eqref{function.multiple.eq1},
\begin{equation} \label{maintheoremanalytic.thm3.eq4}
    \tilde T_{CF}  = \frac{e-1}{D_1(2e-1)}\Sp \left( \Sp \ln R-1 - \max_{1\le l\le L} \Sp \max_{1\le j\le d} \Sp \big|\Im (\omega_l {x}_{0,j})\big| \Sp \right),
\end{equation}
and
\begin{eqnarray}
\quad C_1  & \hskip-0.08in =& \hskip-0.08in  \frac{1}{\sqrt{2\pi} (e-1)} \exp\left( \Sp \frac{3e-1}{2e-1} \ln R + \frac{e-1}{2e-1} \max_{1\le l\le L} \Sp \max_{1\le j\le d} \Sp \big|\Im (\omega_l x_{0,j})\big|
-\frac{e}{2e-1} \Sp \right)\nonumber\\
\quad & \hskip-0.08in \le & \hskip-0.08in  \frac{R^2}{2\pi e(e-1)}.
\end{eqnarray}
\end{theorem}

\vspace{0.3cm}

As a consequence, we have the following approximation theorem for the complex dynamical system \eqref{dynamicsystem}.

 \begin{corollary} \label{maintheoremanalytic.cor3}
 Consider the complex dynamical system \eqref{dynamicsystem}
 with  the periodic vector field ${\bf g}: \R \times \C^d \rightarrow \C^d$
 satisfies Assumption \ref{assump-1} for some $D_0>0$ and $R>e$.
Let $\tilde {\bf  v}_{1, N}(t)=[\tilde{v}_{1,  N}(t), \ldots, \tilde{v}_{2d,  N}(t)]^T, N\ge 1$, be  the first block  in
the finite-section approximation \eqref{Carleman.multiple.eq8}.
 \begin{itemize}

 \item[{(i)}] If the initial state ${\bf x}_0\in {\mathbb C}^d$ satisfies \eqref{carlemanfourier.generalrequirement},
then
\begin{equation*} 
\max\Big(\left| \tilde{v}_{j',N}(t) \Sp e^{-ix_{j'}( t)}-1\right|, \left| \tilde{v}_{j'+d,N}(t) \Sp e^{ix_{j'}( t)}-1\right|\Big) \Sp \leq \Sp  \frac{R^2}{2\pi e(e-1)} \Sp  N^{-3/2} \Sp e^{2D_0(t-\tilde T_{CF}^{*})N}
\end{equation*}
for all  $1\le j'\le d$, and $0 \le t \le \tilde{T}_{CF}^*$. Here
$\tilde T^{*}_{CF}$ is given in \eqref{carlemanfourier.generaltimerequirement}.

\item[{(ii)}] If the initial state ${\bf x}_0\in {\mathbb R}^d$ is real-valued, then
\begin{equation*} \label{maintheoremanalytic.cor3.eq2}
  \max\Big(\left| \tilde{v}_{j',N}(t) \Sp e^{-ix_{j'}( t)}-1\right|, \left| \tilde{v}_{j'+d,N}(t) \Sp e^{ix_{j'}( t)}-1\right|\Big) \Sp
\le \Sp \frac{R^2 e^{2D_0Nt-\frac{(e-1)(\ln R-1) }{(2e-1)}N}}{2\pi e(e-1) N^{3/2}}  
\end{equation*}
for all  $1\le j'\le d$, and $0 \le t \le \frac{(e-1)(\ln R-1) }{2D_0(2e-1)}
$.
\end{itemize}
 \end{corollary}

A similar conclusion to the one in Corollary \ref{maintheoremanalytic.cor3} (ii) has been established in \cite{moteesun2024} under the additional assumption that the governing vector field ${\bf g}$ is real-valued. In particular,
\begin{equation*} \label{maintheoremanalytic.cor3.eq2}
\max\Big(\left| \tilde{v}_{j',N}(t) \Sp e^{-ix_{j'}( t)}-1\right|, \left| \tilde{v}_{j'+d,N}(t) \Sp e^{ix_{j'}( t)}-1\right|\Big) \le \frac{1}{R (\sqrt{R}+1)}   \Sp \Big(\frac{ (1+\sqrt{2D_0t})^2}{R}\Big)^N
\end{equation*}
for all  $1\le j'\le d$ and $0 \le t \le  \frac{(\sqrt{R}-1)^2 }{2D_0}
$. This indicates that
Carleman-Fourier linearization for real dynamical systems with periodic vector fields  may deliver more accurate linearization
over more extensive  time range than for complex dynamical systems,
as
$$\frac{(\sqrt{R}-1)^2 }{2D_0}\ge \frac{(e-1)(\ln R-1) }{2D_0(2e-1)},$$
and
$$\frac{ (1+\sqrt{2D_0t})^2}{R}\le e^{2D_0t-\frac{(e-1)(\ln R-1) }{(2e-1)}}\ {\rm for \ all} \
 0\le t\le \frac{(e-1)(\ln R-1) }{2D_0(2e-1)}.$$
 The last inequality holds as for all $R>e$ and $0\le t\le \frac{(e-1)(\ln R-1)}{2D_0(2e-1)}$, we have
  \begin{eqnarray*}
 & &
 \ln \frac{ (1+\sqrt{2D_0t})^2}{R}- {2D_0t}+\frac{(e-1)(\ln R-1) }{(2e-1)}\nonumber\\
& \le  & \sup_{0\le u\le \sqrt{v}} 2\ln (1+u)-u^2-\frac{e}{e-1} v-1=\sup_{u\ge 0}
\  2\ln (1+u)-1- \frac{2e-1}{e-1} u^2 \nonumber\\
 & = & 2\ln (1+ u_0)-1-\frac{2e-1}{e-1} u_0^2 \approx -0.7076,
 \end{eqnarray*}
 where $u_0=\frac{-1+\sqrt{(6e-5)/(2e-1)}}{2}\approx 0.2983$.

Let us denote $\tilde {\bf x}_0=[\tilde x_{0, 1}, \ldots, \tilde x_{0, 2Ld}]^T$. If we select $\tilde T^{**}<\tilde T_{CF}$ and the order $N$ of the finite-section approximation \eqref{Carleman.multiple.eq8} meets the condition
\begin{equation}\label{maintheoremanalytic.cor3.eq1}
C_1 N^{-3/2} e^{ D_1(\tilde T^{**}-\tilde T_{CF})N} \le 1/2,
\end{equation}
then we can express
\begin{equation} \label{maintheoremanalytic.cor3.eq0}
\tilde v_{j, N}(t)=e^{i\tilde \xi_{j, N}(t)} \quad {\rm with} \quad \tilde \xi_{j, N}(0)= \tilde x_{0, j},
\end{equation}
and define
\[
\tilde\xi_{j', N}^*(t):=\frac{1}{2L}\Sp \sum_{m=0}^1 \Sp \sum_{l=1}^L  ~ (-1)^m \Sp \omega_l^{-1} \Sp \tilde \xi_{j'+d(l-1)+mdL,N}(t)
\]
for all $1\le j'\le d$. Applying Theorem \ref{maintheoremanalytic.thm3} and following the methodology used in the proof of Corollary \ref{maintheoremanalytic.cor1}, we deduce
\begin{equation} \label{maintheoremanalytic.cor3.eq2}
\sup_{1\le j'\le d} \Sp \Big|\tilde\xi_{j', N}^*(t)-x_{j'}(t)\Big| \Sp \le \Sp \frac{ 2}{\pi e(e-1) \min_{1\leq l\leq L}|\omega_l|} N^{-3/2} \Sp e^{ D_1(\tilde T^{**}-\tilde T_{CF})N}, \ \  0\le t\le \tilde T^{**}.
\end{equation}

\subsection{Exponential convergence of finite-section approximation over the entire range}
\label{multifrequency.entirerange.section}

For the nonlinear dynamical system with vector field $\tilde {\bf g}$  in \eqref{function.multiple.def1}, we observe
\begin{equation}\label{positive.tildeg}
\min_{1\le j\le 2dL} \Sp \Im \tilde g_{j, {\bf 0}}(t) \Sp = \Sp -\max _{0\leq m\leq 1}\max_{1\le l\le L} \Sp \max_{1\le j'\le d} \Sp \Im ((-1)^m \omega_l g_{j', {\bf 0}, \ldots, {\bf 0}}(t)) \Sp \le \Sp 0
\end{equation}
for all $t\geq 0$, indicating that the positive imaginary requirement \eqref{assumption2} is not met for the periodic vector field $\tilde {\bf g}$. This motivates us to inspect the nonlinear dynamical system \eqref{complexdynamic.def} with its periodic vector field  with nonnegative frequencies,
\begin{equation}\label{function.positivemultiple}
{\bf   g}(t,{\bf   x}) =\sum_{{\pmb \alpha}_1,\ldots, {\pmb \alpha}_L\in {\mathbb Z}_+^{d}} [g_{1, {\pmb \alpha}_1,\ldots, {\pmb \alpha}_L}(t), \ldots, g_{d, {\pmb \alpha}_1,\ldots, {\pmb \alpha}_L}(t)]^T e^{i\sum_{l=1}^L \omega_l {\pmb \alpha}_l {\bf   x}},
\end{equation}
which is analytic on the upper half plane  and has its Fourier coefficients showcasing uniform exponential decay, i.e.,
\begin{equation} \label{function.positivemultiple.eq21}
\sup_{t\ge 0} ~ \sum_{j'=1}^d ~\sum_{ |{\pmb \alpha}_1|+\ldots+|{\pmb \alpha}_L|=k} |g_{j', {\pmb \alpha}_1,\ldots, {\pmb \alpha}_L}(t)| \Sp \leq \Sp \frac{D_2}{ \left( \sum_{l=1}^L|\omega_l|\right)}\Sp R^{-k}, \ \ k\ge 0,
\end{equation}
 where $D_2$ and $R$ are positive constants. For $1 \le j \le dL$ and ${\pmb \beta} \in {\mathbb Z}^{dL}_+$, we define $\widehat g_{j, {\pmb \beta}}(t) = \omega_l g_{j', {\pmb \alpha}_1, \ldots, {\pmb \alpha}_L}(t)$ under the conditions $j = (l-1) d + j'$ for certain $1 \le l \le L$ and $1 \le j' \le d$, and ${\pmb \beta} = [{\pmb \alpha}_1^T, \ldots, {\pmb \alpha}_L^T]^T$ for some ${\pmb \alpha}_1, \ldots, {\pmb \alpha}_L \in {\mathbb Z}^d_+$. Consequently, we set $\widehat {\bf g}_{\pmb \beta}(t) = [\widehat g_{1, {\pmb \beta}}(t), \ldots, \widehat g_{dL, {\pmb \beta}}(t)]^T$. Then, we can verify that the new extended state vector $\widehat{\mathbf{x}} = [\omega_1\mathbf{x}^T, \ldots, \omega_L\mathbf{x}^T]^T\in \mathbb{C}^{dL}$ satisfies the nonlinear dynamical system described by
\begin{equation} \label{function.multiple.def221}
    \Dot{\widehat{\mathbf{x}}} = \widehat{\mathbf{g}}(t, \widehat{\mathbf{x}}) = \sum_{{\pmb \beta}\in \mathbb{Z}_+^{dL}} \Sp \widehat{\mathbf{g}}_{\pmb{\beta}}(t) \Sp e^{i\pmb{\beta} \widehat{\mathbf{x}}}
\end{equation}
The Fourier coefficients $\widehat{\mathbf{g}}_{\pmb{\beta}}(t)$, where ${\pmb \beta}\in \mathbb{Z}_+^{dL}$, satisfy the following inequality
\begin{equation} \label{function.positivemultiple.eq5}
    \sup_{t\ge 0} \Sp \sum_{j=1}^{dL} \Sp \sum_{ {\pmb \beta}\in \mathbb{Z}_{+, k}^{dL}} \Sp \big|\widehat{g}_{j,\pmb{\beta}}(t)\big|~ \le ~D_2 \Sp R^{-k}
\end{equation}
for $ k\ge 0$, which is in accordance with  \eqref{assumption1} and \eqref{function.multiple.eq5}. Consequently, we define the {\bf Carleman-Fourier linearization} of the nonlinear dynamical system \eqref{complexdynamic.def} with the periodic vector field ${\mathbf{g}}(t,{\mathbf{x}})$ from equations \eqref{function.positivemultiple} and \eqref{function.positivemultiple.eq21} as follows:
\begin{equation}\label{Carleman.positivemultiple.def}
    \Dot{\widehat{\mathbf{w}}}(t)=\widehat{\mathbf{B}} (t) \widehat{\mathbf{w}}(t),
\end{equation}
where $\widehat {\bf   w}=[\widehat {\bf   w}_1^T, \widehat {\bf   w}_2^T, \ldots, \widehat {\bf   w}_N^T, \dots]^{T}$ with $\widehat{\mathbf{w}}_k(t)=\big[\Sp e^{i\pmb{\beta} \widehat{\mathbf{x}}} \Sp \big]_{{\pmb \beta}\in \mathbb{Z}_{+, k}^{dL}}$ for $k\ge 1$. The matrix $\widehat{\mathbf{B}}(t)$ is given by:
\begin{equation} \label{Carleman.positivemultiple.eq7}
\widehat {\bf   B}(t)=
\begin{bmatrix}
    \widehat {\bf   B}_{1,1}(t) & \widehat {\bf   B}_{1,2}(t) & \dots & \widehat {\bf   B}_{1, N}(t)  & \cdots\\
      &\widehat {\bf B}_{2,2}(t) & \cdots & \widehat{\bf   B}_{2, N}(t)  & \cdots\\
      & & \ddots & \vdots & \ddots\\
      & & & \widehat {\bf   B}_{N,N}(t) & \cdots \\
      & & & & \ddots
\end{bmatrix}
\end{equation}
in which
$$\widehat{\mathbf{B}}_{k,l}(t)=\Big[\Sp i\sum_{j=1}^{dL} \Sp \beta_j \Sp \widehat{g}_{j, {\pmb \beta}'-{\pmb \beta}}(t) \Sp \Big]_{{\pmb \beta}\in \mathbb{Z}_{+,k}^{dL}, \Sp {\pmb \beta}'\in \mathbb{Z}_{+,l}^{dL}}$$
for $1\le k\le l$. Similarly, the {\it finite-section approximation} of the Carleman-Fourier linearization \eqref{Carleman.positivemultiple.def} is given by
\begin{equation}
\label{Carleman.positivemultiple.eq8}
\begin{bmatrix}
    \Dot{\widehat{\mathbf{v}}}_{1,N}(t) \\ \Dot{\widehat{\mathbf{v}}}_{2,N}(t) \\ \vdots \\ \Dot{\widehat{\mathbf{v}}}_{N,N}(t)
\end{bmatrix}
=
\begin{bmatrix}
    \widehat{\mathbf{B}}_{1,1}(t) & \widehat{\mathbf{B}}_{1,2}(t) & \dots & \widehat{\mathbf{B}}_{1, N}(t)  \\
    & \widehat{\mathbf{B}}_{2,2}(t) & \dots & \widehat{\mathbf{B}}_{2, N}(t) \\
    & & \ddots & \vdots \\
    & & & \widehat{\mathbf{B}}_{N,N}(t)
\end{bmatrix}
\begin{bmatrix}
    \widehat{\mathbf{v}}_{1,N}(t) \\ \widehat{\mathbf{v}}_{2,N}(t) \\ \vdots \\ \widehat{\mathbf{v}}_{N,N}(t)
\end{bmatrix}
\end{equation}
with initial conditions $\widehat{\mathbf{v}}_{k,N}(0)=\widehat{\mathbf{w}}_k(0)$ for $k=1, \ldots, N$.

Following the argument in Theorem \ref{maintheoremanalytic.thm2}, we demonstrate that the first block $\widehat{\mathbf{v}}_{1,N}(t)$ in the finite-section approximation \eqref{Carleman.positivemultiple.eq8} provides a good approximation of quantity $\widehat{\mathbf{w}}_1(t)$ associated with the state vector ${\mathbf{x}}(t)$ of the nonlinear dynamical system \eqref{complexdynamic.def} over the entire time interval $[0, \infty)$. This is contingent upon the constant term $\widehat{g}_{\mathbf{0}}(t)$ of the Fourier coefficients of the vector field $\widehat{g}(t,{\mathbf{x}})$ satisfying
\begin{equation}
\label{maintheoremanalytic.positive.thm2.eq1}
\min_{1\le l\le L} \Sp \min_{1\le j\le d} \Sp \Im (\omega_l g_{j, \mathbf{0}}(t))) \Sp \geq \Sp \widehat{\mu}_0
\end{equation}
for all $t\ge 0$ and some $\widehat{\mu}_0>0$, and the initial state ${\mathbf{x}}(0)=[x_{0, 1}, \ldots, x_{0, d}]^T$ of the nonlinear dynamical system \eqref{complexdynamic.def} satisfying
\begin{equation}\label{maintheoremanalytic.positivethm2.eq2}
    \widehat{\tau}:=\Big(\Sp \sum_{l=1}^L \Sp \sum_{j=1}^d \Sp e^{-2\Im (\omega_l x_{0, j})} \Sp \Big)^{1/2} \Sp < \Sp \frac{\widehat{\mu}_0 R} {D_2+\widehat{\mu}_0}.
\end{equation}

\begin{theorem}\label{maintheoremanalytic.positive.thm2}
Suppose that  the dynamical system \eqref{complexdynamic.def}
with the periodic analytic vector field function ${\bf g}:\R \times \C^d \rightarrow \C^d$ satisfies  \eqref{function.positivemultiple},
 \eqref{function.positivemultiple.eq21}, and \eqref{maintheoremanalytic.positive.thm2.eq1}. If its initial condition ${\bf x}_0$ satisfies  \eqref{maintheoremanalytic.positivethm2.eq2}, then
     \begin{equation} \label{maintheoremanalytic.thm2.eq2}
    \sup_{1\le l\le L} ~ \sup_{1\le j'\le d} ~  \big|\Sp \widehat v_{j'+(l-1)d,N}(t)-e^{i\omega_l x_{j'}(t)} \Sp \big|
    \Sp \leq \Sp
       \frac{\widehat \mu_0 R }{D_2}  \left(\frac{(D_2+\widehat \mu_0)\widehat \tau}{\widehat \mu_0 R}\right)^{N}
    \end{equation}
for all $t\ge 0$.
\end{theorem}

We remark that the positive imaginary assumption \eqref{maintheoremanalytic.positive.thm2.eq1}
is satisfied for some  $\widehat \mu_0>0$ if all fundamental frequencies $\omega_1, \omega_2, \ldots, \omega_L$ are positive and
$\min_{1\le j\le d} \Im g_{j, {\bf 0}}(t)\ge \mu_0$ for all $t\ge 0$ and some $\mu_0>0$.

\section{Numerical demonstrations}
\label{demonstration.section}

 In this section, we first  consider
 the complex dynamical system
\eqref{simpleexample2.eq1} and test the performance of the corresponding  Carleman  and  Carleman-Fourier linearization.
We observe that the shifted and dilated state  $\tilde{x}(t)=x(t/|a|)+
\ln b$ satisfies \eqref{simpleexample2.eq1} with parameters $a$ and $b$ replaced by $a/|a|$ and $1$ respectively,
where we define  $\ln(z)=\ln|z|+i{\rm Arg}(z)$ for a nonzero complex number $z\neq 0$, and ${\rm Arg}(z)$ as its angle  in $(-\pi, \pi]$.
Also we notice that the reflected state
$-\Re x+i\Im x$ satisfies \eqref{simpleexample2.eq1} with $a$ and $b$ replaced by $-\Re a+i \Im a$ and $\bar b$ respectively.
Thus in our simulations of this section,   we normalized the complex dynamical system \eqref{simpleexample2.eq1}
so that its  parameters $a$ and $b$
satisfy
\begin{equation} \label{simpleexample2.eq2a}
b=1 \ \ {\rm and}\ \
 a=e^{i\phi} \ \ {\rm for \ some} \  \ \phi\in [-\pi/2, \pi/2].
\end{equation}
With the above normalization, one may verify that  the complex dynamical system
\eqref{simpleexample2.eq1} has the origin as an equilibrium  and
 its solution can be  explicitly  expressed as
\begin{equation} \label{simpleexample2.eq2}
    x(t) = at + x_0 + i \ln\left(1 +(e^{ait} - 1)e^{ix_0}\right)
\end{equation}
in a short time period.
Depending on the parameter $a$ and the initial $x_0$,  the corresponding trajectory of
 the complex dynamical system \eqref{simpleexample2.eq1} may blow up at a finite time, exhibit a limit cycle, converge or diverge, see
 Figure \ref{fig:solution1} and
 Section \ref{Carleman.simulationsection1} for detailed description on the behavior of the dynamical system \eqref{simpleexample2.eq1}.

The governing vector field $a-ae^{ix}$ of  the complex dynamical system
\eqref{simpleexample2.eq1}  satisfies the equilibrium condition  \eqref{zeroequil.eq}, the analytic property
\eqref{comparison.eq00} and the uniform decay Assumption \ref{assump-1}
for its Fourier coefficients with $D_0=\max(1, R)$ and arbitrary $R>0$. Therefore, the Carleman linearization and Carleman-Fourier linearization proposed in Sections \ref{carleman.subsection} and \ref{sec:finite-section} apply for the complex dynamical system \eqref{simpleexample2.eq1}.
Furthermore, we show  that the first block $v_{1, N}$ of finite-section approximation to the Carleman-Fourier linearization is essentially the Taylor expansion of order $N-1$ for the function $w_1(t)=\exp(ix(t))$ of the original state function
$x(t)$, see \eqref{simpleexample2.eq5}.
As a  consequence, for any initial state $x_0$ and in the time range $[0, T^*]$,  we have explicit approximation error $| e^{ix_0}|^N \sup_{0\le t\le T^*}| e^{iat}-1|^N$ for the finite-section approximation \eqref{simpleexample2.eq4} to the Carleman-Fourier linearization of the dynamical system \eqref{simpleexample2.eq1},
see \eqref{simpleexample2.eq5+}.
Our simulations in Section  \ref{Carleman.simulationsection1} demonstrate
 theoretical results in Theorems  \ref{maintheoremanalytic.thm1}  and \ref{maintheoremanalytic.thm2} that
the finite-section approximation \eqref{simpleexample2.eq4} has its approximation error independent on the real part of the initial $x_0$,
and the finite-section approximation \eqref{simpleexample2.eq4} has smaller approximation error when
the imagery part of the initial $x_0$ takes larger values, where the governing field is well-approximated by trigonometric polynomials of low degrees.

In Section  \ref{Carleman.simulationsection1}, we also test the performance of the classical Carleman linearization. As expected, the Carleman linearization is a superior linearization technique for the  nonlinear dynamical system \eqref{simpleexample2.eq1} when the initial $x_0$ is not far away from the origin.
Comparing with the Carleman-Fourier linearization, our numerical simulation shows
 that the finite-section approximation of the Carleman-Fourier linearization
 exhibits  exponential convergence on the entire range if $\Im a\ge 0$ and  $\Im x_0>\ln 2$,
while   the finite-section approximation of  the Carleman linearization
has exponential convergence on the entire range when $\Im a<0$.
The possible reason  is that the dynamical system \eqref{simpleexample2.eq4} associated with the finite-section approximation of
the Carleman-Fourier linearization is stable when $\Im a>0$, while
the dynamical system  \eqref{simpleexample2.eq7} associated with  the finite-section approximation of the Carleman linearization is
stable when $\Im a<0$.

Next in Section \ref{kuramoto.section}, we delve into the Kuramoto model \eqref{Kuramoto.def} and showcase the effectiveness of the Carleman-Fourier linearization presented in Sections \ref{sec:finite-section} and \ref{multiple.section}.
The Kuramoto model has been extensively employed to analyze the dynamical behavior of coupled oscillators, and it
 captures
the essence of how individual components, despite differing
intrinsic frequencies, can achieve collective coherence through
mutual interaction \cite{bronski2021, dietert2016, guo2021, heggli2019, ji2014, kowalski1991nonlinear}.
Define the rescaled phases $\hat \theta_p$ and neutralized intrinsic  frequencies
$\hat\omega_p, 1\le p\le d$, by
$$\hat \theta_p(t)= \theta_p \left(\frac{d}{|K|}t\right)- \frac{\hat \omega d}{|K|}t- \frac{1}{d}\sum_{q=1}^d \theta_q(0)
\ \ {\rm and} \ \ \hat \omega_p= \frac{d(\omega_p- \hat \omega)}{|K|}, 1\le p\le d,$$
where $\hat \omega=\sum_{q=1}^d w_q/d$.
Then one may verify that  the rescaled phases
 $\hat \theta_p$ satisfies \eqref{Kuramoto.def} with intrinsic natural frequencies being neutralized and
 the coupling strength $K$ replaced by $Kd/|K|$.
 With the above normalization, we may assume that intrinsic natural frequencies are neutralized and the coupling strength
 and the initial frequencies are normalized,
 \begin{equation} \label{kuramoto.eq3}
\sum_{p=1}^d\omega_p=0,  \ |K|=d \ \ {\rm and}\ \ \sum_{q=1}^d \theta_q(0)=0,
\end{equation}
in the Kuramoto model.  With the above normalization,  we observe that the phases $\theta_p$ in \eqref{Kuramoto.def} satisfy
\begin{equation}\label{kuramoto.eq2}
\sum_{q=1}^d \theta_q (t)=0.
\end{equation}
By \eqref{kuramoto.eq2}, we can reformulate \eqref{Kuramoto.def} as
\begin{equation}\label{Kutamoto.def2}
 \dot \theta_p \Sp = \Sp \omega_p \Sp + \Sp \frac{K}{2di} \Sp \sum_{q=1}^d ~\Big[ e^{i(\theta_q+\sum_{q'\ne p} \theta_{q'})}-e^{i(\theta_p+\sum_{q'\ne q}\theta_{q'})}\Big], \ 1\le p\le d.
\end{equation}
Therefore the  vector field  is analytic on the shifted upper half plane and the Carleman-Fourier linearization proposed in Section \ref{sec:finite-section}
is applicable to the nonlinear system \eqref{Kutamoto.def2}, see the plots on the second row of Figure \ref{CarlemanFourierKuramotomodel1.fig} for the approximation error for its finite-section approach.

The governing field in the Kuramoto model \eqref{Kuramoto.def} is a vector-valued trigonometric function about $\theta_p, 1\le p\le d$, and hence the
Carleman-Fourier linearization proposed in Section \ref{multiple.section}
is applicable to the nonlinear system \eqref{Kuramoto.def}, see the plots in the bottom row of  Figure
\ref{CarlemanFourierKuramotomodel1.fig} for the  approximation error for its finite-section approach.
From the comparison of the Carleman-Fourier linearization of the Kuramoto model \eqref{Kutamoto.def2} in Figure \ref{CarlemanFourierKuramotomodel1.fig} where $d=3$, we observe that for the same order $N$, the finite-section approximation of the Carleman-Fourier linearization in Section \ref{sec:finite-section}
exhibits better approximation properties than  the finite-section approximation of the  Carleman-Fourier linearization in Section \ref{multiple.section} does,
and moreover, for large approximation order $N$,  the size $\binom {N+3}{3}-1$ of the finite-section approximation in Section \ref{sec:finite-section} is much smaller than
the size $\binom {N+4}{4}-1$ of the finite-section approximation in Section \ref{multiple.section} in our simulations.
Additionally, we observe that the finite-section approximations in Sections \ref{sec:finite-section} and \ref{multiple.section} demonstrate excellent approximation performance near an equilibrium point, and on  the sides of a parallelogon where the vector field has small amplitudes.

With the normalization in \eqref{kuramoto.eq3}, the governing vector field
of the Kuramoto model is analytic and hence the classical Carleman linearization  is applicable
for the nonlinear dynamical system  \eqref{Kuramoto.def}, see   Figure \ref{Carlemankuramotomodel.fig}.
 Similar to numerical demonstration in \cite{moteesun2024} for the Carleman linearization and
  Carleman-Fourier linearization of the Kuramoto model \eqref{Kutamoto.def2} with $d=2$, we  see that
  for $d=3$,
the finite-section approximation to the Carleman linearization exhibits exponential
convergence even when the initial is not far away from the origin, and Carleman-Fourier
linearization delivers much accurate linearizations for the Kuramoto model
over more extensive neighborhoods surrounding the equilibrium point, outperforming
traditional Carleman linearization except the natural frequencies  and the initials
 are close to the origin.

\subsection{Comparing Carleman-Fourier Linearization with  Carleman Linearization}\label{Carleman.simulationsection1}
In this subsection, we discuss the behavior of the dynamical system \eqref{simpleexample2.eq1}, and we
 demonstrate and compare the performance of its Carleman-Fourier linearization and Carleman linearization.

First we consider  the behavior of the dynamical system \eqref{simpleexample2.eq1}.
By \eqref{simpleexample2.eq2}, the solution $x(t)$ of the complex dynamical system \eqref{simpleexample2.eq1} may blow up at a finite time $t=t_0>0$ if
\begin{equation}\label{simpleexample2.eq2b}
    1 +(e^{ait_0} - 1)e^{ix_0} = 0 \quad \text{and} \quad 1 +(e^{ait} - 1)e^{ix_0} \neq 0 \  \ {\rm for \ all} \ \ 0\le t< t_0,
\end{equation}
see the black trajectories shown in Figure \ref{fig:solution1}
where simulation parameters $x_0=i\ln  (1-e^{ai\pi/2})$ for $a=1,  i, -i$ respectively.
 One may verify that the requirement \eqref{simpleexample2.eq2b} for the initial state vector $x_0$ is met for some $t_0 > 0$ when $\Re(e^{ix_0}) = \frac{1}{2}$ and $a= 1$, or when $\Re x_0\in 2\pi {\mathbb Z}+\pi$ and $a=-i$, or when $\Re x_0\in 2\pi {\mathbb Z}$, $\Im x_0<0$ and $a=i$.

\begin{figure}[t]
  \centering
    \includegraphics[width=4.9cm, height=3.8cm]{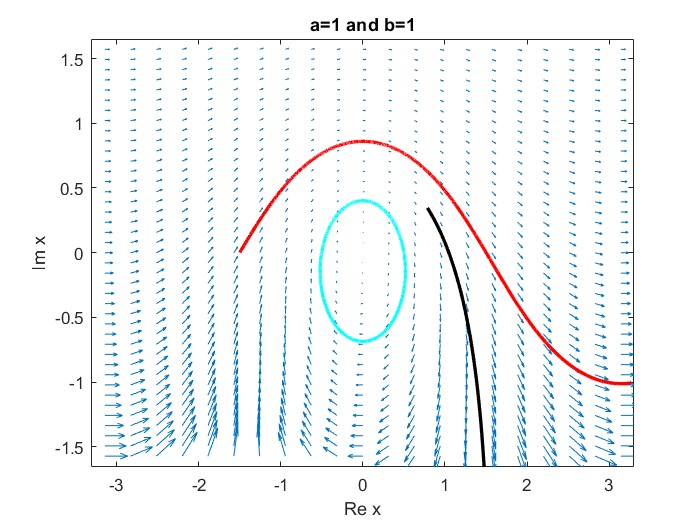}
       \includegraphics[width=4.9cm, height=3.8cm]{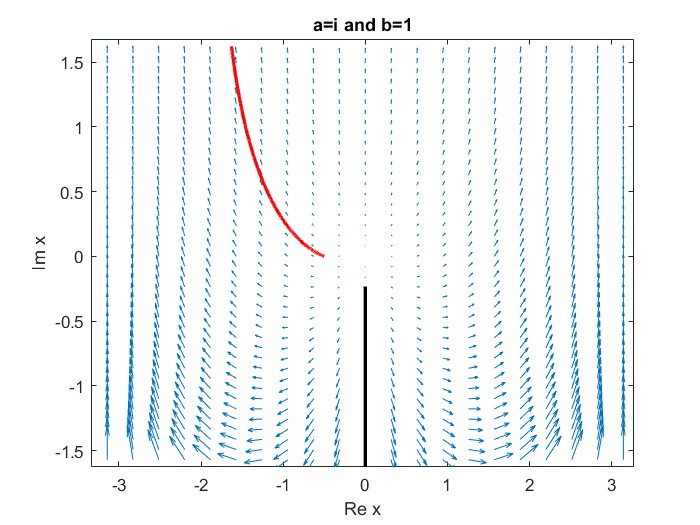}
       \includegraphics[width=4.9cm, height=3.8cm]{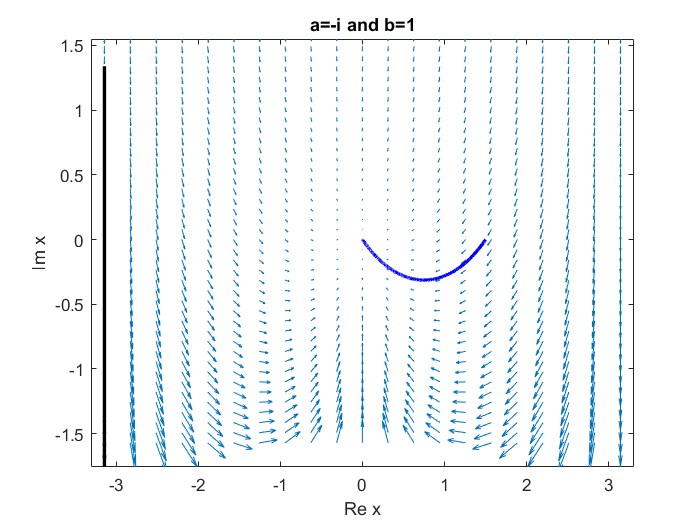}
        \captionsetup{width=1\linewidth}
\caption{Plotted are the vector fields $ a(1-e^{ix})$ of the complex dynamical system \eqref{simpleexample2.eq1} with  $a=1$ (left),  
$a=i$ (middle) and $a=-i$ (right), where $-\pi\le \Re x\le \pi$ and $-\pi/2\le \Im x\le \pi/2$. Trajectories
on the left figure have parameters $a=1$  
   and initial
 $x_0= i\ln  (1-e^{ai\pi/2})\approx 0.7854  + 0.3466i$ (in black), $-1/2$ (in cyan) and $-3/2$ (in  red).
 Presented in the middle are trajectories with $a=i$ and $x_0= i\ln  (1-e^{ai\pi/2}) \approx - 0.2330i$ (in black) and $-1/2$ (in red),
 while on the right are trajectories with $a=-i$ and $x_0= i\ln  (1-e^{ai\pi/2}) \approx -3.1416 + 1.3378i$ (in black) and $3/2$ (in blue).
 Trajectories shown in the figures may
 blow up in a finite time (in black), have limit cycle (in cyan), converge (in  blue) and  diverge (in red).
}
  \label{fig:solution1}
\end{figure}

Now we continue examining the behavior of the dynamical system \eqref{simpleexample2.eq1} when the initial vector $x_0$ does not satisfy condition \eqref{simpleexample2.eq2b} for all $t_0> 0$, i.e., $1 +(e^{ait} - 1)e^{ix_0}\ne 0$ for all $t\ge 0$.
For the case that $\Im a=0$, i.e., $a=1$.
 we observe that $e^{-ix_0}(1+(e^{it}-1)e^{ix_0}), t\ge 0$, is a  circle with center  $e^{-ix_0}-1$ and radius 1.
Therefore  $\ln \left(1-(e^{it}-1)e^{ix_0}\right)$ is a periodic function with a period of $2\pi$ when $|e^{-ix_0}-1|>1$, and
$\ln \left(1-b(e^{it}-1)\right)-it$ is a periodic function with the same period of $2\pi$ when  $|e^{-ix_0}-1|<1$. This implies that when $a=1$, the dynamical system \eqref{simpleexample2.eq1} diverges when $|e^{-ix_0}-1|<1$ and exhibits a limit cycle when $|e^{-ix_0}-1|>1$.
 These behaviors are illustrated by the cyan color limit cycle trajectory in Figure \ref{fig:solution1} with a period of $2\pi$ and the red color trajectory in Figure \ref{fig:solution1}, where $x(t)-t$ forms a periodic function with a period of $2\pi$.

For the case that $\Im a\ne 0$, we observe that
(i) $\lim_{t\to \infty} 1+e^{ix_0}(e^{ait}-1)=1-e^{ix_0}$ when $\Im a>0$; and (ii)
  $\lim_{t\to \infty} e^{-iat} \left(1+e^{ix_0}(e^{ait}-1)\right)=e^{ix_0}$ when $\Im a<0$.
Therefore, the dynamical system \eqref{simpleexample2.eq1} converges when $\Im a<0$, diverges  when $\Im a>0$ and $x_0\not\in 2\pi {\mathbb Z}$, and the solution of the dynamical system \eqref{simpleexample2.eq1} remains at  equilibria $x_0\in 2\pi {\mathbb Z}$.
 This behavior is illustrated by the green color trajectory in the right plot of  Figure \ref{fig:solution1},  
 and the  red color trajectory in the middle plot of Figure \ref{fig:solution1}. 

\medskip

Next, we consider the Carleman-Fourier linearization of the complex dynamical system \eqref{simpleexample2.eq1}.
Set  $w_k=e^{ikx}, k\ge 1$.
Using equations \eqref{simpleexample2.eq1} and \eqref{simpleexample2.eq2a}, we can derive the following equation
\begin{equation} \label{simpleexample2.eq3}
    \Dot{w}_k = ika  w_k - ika  w_{k+1},
\end{equation}
with initial conditions $w_k(0)=\exp(ikx_0)$ for $k\ge 1$. Consequently, the Carleman-Fourier linearization of the complex dynamical system \eqref{simpleexample2.eq1} can be represented by
\begin{equation}
\label{simpleexample2.eq3}
\begin{bmatrix}
    \Dot{{w}}_{1} \\ \Dot{{w}}_{2}  \\ \vdots  \\ \vdots \\ \Dot{{w}}_{N-1}  \\ \Dot{{w}}_{N} \\ \vdots
\end{bmatrix}
    =\begin{bmatrix}
     ai & -ai  &  \ldots&  \ldots & 0 & 0 & \ldots  \\
      &2ai &      \ldots &\ldots & 0 & 0 & \ldots \\
      & & \ddots & \ddots  & \vdots & \vdots & \ldots\\
    &    & & \ddots  & \ddots & \vdots & \ldots \\
      & & &   &(N-1)a i & -(N-1)ai & \ldots\\
      & & &  &  &N ai  & \ddots\\
            & & &  &  &  &\ddots  \\
      \end{bmatrix}
      \begin{bmatrix}
    {w}_{1} \\ {w}_{2} \\ \vdots\\ \vdots \\ w_{N-1} \\{ w}_{N} \\ \vdots
\end{bmatrix}.
\end{equation}
Additionally, the corresponding finite-section approximation is given by
\begin{equation}
\label{simpleexample2.eq4}
\begin{bmatrix}
    \Dot{{v}}_{1,N}(t) \\ \Dot{{v}}_{2,N}(t) \\ \vdots \\ \Dot{{ v}}_{N-1,N}(t)\\ \Dot{{ v}}_{N,N}(t)
\end{bmatrix}
    =\begin{bmatrix}
    ai& -ai& \dots & 0 & 0 \\
      & 2ai & \cdots & 0 & 0\\
      & & \ddots & \vdots & \vdots \\
          & & & (N-1)ai &  -(N-1)ai\\
      & & &  & Nai
      \end{bmatrix}
      \begin{bmatrix}
    {v}_{1,N}(t) \\ {v}_{2,N}(t) \\ \vdots \\ v_{N-1, N}(t) \\{v}_{N,N}(t)
\end{bmatrix}
\end{equation}
with initial conditions $v_{k, N}(0)=\exp(ikx_0)$ for $1\le k\le N$. By induction on $k=N, N-1, \ldots, 1$, it can be verified that:
\begin{equation*} \label{simpleexample2.eq5-}
v_{k, N}(t)= e^{ik(at+x_0)} \Sp \sum_{l=0}^{N-k} ~ \frac{(k+l-1)!}{(k-1)! \Sp l!} \Sp \big(-e^{ix_0}(e^{iat}-1)\big)^l, \ \ 1\le k\le N,
\end{equation*}
 serve as the solution of the finite-section approximation \eqref{simpleexample2.eq4}. It is worth noting that
\begin{equation} \label{simpleexample2.eq5}
v_{1, N}(t)=e^{i(at+x_0)} \sum_{l=0}^{N-1} \big(-e^{ix_0}(e^{iat}-1)\big)^l
\end{equation}
essentially represents the Taylor polynomial of order $N-1$ for the exponential function $w_1(t)=e^{i(at+x_0)} \big(1+e^{ix_0}  (e^{iat}-1)\big)^{-1}=e^{ix(t)}$ of the original state function $x(t)$. Therefore the approximation error is given by:
\begin{equation} \label{simpleexample2.eq5+}
|v_{1,N}(t)e^{-ix(t)}-1|=|e^{ix_0}(e^{iat}-1)|^N,  \ \  N\ge 1.
\end{equation}

\begin{figure}[t]
  \centering
   \includegraphics[width=4.9cm,  height=3.8cm]{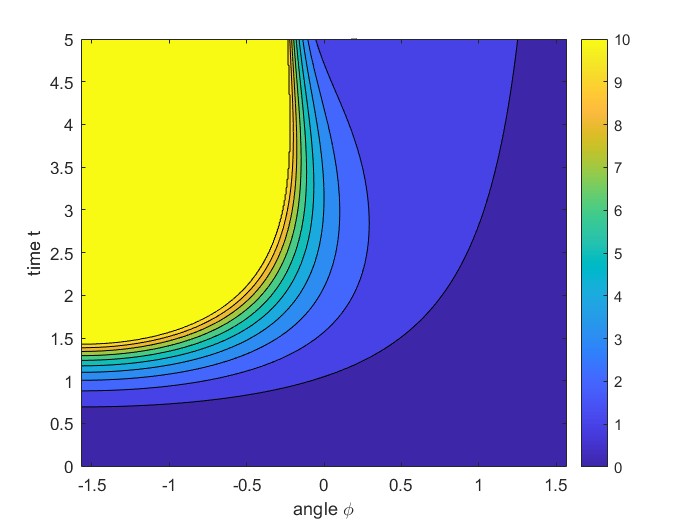}
   \includegraphics[width=4.9cm,  height=3.8cm]{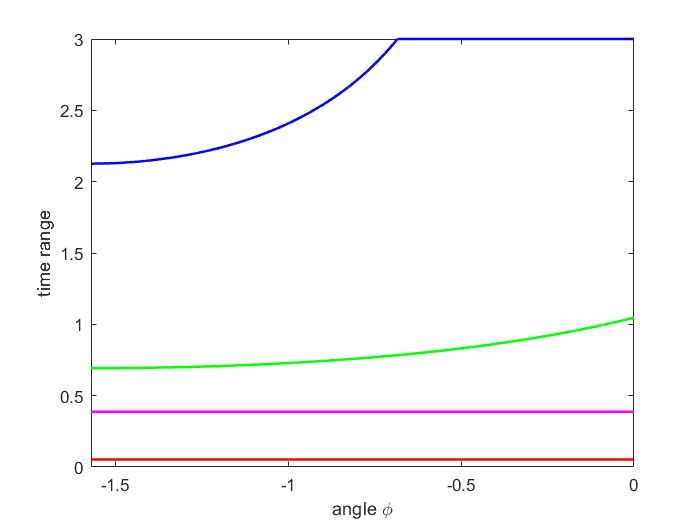}
    \includegraphics[width=4.9cm,  height=3.8cm]{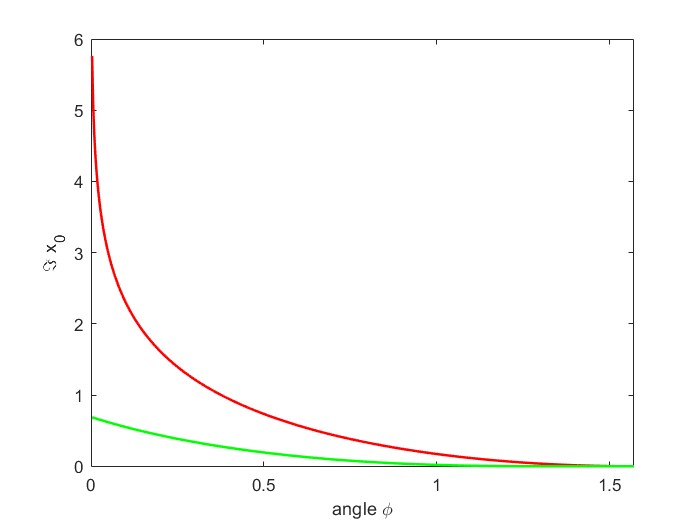}
        \captionsetup{width=1\linewidth}
\caption{Plotted on the left is  the function $\min\{ h(\varphi, t), 10\},  -\pi/2\le \varphi\le \pi/2, 0\le t\le 5$, where $h$ is given in \eqref{hh.def00}.
Presented in the middle  is the actual time range $\min(T^*(\varphi), 3), -\pi/2\le \varphi\le 0$, in \eqref{actualtimerage.orderone}, where $\Im x_0=0$ (in green) and
 $\Im x_0=2$ (in blue), and
the  time range $T_{CF}^*$ in Theorem  \ref{maintheoremanalytic.thm1} 
 when $\Im x_0=0$ (in red) and
 when $\Im x_0=2$ (in magenta).
Shown on the right
is the  requirement on the initial $x_0$ for the exponential convergence on $[0, \infty)$,
which is $\frac{1}{2} \ln \sup_{t\ge 0} h(\varphi, t), 0<\varphi\le \pi/2$, in \eqref{timerange.orderoneeq1} (in green) and the theoretical lower  bound
$-\ln \sin \varphi, 0<\varphi\le \pi/2$ in Theorem  \ref{maintheoremanalytic.thm2} (in red). }
  \label{fig:solu1}
\end{figure}

Using the expression \eqref{simpleexample2.eq2a} for the parameter $a$, we define  
\begin{equation}
\label{hh.def00}
h(\phi, t)=|e^{iat}-1|^2= e^{-2t \sin \phi}-2 e^{-t\sin \phi} \cos(t\cos \phi)+1, \ t\ge 0,
\end{equation}
see the left plot of Figure \ref{fig:solu1} for the function $\min (h(\phi, t), 10), -\pi/2\le \phi\le \pi/2, 0\le t\le 5$.
 By \eqref{simpleexample2.eq5+},
  the first block, $v_{1, N}(t)$ for $N\ge 1$, of the finite-section approximation \eqref{simpleexample2.eq4} exhibits exponential convergence to $w_1(t)=e^{ix(t)}$ in the time interval $[0, T^*]$ if the condition
\begin{equation} \label{simpleexample2.eq6}
h(\phi, t) < e^{2 \Im x_0} \quad \text{for all} \quad 0\le t <  T^*
\end{equation}
is satisfied, c.f.  Theorems \ref{maintheoremanalytic.thm1} and \ref{maintheoremanalytic.thm2}.

In the case that  $\phi\in [-\pi/2, 0)$,  the function $ h(\phi, t), t\ge 0$ is unbounded. Therefore,  for any initial state $x_0$,
the actual time range
\begin{equation} \label{actualtimerage.orderone}
T^*(\phi)=\sup\{T^*\ |\    \eqref{simpleexample2.eq6}\  {\rm   holds}\}
\end{equation}
for the convergence of $v_{1, N}(t), N\ge 1$, is finite. Illustrated
in  the middle plot of Figure \ref{fig:solu1}
 are the maximal time range $\min (T^*(\phi), 5), -\pi/2\le \phi\le 0$, for $\Im x_0=0$ (in green) and $\Im x_0=2$ (in blue).
We remark that  the
 time range $T_{CF}^*$ in \eqref{maintheoremanalytic.thm1.eq4}, per  Theorem \ref{maintheoremanalytic.thm1}, is given by
\begin{eqnarray}\label{maximaltimerage.orderone}
T_{CF}^*& \hskip-0.08in = & \hskip-0.08in \sup_{\ln R+\Im x_0-1>0} \frac{e-1}{(2e-1)\max(1, R)} (\ln R+\Im x_0-1)\nonumber\\
& \hskip-0.08in = & \hskip-0.08in \frac{e-1}{2e-1}\times \left\{\begin{array}
{ll} \exp(\Im x_0-2) & {\rm if} \ \Im x_0\le 2,\\
(\Im x_0-1)  & {\rm if} \ \Im x_0>2,
\end{array}\right.
\end{eqnarray}
see the middle plot of Figure \ref{fig:solu1} where
 $T_{CF}^*\approx 0.0524$ for $\Im x_0=0$ (in red) and
$T_{CF}^*\approx  0.3873$ for $\Im x_0=2$ (in magenta).
We observe that  the time range $T_{CF}^*$ in \eqref{maintheoremanalytic.thm1.eq4} is independent on the selection of $a=\exp(i\phi)$, and it is much smaller than the actual time range $T^*(\phi), -\pi/2\le \phi<0$, for the exponential convergence of the finite-section approximation to the Carleman-Fourier  linearization.

For the scenarios when $\phi=0$, 
 it can be verified that the maximum time range for the convergence of $v_{1, N}(t)$ can be evaluated explicitly,
\begin{equation*}
T^*(\phi)=\left\{\begin{array}{ll}
2 \arcsin \frac{\exp(\Im x_0)}{2} & {\rm if} \quad  \Im x_0\le \ln 2\\
+\infty & {\rm otherwise}.\end{array} \right.
\end{equation*}
Illustrated in the middle plot of
 Figure \ref{fig:solu1} is
 $T^*(0)\approx \pi/3\approx 1.0472$ for $\Im x_0=0$ (in red) and
$T^*(0)=+\infty$ for $\Im x_0=2$ (in blue).

 For the case when $\phi\in (0, \pi/2]$, we have $0\le h(\phi,t)\le 4$, and the constants $D_0$ and $\mu_0$ in \eqref{assumption1} and \eqref{assumption2} are given by $\mu_0=\Im a=\sin \phi$ and $D_0=\max(1, R)$ with arbitrary $R>0$. Using \eqref{simpleexample2.eq5}, we can conclude that the first block $v_{1, N}(t)$ for $N\ge 1$ in the finite-section approximation \eqref{simpleexample2.eq4} provides a satisfactory approximation to $w_1(t)=e^{ix(t)}$ over the entire time range $[0, \infty)$, provided that
 \begin{equation}\label{timerange.orderoneeq1} \Im x_0>\frac{1}{2} \ln h(\phi, t) \ {\rm for \ all} \ t\ge 0,\end{equation}
 which is the region above the green line on the right plot of Figure \ref{fig:solu1}.
   The requirement \eqref{maintheoremanalytic.thm2.eq1} for the initial condition, as per Theorem \ref{maintheoremanalytic.thm2}, is given by
\begin{equation} \label{timerange.orderoneeq2} \Im x_0> - \ln \sup_{R>0}  \frac{\mu_0 R}{D_0+\mu_0}= - \ln \sup_{R>0}  \frac{R\sin \phi}{\max(1, R)+\sin \phi}=-\ln \sin\phi,
\end{equation}
 which is the region above the  red line on the right plot of Figure \ref{fig:solu1}.
 The lower bounds in \eqref{timerange.orderoneeq1} and \eqref{timerange.orderoneeq2} for  the imaginary part
 of the initial state $x_0$ are the same for $\phi=\pi/2$, since  
 $\sup_{t\ge 0} h(\pi/2, t)= \sup_{t\ge 0} |e^{-t}-1|^2=1$.
From the right plot of Figure \ref{fig:solu1} we observe that
 \begin{equation}\label{timerange.inequality}
 -\ln \sin\phi > \frac{1}{2} \ln h(\phi, t)\ \ {\rm for \ all}  \ t\ge 0\ {\rm and}\ 0<\phi<\pi/2,\end{equation}
 see Section \ref{timerange.inequality.pfsection} for the detailed proof.

 Figures \ref{fig:cferrorOrderOne} depicts the approximation performance of the finite-section approach \eqref{simpleexample2.eq4}, where $a=e^{i\phi}$ and
\begin{eqnarray}
\label{fouriererror.def}
E_{CF}(x_0, T^*, N) &\hskip-0.08in  = & \hskip-0.08in  \max_{0\le t\le T^*} \log_{10}  \big| v_{1, N}(t) e^{-ix(t)}-1 \big|\nonumber\\
 & \hskip-0.08in  = & \hskip-0.08in
N \Big(-\Im x_0\log_{10} e+ \frac{1}{2}\log_{10} \big(\max_{0\le t\le T^*}  h(\phi, t)\big)\Big).
\end{eqnarray}
This demonstrates that the first component $v_{1, N}(t)$ in the finite-section approximation \eqref{simpleexample2.eq4} provides a better approximation to the original state $x(t)$ of the dynamical system \eqref{simpleexample2.eq1} in a longer time range when $\phi\in [-\pi/2, 0)$ and in the whole time range $[0, \infty)$ when $\phi\in (0, \pi/2]$, provided that  the imaginary  $\Im x_0$ of initial state $x_0$ takes larger value.
It is also observed that the proposed Carleman-Fourier linearization  has
better performance for the
complex dynamical system \eqref{simpleexample2.eq1} with the parameter $a$ having positive  imaginery part
than for  the one with the parameter $a$ having negative imaginery part. We believe that the possible reason
is that  the  finite-section approximation
\eqref{simpleexample2.eq4}  associated with the Carleman-Fourier linearization of the corresponding  dynamical system
 is stable when $\Im a<0$, while it is unstable when $\Im a>0$.

\begin{figure}[t]
  \centering
    \includegraphics[width=4.9cm, height=3.8cm]{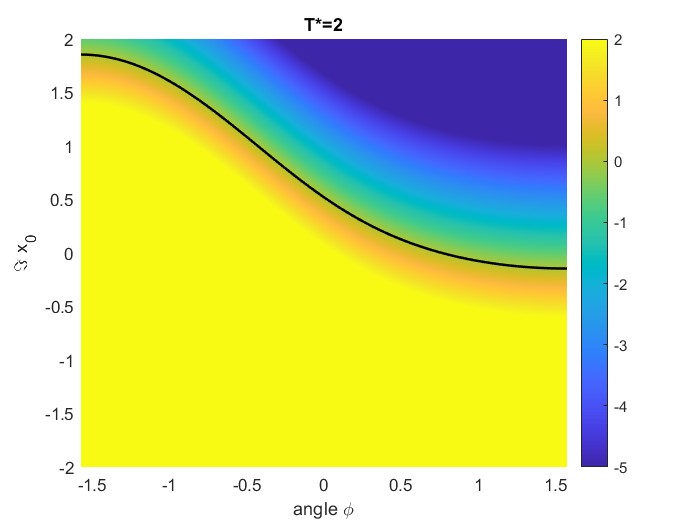}
    \includegraphics[width=4.9cm, height=3.8cm]{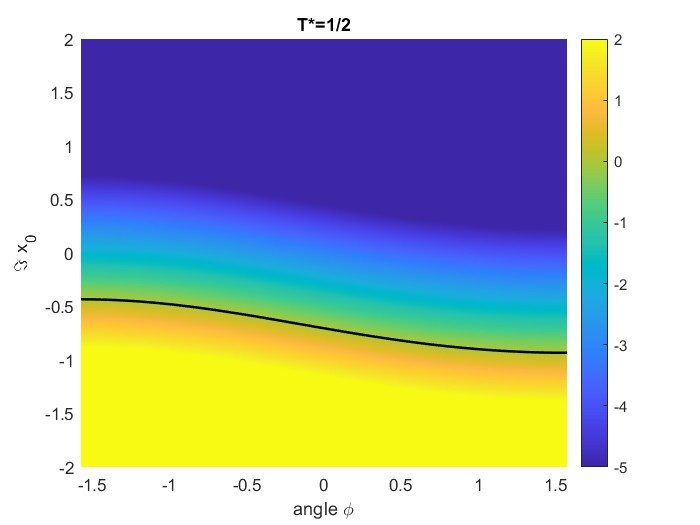}
    \includegraphics[width=4.9cm, height=3.8cm]{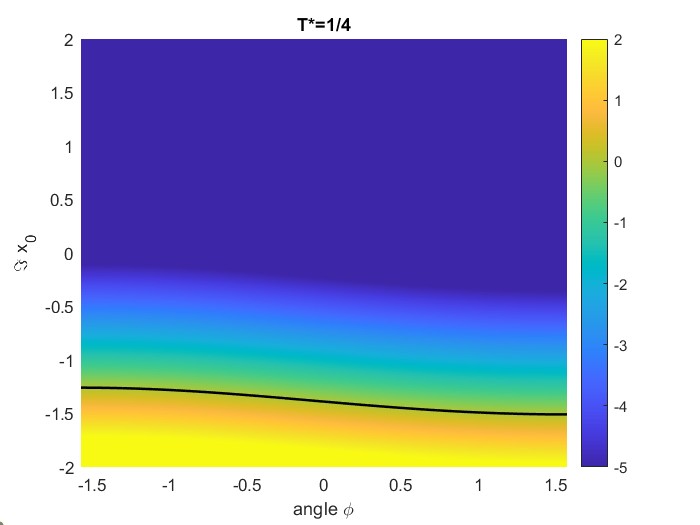}
    \\
\includegraphics[width=4.9cm, height=3.8cm]{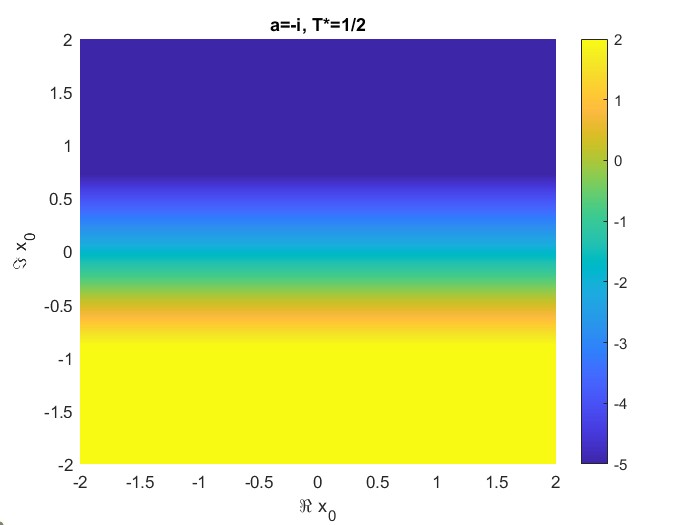}
    \includegraphics[width=4.9cm, height=3.8cm]{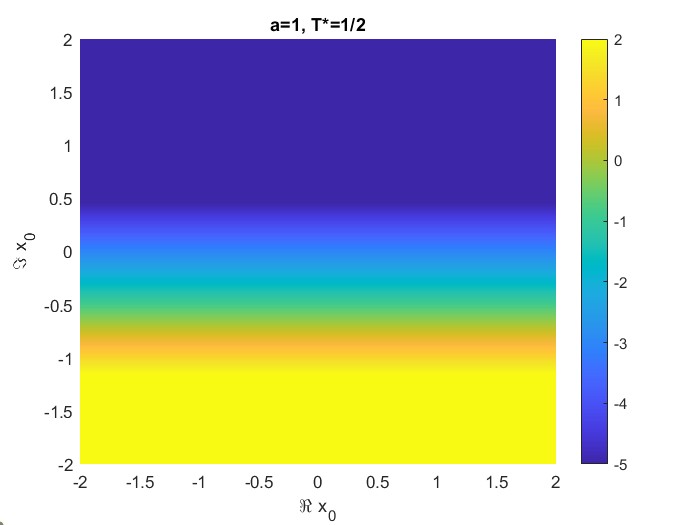}
\includegraphics[width=4.9cm, height=3.8cm]{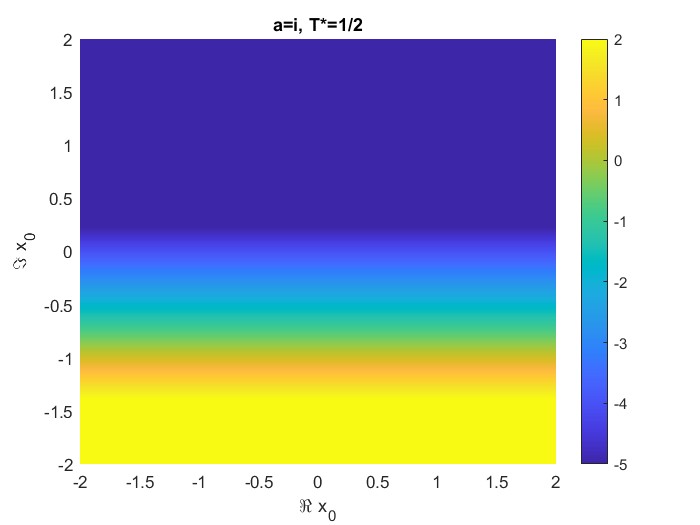}
\\
\includegraphics[width=4.9cm, height=3.8cm]{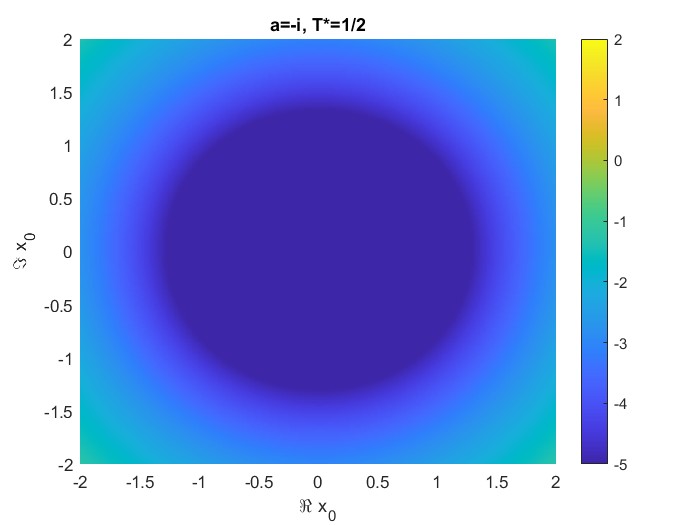}
\includegraphics[width=4.9cm, height=3.8cm]{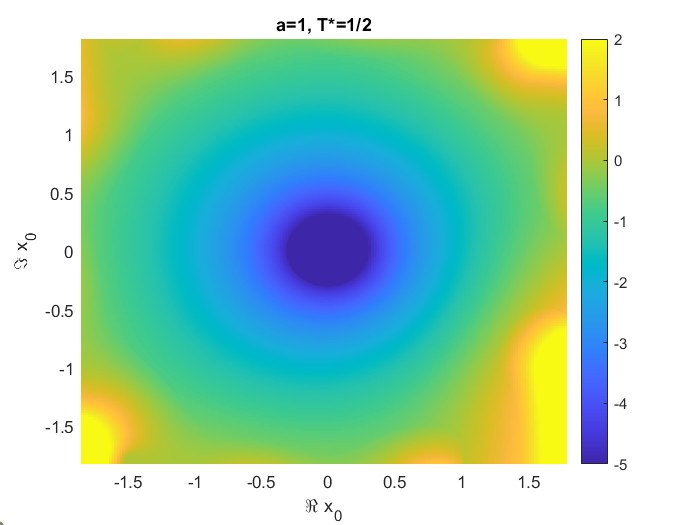}
\includegraphics[width=4.9cm, height=3.8cm]{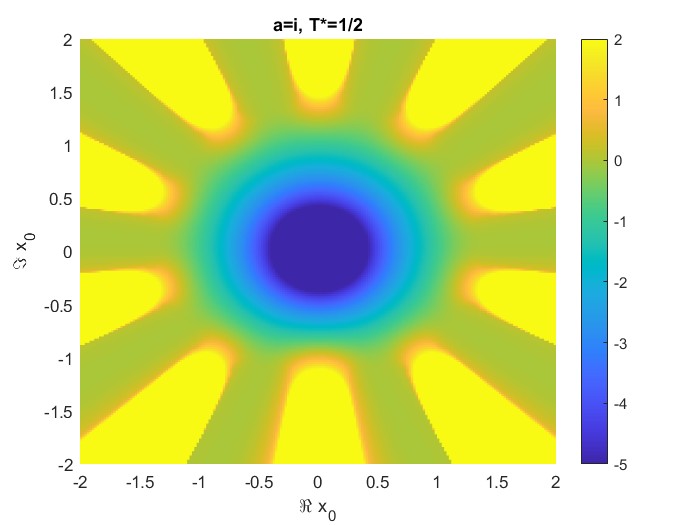}
         \caption{Plotted  on the top are
          the finite-section approximation errors $\max(\min(E_{CF}(x_0, T^*, N), 2), -5)$ of the Carleman-Fourier linearization, defined  in \eqref{fouriererror.def}, where  $-\pi/2\le\phi\le \pi/2$  as the $x$-axis
          and $-2\le \Im x_0\le 2$ as the $y$-axis,
          and level curve $E_{CF}(x_0, T^*, N)=0$ (in black)
          for $N=10$ and  $T^*=2$ (left),
           $1/2$ (middle) and $1/4$ (right) respectively. Shown in the middle are
         $\max(\min(E_{CF}(x_0, T^*, N), 2), -5)$ with  $-2\le \Re x_0 \le 2$  as the $x$-axis
          and $-2\le \Im x_0\le 2$ as the $y$-axis,
          for $N=10, T^*=1/2$ and  $\phi=-\pi/2$ (left),
           $0$ (middle) and $\pi/2$ (right) respectively.
           Plotted at the bottom  are
          the finite-section approximation errors $\max(\min(E_{C}(x_0, T^*, N), 2), -5)$ of Carleman linearization, defined in \eqref{OrderOneCarlemanErrod.def}, where $-2\le  \Re x_0,  \Im x_0\le 2, T^*=1/2, N=10$ and  $a=-i$ (left), $0$ (middle) and $i$ (right).
           }
  \label{fig:cferrorOrderOne}
\end{figure}

We finish this subsection with demonstration to the performance of  the Carleman linearization of the complex dynamical system \eqref{simpleexample2.eq1}.
Write $a(1-e^{ix})= -a \sum_{n=1}^\infty i^n x^n/n!$. Then the finite-section approximation to the classical Carleman linearization  is given by
\begin{equation}
\label{simpleexample2.eq7}
\hspace{-0.1cm} \begin{bmatrix}
    \Dot{{y}}_{1,N}(t) \\ \Dot{{y}}_{2,N}(t) \\ \vdots \\ \Dot{{ y}}_{N-1,N}(t) \\ \Dot{{ y}}_{N,N}(t)
\end{bmatrix}
   = \begin{bmatrix}
    -ai&  - \frac{a i^2 }{2!}& \cdots  & \cdots & - \frac{a i^{N-1} }{(N-1)!} & - \frac{ a i^{N}}{N!} \\
     & -2ai & \cdots  & \cdots  & -  \frac{2a i^{N-2}}{(N-2)!} & - \frac{2ai^{N-1}}{(N-1)!} \\
     &  &\ddots & \ddots & \vdots & \vdots \\
         &  & & \ddots & \vdots & \vdots \\
         &  & & & -(N-1)ai & -\frac{a(N-1)i^2}{2!} \\
         &  & & &  & -Nai
      \end{bmatrix}
     \begin{bmatrix}
    {y}_{1,N}(t) \\ {y}_{2,N}(t) \\ \vdots \\ y_{N-1, N}(t) \\{y}_{N,N}(t)
\end{bmatrix}
\end{equation}
with initial $y_{k, N}(0)=x_0^k$ for  $1\le k\le N$. Following the arguments in \cite{Amini2022}, the first component $ {y}_{1,N}(t)$ for $N\ge 1$, in the finite-section approximation \eqref{simpleexample2.eq7} provides a superb approximation to the solution $x(t)$ of the original dynamical system \eqref{simpleexample2.eq1} in a short time range when
the initial $x_0$ of the original dynamical system \eqref{simpleexample2.eq1} is near the origin. Shown in the bottom plots of Figure \ref{fig:cferrorOrderOne} 
demonstrates these conclusions, where
\begin{equation}\label{OrderOneCarlemanErrod.def}
E_C(x_0, T^*, N)= \max_{0\le t\le T^*} \Sp \log_{10} \big|e^{i(y_{1, N}(t)-x(t))}-1\big|
\end{equation}
and $a=1, i, -i$.
Unlike the Carleman-Fourier linearizaion, we observe that the proposed Carleman linearization  has
better performance for the
complex dynamical system \eqref{simpleexample2.eq1} with the parameter $a$ having negative imaginery part
than for  the one with the parameter $a$ having positive imaginery part. We believe that the reason could be the same,  as we  notice that, contrary to the Carleman-Fourier linearization,
the  finite-section approximation
\eqref{simpleexample2.eq7}  associated with the Carleman linearization of the corresponding  dynamical system
 is stable when $\Im a>0$, while it is unstable when $\Im a<0$.
 Comparing the performance between the Carleman-Fourier linearization and the Carleman linearization, we see that the proposed Carleman-Fourier linearization has  much better performance than the Carleman linearization
 when $\Im x_0$ is large, 
 while the  Carleman linearization, as expected,
  is a superior linearization technique of  a nonlinear dynamical system when the initial is not far from the origin.


\subsection {Carleman-Fourier linearization of the first-order Kuramoto model} \label{kuramoto.section}
In this subsection, we first consider the behavior of the Kuramoto model.
 For $d=2$, we see that $\theta_2=-\theta_1$, and the first phase $\theta_1$ of the Kuramoto model satisfies
$\Dot \theta_1=\omega_1- \tilde K \sin(2\theta_1)$, 
where $\tilde K=K/d=\pm 1$.
Therefore, the first phase $\theta_1(t)$ converges to one of the equilibria  $\{\frac{\tilde K}{2} \arcsin \omega_1, \frac{\pi}{2}- \frac{\tilde K}{2} \arcsin \omega_1\}+ \pi {\mathbb Z}$
   when $|\omega_1|\le 1$, and $\theta_1(t)$ diverges when $|\omega_1|>1$.
For $d=3$, the dynamical system corresponding to the first and second phase is given by
\begin{equation}\label{kuramotodimension3.def}
\left\{
\begin{array}{l}
\dot \theta_1=\omega_1- \tilde K \sin (\theta_1-\theta_2)- \tilde K \sin (2\theta_1+\theta_2)\\
\dot \theta_2=\omega_2-\tilde K\sin (\theta_2-\theta_1)-\tilde K \sin (2\theta_2+\theta_1)
\end{array}
\right.
\end{equation}
and the third phase is defined by $\theta_3=-\theta_1-\theta_2$, where $\tilde K=\pm 1$.
Shown in Figure \ref{CarlemanFourierKuramotomodel1.fig} is the vector field of the first and second phases in  the Kuramoto model with $d=3$.
The governing vector field in the  dynamical system \eqref{kuramotodimension3.def} is periodic with respect to
$(2\pi, 0), (0, 2\pi)$ and $(2\pi/3, 2\pi/3)$, and  the corresponding fundamental domain ${\mathbb R}^2/{\bf G}$ is
the polygon with vertices $(\pm 2\pi/3, 0), (0, \pm 2\pi/3), (\pm 2\pi/3, \mp 2\pi/3)$
where ${\bf G}$ is the  additive group generated by $(2\pi, 0), (0, 2\pi)$ and $(2\pi/3, 2\pi/3)$;
see the top left plot of Figure \ref{CarlemanFourierKuramotomodel1.fig}.
It is observed that phase trajectories of the dynamical system \eqref{kuramotodimension3.def} may  converge
to some equilibrium or diverge, 
 depending on the intrinsic frequencies $\omega_1$ and $\omega_2$, the initial phases $\theta_1(0)$ and $\theta_2(0)$, and also the coupling strength $\tilde K$.
Also the number of (un)stable equilibria may vary. For instance,
 the dynamical system \eqref{kuramotodimension3.def}
with coupling strength $K=-d=-3$ and zero intrinsic
frequencies has  equilibria $\{(0,0), (-\frac{\pi}{3}, \frac{\pi}{3}), (0, -\frac{2\pi}{3}), (\frac{\pi}{3}, -\frac{2\pi}{3}), (\frac{2\pi}{3}, -\frac{2\pi}{3}), (\frac{2\pi}{3}, -\frac{\pi}{3}), (\frac{2\pi}{3}, 0)\}+{\bf G}$, while
 the dynamical system \eqref{kuramotodimension3.def}
with coupling strength $K=-3$ and  intrinsic
frequencies $(\omega_1, \omega_2)=(1/2, 1/2)$  (respectively $(0, 1)$) has  equilibrium points
$\{(-\frac{5\pi}{6}, -\frac{5\pi}{6}), (-\frac{\pi}{6}, -\frac{\pi}{6})\}+{\bf G}$) (resp. $\{(0, \theta_1^*), (0, -\frac{\pi}{2})\}+{\bf G}$ where $\theta_1^*\approx -0.3352$ is a solution of the equation
$1+\sin  \theta+ \sin (2\theta)=0$), see the dark blue position on the top  plots of Figure \ref{CarlemanFourierKuramotomodel1.fig}.

\begin{figure}[t]
  \centering
  \includegraphics[width=4.9cm, height=3.8cm]{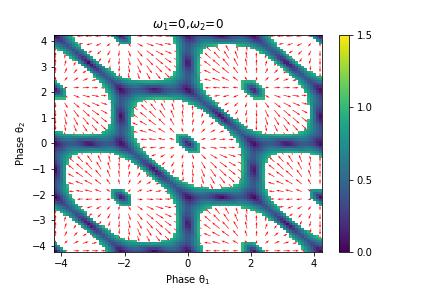}
   \includegraphics[width=4.9cm,  height=3.8cm]{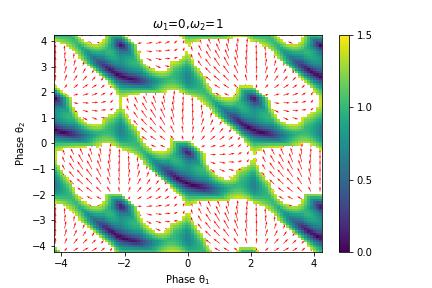}
   \includegraphics[width=4.9cm, height=3.8cm]{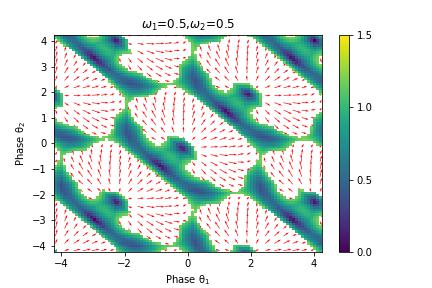} \\
    \includegraphics[width=4.9cm, height=3.8cm]{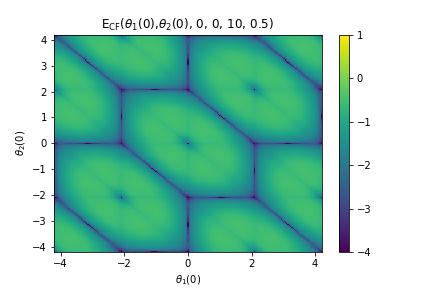}
   \includegraphics[width=4.9cm,  height=3.8cm]{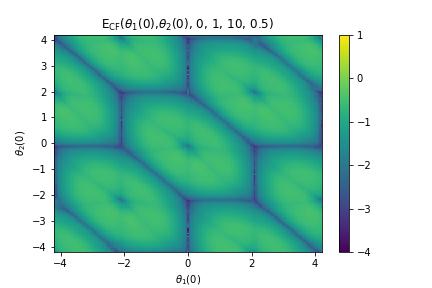}
   \includegraphics[width=4.9cm, height=3.8cm]{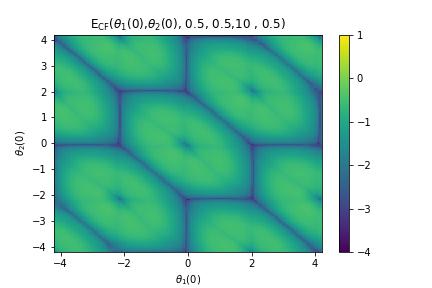} \\
       \includegraphics[width=4.9cm, height=3.8cm]{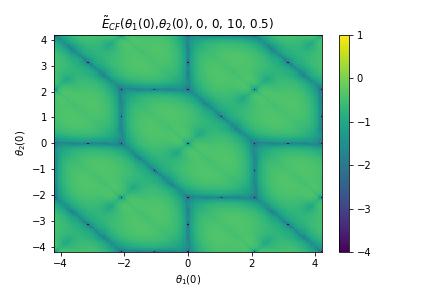}
      \includegraphics[width=4.9cm, height=3.8cm]{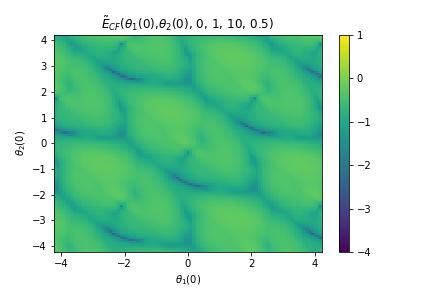}
      \includegraphics[width=4.9cm,  height=3.8cm]{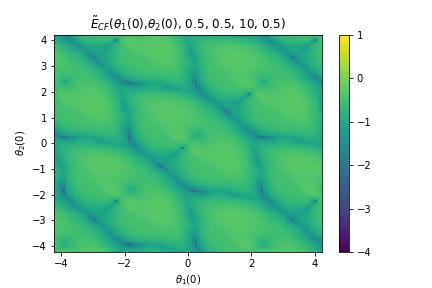}\\
    \captionsetup{width=1\linewidth}
\caption{Plotted on the top row are the vector fields of the dynamical system \eqref{kuramotodimension3.def} for $-4\pi/3\le \theta_1(0), \theta_2(0)\le 4\pi/3$
and the shadowed regions on which  the vector field has  relatively small  magnitude, where  $\tilde K=-1$ and
$(\omega_1, \omega_2)= (0, 0)$ (top left), $(0, 1)$ (top middle),
 $(0.5, 0.5)$ (top right) respectively. Plotted in the middle  and bottom rows are the approximation error
$E_{CF}(\theta_1(0), \theta_2(0), \omega_1, \omega_2, N, T)$ in \eqref{hatECFNt.def} and
$\tilde E_{CF}(\theta_1(0), \theta_2(0), \omega_1, \omega_2, N, T)$ in \eqref{tildeECNt.def},
where $-4\pi/3\le \theta_1(0), \theta_2(0)\le 4\pi/3, N=10, T=0.5, K=-1,  \theta_3(0)=-\theta_1(0)-\theta_2(0), \omega_3=-\omega_1-\omega_2$, and $(\omega_1, \omega_2)=(0, 0)$ (left), $(0, 1)$ (middle) and $(0.5, 0.5)$ (right) respectively.}
  \label{CarlemanFourierKuramotomodel1.fig}
\end{figure}

For Kuramoto model with $d=2$, the performances of its Carleman linearization and Carleman-Fourier linearization have been discussed in \cite{moteesun2024}. It is shown that the finite-section approximation to
the Carleman linearization exhibits exponential convergence
when the initial is not far away from the origin, and
Carleman-Fourier linearization delivers accurate linearizations for systems
featuring periodic vector fields over more extensive neighborhoods surrounding the
equilibrium point, outperforming traditional Carleman linearization except the natural frequency $\omega_1$ and the initial $\theta_1(0)$ are  close to the
origin. Next, we test the performance of Carleman-Fourier linearization and  Carleman linearization  for the Kuramoto model with  $d=3$.

With the normalization described in \eqref{kuramoto.eq2} for the Kuramoto model, the dynamical system \eqref{Kutamoto.def2}
with $d=3$ can be written as follows:
\begin{eqnarray}\label{kuramotodimension3analytic}
\hskip-0.12in  \left(\begin{array}{c}\dot
\theta_1\\
\dot\theta_2\\
\dot\theta_3\end{array}
\right)& \hskip-0.15in = & \hskip-0.15in \left(\begin{array}{c}
\omega_1\\
\omega_2\\
\omega_3\end{array}\right)\hskip-0.05in +
K^*\left\{
\left(\begin{array}{c}
-1\\
0\\
1\end{array}\right) e^{i(2\theta_1+\theta_2)}
+
\left(\begin{array}{c}
-1\\
1\\
0\end{array}\right) e^{i(2\theta_1+\theta_3)}
+
\left(\begin{array}{c}
0\\
-1\\
1\end{array}\right) e^{i(\theta_1+2\theta_2)}\right.\nonumber \\
  & \hskip-0.15in  & 
  \quad \left.
  \left(\begin{array}{c}
0\\
1\\
-1\end{array}\right)
e^{i(\theta_1+2\theta_3)}+
  \left(\begin{array}{c}
1\\
-1\\
0\end{array}\right)
e^{i(2\theta_2+\theta_3)}+
  \left(\begin{array}{c}
1\\
0\\
-1\end{array}\right)
e^{i(\theta_2+2\theta_3)}\right\},
\end{eqnarray}
where $K^*= K/(2di)$.
One can verify that the governing vector field of the above dynamical system is a vector-valued  periodic function with period $(2\pi, 0, 0), (0, 2\pi, 0), (0, 0, 2\pi)$ and $(2\pi/3, 2\pi/3, 2\pi/3)$. Define  the approximation error of the finite-section approximation of order $N\ge 1$ to the Carleman-Fourier linearization for the  dynamical system \eqref{kuramotodimension3analytic} in the logarithmic scale by
\begin{eqnarray}\label{hatECFNt.def}
E_{CF}(\theta_1(0), \theta_2(0), \omega_1, \omega_2, N, T) & \hskip-0.08in  = &
\hskip-0.08in   \sup_{0\le t\le T} \log_{10} \min
\Big\{ 10, \max \Big(10^{-4},  \nonumber \\
 & &  \quad  \max_{1\le p\le 3} \big| v_{p, N}(t)e^{-i\theta_p(t)}-1 \big|\Big)\Big\},
\end{eqnarray}
where $\omega_3=-\omega_1-\omega_2, \theta_3(0)=-\theta_1(0)-\theta_2(0)$
and $v_{p, N}, 1\le p\le 3$, forms the first block  ${\bf v}_{1, N}$ of the
 finite-section approximation \eqref{Carleman.eq7}.
Shown in the middle plots of Figure \ref{CarlemanFourierKuramotomodel1.fig} are the performance of Carleman-Fourier linearization for the  dynamical system \eqref{kuramotodimension3analytic} with the periodic governing field having positive frequencies,
 which demonstrates the theoretical result in Theorem  \ref{maintheoremanalytic.thm1} about exponential convergence of finite-section approximation
\eqref{Carleman.eq7} to the  Carleman-Fourier linearization of the  dynamical system \eqref{kuramotodimension3analytic} in a time range.

With the normalization described in \eqref{kuramoto.eq2} for the Kuramoto model,
the phases $\theta_1$ and $\theta_2$ of the first and second oscillators satisfies  \eqref{kuramotodimension3.def}. Define
the extended variables by
$  [\tilde\theta_1, \tilde \theta_2,  \tilde\theta_3, \tilde\theta_4]=[\theta_1, \theta_2, -\theta_1, -\theta_2]$.
Then the dynamical system associated with the above extended variables is given by
\begin{eqnarray}\label{kuramotodimension3extended}
\left(\begin{array}{c}
\dot{\tilde{\theta}}_1\\
\dot{\tilde{\theta}}_2\\
\dot{\tilde{\theta}}_3\\
\dot{\tilde{\theta}}_4\end{array}
\right)& \hskip-0.15in = & \hskip-0.15in \left(\begin{array}{c}
\omega_1\\
\omega_2\\
-\omega_1\\
-\omega_2\end{array}\right)\hskip-0.05in +
\tilde {K}\left\{
\left(\begin{array}{c}
-1\\
1\\
1\\
-1\end{array}\right) e^{i(\tilde\theta_1+\tilde\theta_4)}
+
\left(\begin{array}{c}
1\\
-1\\
-1\\
1\end{array}\right) e^{i(\tilde\theta_2+\tilde\theta_3)}
+
\left(\begin{array}{c}
-1\\
0\\
1\\
0\end{array}\right) e^{i(2\tilde\theta_1+\tilde\theta_2)}\right.\nonumber \\
  & \hskip-0.15in  & \hskip-0.15in
  \left.
  +\left(\begin{array}{c}
0\\
-1\\
0\\
1\end{array}\right)
e^{i(\tilde\theta_1+2\tilde\theta_2)}+
  \left(\begin{array}{c}
1\\
0\\
-1\\
0\end{array}\right)
e^{i(2\tilde\theta_3+\tilde\theta_4)}+
  \left(\begin{array}{c}
0\\
1\\
0\\
-1\end{array}\right)
e^{i(\tilde\theta_3+2\tilde\theta_4)}\right\},
\end{eqnarray}
where $\tilde K=\frac{K}{2di}$.
Define the approximated error of finite-section approach to its Carleman-Fourier linearization in the logarithmic scale  by
\begin{eqnarray}\label{tildeECNt.def}
\tilde E_{CF}(\theta_1(0), \theta_2(0), \omega_1, \omega_2, N, t)
 & = & \sup_{0\le t'\le t} \log_{10} \min\Biggl\{\max\Bigl(\max_{1\le p\le 2} \big|\tilde v_{p,N}(t')e^{-i\theta_p(t')}-1 \big|,\nonumber \\
 & &  \max_{1\le p\le 2} \big|\tilde v_{p+2,N}(t')e^{i\theta_p(t')}-1 \big|, 10^{-4}\Bigl), 10\Biggl \},
\end{eqnarray}
where $\tilde v_{q, N}, 1\le q\le 4$, form the first block of the finite-section approach in \eqref{Carleman.multiple.eq8}.
Shown in the bottom plots of  Figure \ref{CarlemanFourierKuramotomodel1.fig}
are the approximation error $\tilde E_{CF}(\theta_1(0), \theta_2(0), \omega_1, \omega_2, N, t), -4\pi/3\le \theta_1(0)\le 4\pi/3$, which demonstrates the
exponential convergence conclusion in Theorem \ref{maintheoremanalytic.thm3}. We observe that
the approximation errors
$ E_{CF}(\theta_1(0), \theta_2(0), \omega_1, \omega_2, N, t)$ and $\tilde E_{CF}(\theta_1(0), \theta_2(0), \omega_1, \omega_2, N, t)$
of finite-section approach of our two Carleman-Fourier linearizations of the Kuramoto model with $d=3$ are
periodic about the initial phases $\theta_1(0)$ and $\theta_2(0)$ with period $(2\pi, 0), (0, 2\pi)$ and $(2\pi/3, 2\pi/3)$. We also notice that
the finite-section approach has small approximation error around the equilibria and the sides of the parallelogon with vertices
$(\pm 2\pi/3, 0), (0, \pm 2\pi/3)$ and $(\pm 2\pi/3, \mp 2\pi/3)$, which coincides with
the position of the initial phases, where the vector field has small amplitudes, see  Figure \ref{CarlemanFourierKuramotomodel1.fig}.

Using the Taylor expansion for the sine function, we can rewrite the dynamical system \eqref{kuramotodimension3.def} for the state vector $(\theta_1, \theta_2)$ as follows:
\begin{equation} \label{kuramotodimension3.Taylordef}
\left(\begin{array}{c}\dot
\theta_1\\
\dot\theta_2\end{array}
\right)=\left(\begin{array}{c}
\omega_1\\
\omega_2\end{array}\right)+\tilde K \sum_{k=0}^\infty \sum_{l=0}^{2k+1}
\frac{(-1)^k }{ (2k+1-l)! l! }
\left(\begin{array}{c}(-1)^l- 2^l\\
(-1)^{2k+1-l}-2^{2k+1-l}\end{array}\right) \theta_1^l \theta_2^{2k+1-l},
\end{equation}
where $\tilde K=\pm 1$.
Shown in  Figure \ref{Carlemankuramotomodel.fig}
are the approximation errors in the logarithmic scale,
\begin{eqnarray}\label{hatECNt.def}
E_C(\theta_1(0), \theta_2(0), \omega_1, \omega_2, N, T) & \hskip-0.08in  = & \hskip-0.08in   \sup_{0\le t\le T}
 \log_{10} \min\Big\{10,\  \max\Big (10^{-4},\  \nonumber \\
& & \big|e^{i(y_{1, N}(t)-\theta_1(t))}-1 \big|, \big|e^{i(y_{2, N}(t)-\theta_2(t))}-1 \big|\Big)\Big\},
\end{eqnarray}
where $(y_{1, N}, y_{2, N})$ is the first block of the finite-section approximation of order $N$ to the Carleman linearization of the dynamical system \eqref{kuramotodimension3.Taylordef}.
 This demonstrates the consistence with the theoretical conclusion in Theorem  \ref{maintheorem0}
   that finite-section approximation of the traditional Carleman linearization offers an applauding estimate to the original dynamical system when the initial phases are very close to the origin.
   Comparing the performances shown in Figures \ref{CarlemanFourierKuramotomodel1.fig} and \ref{Carlemankuramotomodel.fig}, the Carleman-Fourier linearization has much better performance than the classical Carleman linearization does when the initial phases of the Karumoto model are a bit far away from the origin.

\begin{figure}[t]
  \centering
    \includegraphics[width=4.9cm, height=3.8cm]{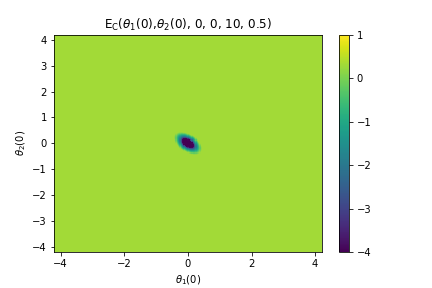}
    \includegraphics[width=4.9cm, height=3.8cm]{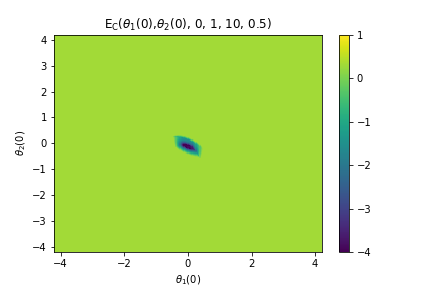}
    \includegraphics[width=4.9cm, height=3.8cm]{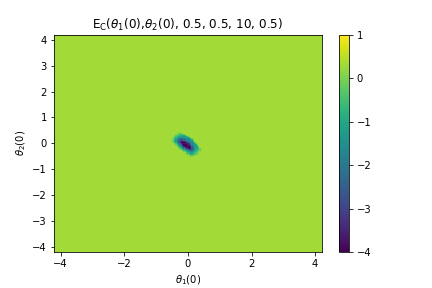}
\caption{Plotted are the approximation error
$E_C(\theta_1(0), \theta_2(0), \omega_1, \omega_2, N, T), -4\pi/3\le  \theta_1(0), \theta_2(0)\le 4\pi/3$, in
\eqref{hatECNt.def} of the finite-section method to the Carleman linearization of the dynamical system \eqref{kuramotodimension3.Taylordef}, where $\tilde K=-1, N=10, T=0.5$ and $(\omega_1, \omega_2)=(0, 0)$(left),
$(0, 1)$(middle) and $(0.5, 0.5)$(right). }
  \label{Carlemankuramotomodel.fig}
\end{figure}

\section{Proofs}
\label{proof.section}

In this section, we collect the proofs of Theorems \ref{maintheoremanalytic.thm1} and \ref{maintheoremanalytic.thm2}, and also the estimates in
\eqref{comparison.eq00} and \eqref{timerange.inequality}.

\subsection{Proof of Theorem \ref{maintheoremanalytic.thm1}}
\label{maintheoremanalytic.thm1.pfsection}

Following the argument used in \cite{Amini2022},  we have the following estimate about
  $\|{\bf   w}_1(t)\|_{\infty}$ on a short time range. 

\begin{lemma} \label{maintheoremanalytic.lem2}
 Let  ${\bf   x}$ be the solution of the complex dynamical system \eqref{dynamicsystem}
 with the initial ${\bf x}_0$ satisfying \eqref{maintheoremanalytic.thm1.eq1} and
   the vector field ${\bf g}$  satisfying  \eqref{assumption0} and Assumption \ref{assump-1},
   and set ${\bf w}_1(t)=e^{i{\bf x}(t)}, t\ge 0$.
Then
   \begin{equation} \label{maintheoremanalytic.lem2.eq1}
   \|{\bf   w}_1(t)\|_\infty\le  \|{\bf   w}_1(0)\|_\infty^{(e-1)/(2e-1)} (R/e)^{e/(2e-1)}<R/e\ \  {\rm for \ all} \ 0\le t\le T^*,
   \end{equation}
   where  $T^*$ is given in \eqref{maintheoremanalytic.thm1.eq4}.
\end{lemma}

\begin{proof}
   Set
   \begin{equation}\label{M0.def}
   M_0=\|{\bf   w}_1(0)\|_\infty^{(e-1)/(2e-1)} (R/e)^{e/(2e-1)}. \end{equation}
   If $\|{\bf   w}_1(t)\|_{\infty }\leq M_0$ for all $t\geq 0$, the proof is completed.
Otherwise, by the continuity of the states $x_j(t), 1\le j\le d$, there exists some $T>0$ such that
\begin{equation} \label{maintheoremanalytic.lem2.pfeq0}
\|{\bf   w}_1(t)\|_{\infty}\leq M_0\ \ {\rm for \ all} \ \ 0\le t\le T, \ \ {\rm  and} \ \
   \|{\bf   w}_1(T)\|_{\infty}= M_0.\end{equation}
Then it suffices to prove that
\begin{equation} \label{maintheoremanalytic.lem2.pfeq0}
T\ge T^*.
\end{equation}

By \eqref{Carleman.eq1}, \eqref{Carleman.eq4b},
we have
\begin{eqnarray*}
\|{\bf   w}_1(t)\|_\infty & \hskip-0.08in  \le & \hskip-0.08in
\|{\bf   w}_1(0)\|_\infty
+\int_{0}^t\sum_{l=1}^\infty \|{\bf   B}_{1, l}(s)\|_{{\mathcal S}} \|{\bf   w}_l(s)\|_\infty ds \\
     &\le  &  \|{\bf   w}_1(0)\|_\infty+\int_{0}^t\sum_{l=1}^{\infty} D_0 R^{1-l} \|{\bf   w}_1(s)\|_\infty^l ds\\
     &\hskip-0.08in  \le & \hskip-0.08in   \|{\bf   w}_1(0)\|_\infty+\frac{D_0 }{1-M_0/R}\int_{0}^t \|{\bf   w}_1(s)\|_\infty ds,\ \ 0\le t\le T.
\end{eqnarray*}
Therefore, by  the integral form of Gronwall's inequality, we have
\begin{align*}
   \|{\bf   w}_1(t)\|_\infty \le \|{\bf   w}_1(0)\|_\infty e^{\frac{D_0 }{1-M_0/R}t}, \ 0\le t\le T.
\end{align*}
This together with  the assumption that $\|{\bf   w}_1(T)\|_{\infty}= M_0$
 proves
 \begin{equation*}
 T\ge \frac{1-M_0/R}{D_0} \ln \frac{M_0}{\|{\bf   w}_1(0)\|_\infty}\ge T^*,
 \end{equation*}
 where the last inequality holds as $M_0<R/e$ by \eqref{maintheoremanalytic.thm1.eq1}.
This proves \eqref{maintheoremanalytic.lem2.pfeq0} and hence completes the proof.
\end{proof}

To prove Theorem \ref{maintheoremanalytic.thm1}, we need two  technical lemmas which follow from \cite[Lemmas 5.3 and 5.4]{Amini2022} in  about solutions of ordinary differential systems.

	\begin{lemma}\label{maintheoremanalytic.lem3} Let ${\bf B}_{k, k}(t), k \ge 1$, be as in \eqref{Carleman.eq4}.
Consider   the  ordinary differential system
	\begin{equation}\label{maintheoremanalytic.lem3.eq1}
	\Dot{\bf   u}_k(t)= {\bf   B}_{k, k}(t) {\bf   u}_k(t)+ {\bf   v}_k(t), \  t\ge 0
	\end{equation}
with   zero  initial  ${\bf u}_k(0)={\bf   0}$, where ${\bf v}_k$ is a vector-valued  continuous function  about $t\ge 0$.
Then
	\begin{equation}\label{maintheoremanalytic.lem3.eq2}
	{\bf   u}_k(t)=\int_{0}^t {\bf    K}_k(t, s)  {\bf   v}_k(s) ds
	\end{equation}
	where ${\bf    K}_k(t, s), t\ge s\ge 0$, is a diagonal matrix with diagonal entries
	\begin{equation}  \label{maintheorem.lem3.eq3}
\exp \Big(i  \int_{s}^t {\pmb \alpha}^T {\bf g}_{ {\bf   0}}(u) du\Big), \ {\pmb \alpha}\in {\mathbb Z}_{+, k}^d.
	\end{equation}
	\end{lemma}

\begin{lemma} \label{maintheoremanalytic.lem4}
	Let  $D_0>0$ and
	$U_k, 1\le k\le N$, be nonnegative functions satisfying
	\begin{equation}\label{maintheorem.lem4.pf1}
		0\le U_k(t)\le   D_0k \int_{0}^t e^{D_0 k (t-s)} \left(\sum_{l=k+1}^N   U_l(s) +1 \right) ds,\ \  t\ge 0,
	\end{equation}
 then
	\begin{equation} \label{maintheorem.lem4.pf2}
		U_{2}(t)+\cdots+U_N(t)+1\le
		\frac{N^{N-2}}{(N-2)!}  e^{D_0Nt}, \ t\ge 0.
	\end{equation}
\end{lemma}

Now we are ready to prove Theorem \ref{maintheoremanalytic.thm1}.

\begin{proof}[Proof of Theorem \ref{maintheoremanalytic.thm1}]
Define ${\bf   u}_k(t)={\bf   v}_{k,N}(t)-{\bf   w}_k(t), 1\le k\le N$. Then one may verify that
\begin{equation}\label{maintheoremanalytic.thm1.pfeq1}
    \Dot{{\bf   u}}_k(t) 
    = {\bf   B}_{k,k}(t){\bf   u}_k(t)+\sum_{l=k+1}^N {\bf   B}_{k,l}(t){\bf   u}_{l}(t)-\sum_{l=N+1}^{\infty}{\bf   B}_{k,l}(t){\bf   w}_{l}(t),
\end{equation}
and
\begin{equation} \label{maintheoremanalytic.thm1.pfeq2}
{\bf   u}_k(0)={\bf   0}, \ 1\le k\le N.
\end{equation}
Therefore
\begin{equation} \label{maintheoremanalytic.thm1.pfeq3}
    {\bf   u}_k(t)=\int_{0}^t
{\bf   K}_k(t,s)\Big( \sum_{l=k+1}^N {\bf B}_{k,l}(s){\bf  u}_{l}(s)-\sum_{l=N+1}^{\infty} {\bf B}_{k,l}(s){\bf  w}_{l}(s)\Big) ds,
\end{equation}
by Lemma \ref{maintheoremanalytic.lem3},
where the kernel ${\bf   K}_k(t,s)$ satisfies
\begin{equation} \label{maintheoremanalytic.thm1.pfeq4}
        \|{\bf   K}_k(t,s)\|_{{\mathcal S}}
  \leq  e^{D_0 k(t-s)}
 \ \ \text{for \ all} \ \ 0\leq s\leq t
    \end{equation}
    by \eqref{maintheorem.lem3.eq3} and  Assumption \ref{assump-1}.

Let $M_0$ be as in \eqref{M0.def}.
By Lemma \ref{maintheoremanalytic.lem2}, we have
 \begin{equation} \label{maintheoremanalytic.thm1.pfeq6}
\|{\bf w}_1(t)\|_\infty\le M_0<R/e \ \ {\rm for \ all}\  0\le t\le T^*.
\end{equation}
By \eqref{Carleman.eq4b},
 \eqref{maintheoremanalytic.thm1.pfeq3},
\eqref{maintheoremanalytic.thm1.pfeq4}, \eqref{maintheoremanalytic.thm1.pfeq6}
and  the observation that
$\|{\bf   w}_l(t)\|_\infty\le \|{\bf   w}_1(t)\|_\infty^l, l\ge 1$,
we obtain
\begin{eqnarray*}
\|{\bf   u}_k(t)\|_{\infty}&\hskip-0.08in \leq & \hskip-0.08in  D_0 k
\int_{t_0}^t e^{D_0 k (t-s)}
\Bigg(\sum_{l=k+1}^N \frac{\|{\bf   u}_{l}(s)\|_\infty}{ R^{l-k}} +
\sum_{l=N+1}^{\infty} \frac{\|{\bf   w}_{l}(s)\|_{\infty}}{R^{l-k}}\Bigg)ds\nonumber\\
    &\leq &  D_0 k\int_{t_0}^te^{D_0k(t-s)}
    \Bigg(\sum_{l=k+1}^N\frac{\|{\bf   u}_l(s)\|_\infty}{R^{l-k}} +\frac{M_0^{N}}{(e-1) R^{N-k}}\Bigg)ds,
\end{eqnarray*}
where  $0\le t\le T^*$.
This implies that
\begin{equation}\label{maintheoremanalytic.thm1.pfeq7}
    U_k(t)\leq  D_0k\int_{t_0}^te^{D_0k(t-s)}\Bigg(\sum_{l=k+1}^N U_l(s)+1\Bigg)ds,
\end{equation}
where
 $U_k(t)=\frac{(e-1) R^{N}}{M_0^{N}} R^{-k}\|{\bf   u}_k(t)\|_{\infty}, 1\le k\le N$.
By  \eqref{maintheoremanalytic.thm1.pfeq7} and Lemma \ref{maintheoremanalytic.lem4}, we obtain
\begin{equation} \label{maintheoremanalytic.thm1.pfeq8}
    U_1(t) \le  (2\pi)^{-1/2} N^{-3/2} e^{D_0 Nt+N}, 0\le t\le T^*.
\end{equation}
Therefore for $0\le t\le T^*$, we get
\begin{eqnarray} \label{maintheoremanalytic.thm1.pfeq8+}
\|{\bf   v}_{1,N}(t)-{\bf   w}_1(t)\|_{\infty}
& \hskip-0.08in = & \hskip-0.08in \|{\bf   u}_1(t)\|_{\infty}=
\frac{  M_0^{N} R} {(e-1) R^{N}}  U_1(t) \nonumber\\
   & \hskip-0.08in\leq & \hskip-0.08in
  \frac{ R}{\sqrt{2\pi} (e-1)} N^{-3/2} 
e^{ D_0(t-T^*)N}.
\end{eqnarray}

Applying \eqref{assumption0} and \eqref{maintheoremanalytic.thm1.pfeq6} to the original complex dynamical system \eqref{complexdynamic.def},
we have
\begin{equation*} \label{maintheoremanalytic.thm1.pf9}
\|\Dot{\bf x}(t)\|_\infty\le   \sum_{{\pmb \alpha}\in {\mathbb Z}_+^d} \|{\bf g}_{{\pmb \alpha}}(t)\|_1 |e^{i{\pmb \alpha} {\bf   x}}|
 \le  D_0\sum_{k=0}^\infty \Big(\frac{M_0}{R}\Big)^{k}\le \frac{D_0 e}{e-1},
\end{equation*}
which implies that
\begin{equation}  \label{maintheoremanalytic.thm1.pf10}
\|({\bf w}_1(t))^{-1}\|_\infty=\big\|e^{-i{\bf x}(t)}\big\|_\infty\le \exp\Big(\max_{1\le j\le d}\Im x_{0, j}+\frac{eD_0 T^*}{e-1} \Big),  0\le t\le T^*.
\end{equation}
 This together with \eqref{maintheoremanalytic.thm1.pfeq8+} implies that
\begin{eqnarray*}
 |{ v}_{j, N}(t) e^{-ix_j(t)}-1| & \hskip-0.08in \le &
\hskip-0.08in  \frac{ R}{\sqrt{2\pi} (e-1)} N^{-3/2} \exp\big(D_0 (t-T^*)N\big)\nonumber\\
& & \times \exp\Big( \frac{eD_0T^*}{e-1} +\max_{1\le j\le d}\Im x_j(0)\Big)
\end{eqnarray*}
hold for all $0\le t\le T^*$ and $1\le j\le d$. This completes the proof.
\end{proof}

\subsection{Proof of Theorem \ref{maintheoremanalytic.thm2}}
\label{maintheoremanalytic.thm2.pfsection}

To prove Theorem \ref{maintheoremanalytic.thm2}, we need  an estimate about
  $\|{\bf   w}_1(t)\|_{\infty}, t\ge 0$. 

\begin{lemma} \label{maintheoremanalytic.thm2.lem1}
Consider
 the  complex dynamical system \eqref{dynamicsystem}
with the vector field ${\bf g}(t, {\bf   x})$ satisfying  \eqref{assumption1}, \eqref{assumption0} and \eqref{assumption2},
and the initial  ${\bf   x}_0$
 satisfying \eqref{maintheoremanalytic.thm2.eq1}.
Let ${\bf   x}=[x_1(t), \ldots, x_d(t)]^T$ be the solution of the nonlinear dynamical system \eqref{complexdynamic.def},
and  $u(t), t\ge 0$, be as in \eqref{ut.def}.
 Then
  \begin{equation}\label{maintheoremanalytic.thm2.lem1.eq1}
    u(t)
    \leq u(0) \exp\Big(-\Big(\mu_0-\frac{D_0u(0)}{R-u(0)}\Big)t\Big),  \ \ t\ge 0.
\end{equation}
\end{lemma}

\begin{proof}
By the continuity of the function $u$ about $t$, there exists $\delta>0$ such that
\begin{equation}
\label{maintheoremanalytic.thm2.lem1.pf.eq1}
u(t)<\frac{\mu_0 R}{D_0+\mu_0}, \ 0\le t\le \delta.
\end{equation}
Then  by \eqref{complexdynamic.def}, \eqref{maintheoremanalytic.thm2.eq1} and \eqref{maintheoremanalytic.thm2.lem1.pf.eq1},  we have
\begin{eqnarray} \label{maintheoremanalytic.thm2.lem1.pf.eq2}
    \frac{d}{dt}u(t)^2 
    & \hskip-0.08in =& \hskip-0.08in  -2 \sum_{j=1}^d |e^{ix_j(t)}|^2
    \Im \Big(  g_{j, {\bf   0}}(t)+
    \sum_{{\pmb \beta} \in \mathbb{Z}^d_{++}}g_{j,{\pmb \beta}}(t){  w}_{{\pmb \beta}}(t)\Big)\nonumber\\
    & \hskip-0.08in \le & \hskip-0.08in -2\mu_0 u(t)^2+ 2  u(t)^2  \frac{D_0 u(t) }{R-u(t)} <0,\ \ 0\le t\le \delta.\qquad
\end{eqnarray}
Applying the above procedure repeatedly,  we conclude that
\begin{equation}
\label{maintheoremanalytic.thm2.lem1.pf.eq3}
u(t)\le u(0)<\frac{\mu_0 R}{D_0+\mu_0}, \ \  t\ge 0.
\end{equation}
Using the bound estimate in \eqref{maintheoremanalytic.thm2.lem1.pf.eq3} and following the similar argument used to establish
 \eqref{maintheoremanalytic.thm2.lem1.pf.eq2}, we obtain
\begin{equation} \label{maintheoremanalytic.thm2.lem1.pf.eq4}
    \frac{d}{dt}u(t)^2
     \le  2\Big(-\mu_0 +  \frac{D_0 u(0) }{R-u(0)}\Big) u(t)^2, \ t\ge 0.
\end{equation}
Dividing $u(t)^2$ at both sides of the above inequality and then integrating on the interval $[0, t]$ completes the proof.
\end{proof}

\begin{proof}
[Proof of Theorem \ref{maintheoremanalytic.thm2}]
Define ${\bf   u}_k(t)={\bf   v}_{k,N}(t)-{\bf   w}_k(t), 1\le k\le N$.  Following the argument in
the proof of Theorem \ref{maintheoremanalytic.thm1},
we see that
${\bf   u}_k, 1\le k\le N$, satisfy  %
\begin{equation} \label{maintheoremanalytic.thm2.pfeq1}
   {\bf   u}_k(t)=\int_{0}^t
{\bf   K}_k(t,s)
\Big(\sum_{l=k+1}^N {\bf   B}_{k,l}(s){\bf   u}_{l}(s)-\sum_{l=N+1}^{\infty}{\bf   B}_{k,l}(s){\bf   w}_{l}(s)\Big)
ds, \ \ t\ge 0,
\end{equation}
where $ {\bf   K}_k(t, s), t\ge s\ge 0$, is the kernel function in Lemma \ref{maintheoremanalytic.lem3}.
By \eqref{assumption2} and \eqref{maintheorem.lem3.eq3}, we see that for $0\le s\le t$,
\begin{equation} \label{maintheoremanalytic.thm2.pfeq2}
    \|{\bf   K}_k(t,s)\|_{{\mathcal S}}
   \leq  \exp\Big( \max_{|{\pmb \alpha}|=k} -\sum_{j=1}^d \alpha_j \int_s^t \Im g_{j, {\bf   0}}(u) du\Big)
   \leq \exp(-k\mu_0(t-s)),
             \end{equation}
 where ${\pmb \alpha}=[\alpha_1, \ldots, \alpha_d]^T\in {\mathbb Z}_{+, k}^d$.
Therefore
\begin{eqnarray} \label{maintheoremanalytic.thm2.pfeq3}
R^{-k}\|{\bf   u}_k(t)\|_{\infty}&\hskip-0.08in \leq & \hskip-0.08in \int_{0}^t e^{-k\mu_0(t-s)}
\Big(\sum_{l=k+1}^N D_0 k R^{-l}\|{\bf   u}_l(s)\|_{\infty}
+
\sum_{l=N+1}^{\infty}  D_0 kR^{-l}\|{\bf   w}_{l}(s)\|_{\infty}\Big)ds\nonumber\\
& \hskip-0.08in \leq & \hskip-0.08in  D_0 k\int_{0}^t e^{-k\mu_0(t-s)}
\Big(\sum_{l=k+1}^N  R^{-l}\|{\bf   u}_l(s)\|_{\infty} + \frac{\mu_0}{ D_0} \Big(\frac{\|\exp(i{\bf x}_0)\|_2}{R}\Big)^{N}
\Big)ds,
\end{eqnarray}
where the first estimate holds by \eqref{Carleman.eq4b}, \eqref{maintheoremanalytic.thm2.pfeq1} and \eqref{maintheoremanalytic.thm2.pfeq2}, 
and the second inequality follows from \eqref{maintheoremanalytic.thm2.eq1}  and the observation that
$$ \|{\bf   w}_l(s)\|_{\infty}\le \|{\bf   w}_1(s)\|_{\infty}^l\le
\|{\bf   w}_1(s)\|_2^l\le \|{\bf   w}_1(0)\|_2^l=\|\exp(i{\bf x}_0)\|_2^l, \ s\ge 0
$$
by Lemma \ref{maintheoremanalytic.thm2.lem1}.
Applying \eqref{maintheoremanalytic.thm2.pfeq3} repeatedly, we may show
$$
\sum_{l=k}^N  R^{-l}\|{\bf   u}_l(s)\|_{\infty} + \frac{\mu_0}{ D_0} \Big(\frac{\|\exp(i{\bf x}_0)\|_2}{R}\Big)^{N}
\le \frac{\mu_0}{D_0}
\Big(\frac{D_0+\mu_0}{\mu_0}\Big)^{N+1-k}
 \Big(\frac{\|\exp(i{\bf x}_0)\|_2}{R}\Big)^{N}$$
by induction on $k=N, \ldots, 1$.
Taking $k=1$ in the above estimate proves the desired conclusion \eqref{maintheoremanalytic.thm2.eq02}.
\end{proof}

\subsection{Proof of \eqref{comparison.eq00}} \label{comparison.eq00.pfsection}  The first inequality in \eqref{comparison.eq00} holds  as
$ (\ln R)^2<2 R$ and
\begin{eqnarray*}
& &  \inf_{1\le \|{\bf x}_0\|_\infty<\ln R/e}
\frac{\ln R-1- \max(\ln \|\exp(i{\bf x}_0)\|_\infty,\ln  \|\exp(-i{\bf x}_0)\|_\infty)}{\ln \ln R- \ln \|{\bf x}_0\|_\infty-1}\\
& \ge&  \inf_{u>0, \|{\bf x}_0\|_\infty\ge 1}
\frac{e^{u+1}\|{\bf x}_0\|_\infty -1-\|{\bf x}_0\|_\infty}{u}=\inf_{u>0} \frac{e^{u+1}-2}{u}\approx 4.9215>4.
\end{eqnarray*}

The second estimate in \eqref{comparison.eq00} follows as $g(t):=\ln r_C(t)-\ln \tilde r_{CF}(t)$ is a linear function about $t$ satisfying
$g(T_C^*)=-\ln \tilde r_{CF}(T_C^*)>0$ and
$$g(0)\ge \frac{e-1}{2e-1} \Big(\ln \|{\bf x}_0\|_\infty-\ln\ln R+\ln R- \|{\bf x}_0\|_\infty\Big)>0.$$

 \subsection {Proof of  \eqref{timerange.inequality}} \label{timerange.inequality.pfsection}
With the substitution of $t\cos \phi$ by $s$ and $\tan \phi$ by $u$, it suffices to show that
 \begin{equation}\label{timerange.inequality2}
 \frac{u^2}{1+u^2} \sup_{s\ge 0} (e^{-2su}-2 e^{-su} \cos s+1)<1, \ \ u>0.
 \end{equation}
 Observe that
 $$  \frac{u^2}{1+u^2}  \sup_{0\le s\le \pi/3}   (e^{-2su}-2 e^{-su} \cos s+1)\le  \frac{u^2}{1+u^2} <1,$$
as $2\cos s\ge 1$ for all $0\le s\le \pi/3$,
 \begin{eqnarray*}
 \frac{u^2}{1+u^2}\sup_{s\ge 2\sqrt{3}/e} (e^{-2su}-2 e^{-su} \cos s+1) &\hskip-0.08in  < & \hskip-0.08in
 \frac{u^2}{1+u^2} \sup_{s\ge 2\sqrt{3}/e} (3 e^{-su}+1)\nonumber\\
 & \hskip-0.08in = & \hskip-0.08in 1+\frac{3u^2\exp(-2\sqrt{3} u/e)-1}{1+u^2}\le 1,
 \end{eqnarray*}
since the function $3u^2\exp(-2\sqrt{3} u/e), 0<u<\infty$, attains its maximal value $1$ at $u=e/\sqrt{3}$, and
 \begin{eqnarray*}
 &\hskip-0.08in    & \hskip-0.08in    \frac{u^2}{1+u^2}\sup_{ \pi/3\le s\le 2\sqrt{3}/e} (e^{-2su}-2 e^{-su} \cos s+1)\nonumber\\
 &\hskip-0.08in  \le  & \hskip-0.08in
 \frac{u^2}{1+u^2} \sup_{ \pi/3\le s\le 2\sqrt{3}/e}\big( (1-2\cos (2\sqrt{3}/e)) e^{-su}+1\big)\nonumber\\
 &\hskip-0.08in  \le  & \hskip-0.08in 1+  \frac{(1-2\cos (2\sqrt{3}/e)) u^2e^{-\pi u/3}-1}{1+u^2}
 < 1
  \end{eqnarray*}
because the function    $(1-2\cos (2\sqrt{3}/e)) u^2e^{-\pi u/3}, 0<u<\infty$, attains its maximal value
$(1-2\cos (2\sqrt{3}/e)) (6/\pi)^2 e^{-2}\approx 0.2053<1$ at the value $6/\pi$.
Combining the above three estimates proves \eqref{timerange.inequality2} and hence
the desired result \eqref{timerange.inequality} for the exponential convergence time range.

\section{Conclusion}
This paper introduces the Carleman-Fourier linearization method, extending traditional Carleman linearization, to complex nonlinear dynamical systems with periodic vector fields and multiple fundamental frequencies. By leveraging Fourier basis functions, this approach achieves a sparse representation of periodic vector fields, effectively capturing both periodic and nonlinear behaviors. The method transforms the system into an infinite-dimensional linear model, with finite-section approximations providing exponential convergence to the original system's state vector. Explicit error bounds are established, demonstrating the accuracy of approximations over larger regions and extended time horizons, especially near equilibrium points. The framework is applicable to systems with analyticity conditions on their vector fields and is extended to handle cases where these conditions are not strictly satisfied, such as in the Kuramoto model. The Carleman-Fourier linearization outperforms traditional methods in terms of precision and convergence, particularly for systems with exponentially decaying Fourier coefficients. These improvements enable robust analyses and reliable long-term predictions, which are crucial for applications like model predictive control, safety verification, and quantum computing.


\begin{thebibliography}{99}

\bibitem{abdia2023} M. Abudia, J. A. Rosenfeld and R. Kamalapurkar, Carleman lifting for nonlinear system identification with guaranteed error bounds,
{\em 2023 American Control Conference (ACC), San Diego, CA, USA}, 2023, pp. 929--934.  


\bibitem{Amini2022} A. Amini,  C. Zheng, Q. Sun and N. Motee,
 Carleman linearization of nonlinear systems and its finite-section approximations,
 {\em Discrete and Continuous Dynamical Systems Series B}, 
 30(2), 2025, pp. 577--603.


\bibitem{amini2021error}
A. Amini, Q. Sun and N. Motee, Error bounds for Carleman linearization of general nonlinear
systems, {\em 2021 Proc. the Conference on Control and its Applications}, SIAM, 2021, pp. 1--8.

\bibitem{amini2020approximate}
A. Amini,  Q. Sun and N. Motee,
 Approximate optimal control design for a class of nonlinear systems by lifting Hamilton-Jacobi-Bellman equation,
{\em 2020 American Control Conference (ACC)}, IEEE, 2020, pp. 2717--2722.

\bibitem{amini2020quadratization}
A. Amini, Q. Sun and N. Motee, Quadratization of Hamilton-Jacobi-Bellman equation for near-optimal control of nonlinear systems,
 {\em 59th IEEE Conference on Decision and Control (CDC)}, IEEE, 2020, pp. 731--736.

\bibitem{Akiba2023}  T.
Akiba, Y.  Morii  and K.  Maruta, Carleman linearization approach for chemical kinetics integration toward quantum computation,
{\em Scientific Reports}, 13(1), 2023, pp. 3935.

\bibitem{brockett2014early}
R. Brockett,
The early days of geometric nonlinear control,
{\em Automatica}, 50(9), 2014, pp. 2203--2224.


\bibitem{bronski2021}
J. C. Bronski, T. E. Carty and L. DeVille, Synchronisation conditions in the Kuramoto model
and their relationship to seminorms, {\em Nonlinearity}, 34(8), 2021, pp. 5399--5433.




\bibitem{Brunton2016} S. L.
Brunton,  J. L.  Proctor  and J. N.  Kutz,
 Discovering governing equations from data by sparse identification of nonlinear dynamical systems,
{\em Proceedings of the National Academy of Sciences}, 113(15), 2016, pp. 3932--3937.

\bibitem{dietert2016} H. Dietert, Stability and bifurcation for the Kuramoto model, {\em Journal de Math\'{e}matiques Pures et Appliqu\'{e}es},
105(4), 2016, pp. 451--489.






\bibitem{forets2017explicit} M.
Forets and A.  Pouly,
Explicit error bounds for Carleman linearization, 2017,
{\em arXiv preprint arXiv:1711.02552}.

\bibitem{forets2021reachability} M.
Forets and C.  Schilling, Reachability of weakly nonlinear systems using Carleman linearization,
{\em International Conference on Reachability Problems},  Springer, 2021, pp. 85--99.

\bibitem{guo2021}
Y. Guo, D. Zhang, Z. Li, Q. Wang and D. Yu, Overviews on the applications of the Kuramoto
model in modern power system analysis, {\em International Journal of Electrical Power \& Energy Systems}, 129, 2021,  article no. 106804.




\bibitem{hashemian2015fast} N.
Hashemian and A.  Armaou,
Fast moving horizon estimation of nonlinear processes via Carleman linearization,
 {\em 2015 American Control Conference (ACC)}, IEEE, 2015, pp. 3379--3385.


\bibitem{heggli2019}
O. A. Heggli, J. Cabral, I. Konvalinka, P. Vuust and M. L. Kringelbach, A Kuramoto model of
self-other integration across interpersonal synchronization strategies, {\em PLoS Computational Biology},
15(10), 2019, article no. e1007422, 17pp.



\bibitem{ji2014}
P. Ji, T. K. Peron, F. A. Rodrigues and J. Kurths, Low-dimensional behavior of
Kuramoto model with inertia in complex networks, {\em Scientific Reports}, 4(1), 2014, article no.  4783.






\bibitem{Korda2018} M. Korda  and I. Mezi\'{c}, Linear predictors for nonlinear dynamical systems: Koopman operator meets model predictive control,
{\em Automatica}, 93, 2018, pp. 149--160.

\bibitem{KordaMezic2018}  M.
Korda and I.  Mezi\'{c},
On convergence of extended dynamic mode decomposition to the Koopman operator,
{\em Journal of Nonlinear Science}, 28, 2018, pp. 687--710.

\bibitem{kowalski1991nonlinear} K. Kowalski  and W.H. Steeb,
 Nonlinear dynamical systems and Carleman linearization,  World Scientific, 1991.



\bibitem{krener1974linearization}  A. J.
Krener,
Linearization and bilinearization of control systems,
{\em Proceedings of the 1974 Allerton Conference on Circuit and Systems Theory, Urbana III}, 1974.

\bibitem{krener1975bilinear} A. J. Krener,
Bilinear and nonlinear realizations of input-output maps,
{\em SIAM Journal on Control}, 13(4), 1975, pp. 827--834.

\bibitem{kuramoto1984}
Y. Kuramoto, Chemical oscillations, waves, and turbulence, New York, Springer-Verlag,
1984. 

\bibitem{liu2021efficient} J.P. Liu, H. O.  Kolden, H. K.  Krovi, N. F.  Loureiro, K. Trivisa and A. W. Childs,
Efficient quantum algorithm for dissipative nonlinear differential equations,
{\em Proceedings of the National Academy of Sciences}, 118(35), 2021, article no.  e2026805118.

\bibitem{loparo1978estimating} K.
Loparo and G.  Blankenship,
Estimating the domain of attraction of nonlinear feedback systems,
{\em IEEE Transactions on Automatic Control}, 23(4), 1978, pp. 602--608.

\bibitem{minisini2007carleman}  J. Minisini, A.  Rauh and  E. P. Hofer,
Carleman linearization for approximate solutions of nonlinear control problems: Part 1--theory,
{\em Proc. of the 14th Intl. Workshop on Dynamics and Control}, 2007, pp. 215--222.


\bibitem{moteesun2024} N. Motee and Q. Sun, On exponential convergence of Carleman-Fourier
linearization of nonlinear real dynamical systems, {\em In prepartion}.


\bibitem{pruekprasert2022moment}  S.
Pruekprasert,  J.  Dubut, T.  Takisaka, C.  Eberhart  and  A. Cetinkaya,
Moment propagation through Carleman linearization with application to probabilistic safety analysis, 2022,
{\em arXiv preprint arXiv:2201.08648}.

\bibitem{rauh2009carleman}  A. Rauh, J.  Minisini and H. Aschemann,
Carleman linearization for control and for state and disturbance estimation of nonlinear dynamical processes,
{\em IFAC Proceedings Volumes},  42(13), 2009, pp. 455--460.

\bibitem{rotondo2022towards}  D.
Rotondo, G. Luta and J. H. U. Aarvag,
Towards a Taylor-Carleman bilinearization approach for the design of nonlinear state-feedback controllers,
{\em European Journal of Control 68}, 2022, article no.  100670.


\bibitem{steeb1980non}
W. H. Steeb and  F. Wilhelm,
Non-linear autonomous systems of differential equations and Carleman linearization procedure,
{\em Journal of Mathematical Analysis and Applications}, 77(2), 1980, pp. 601--611.  

\bibitem{surana2024} A. Surana, A. Gnanasekaran  and  T. Sahai,
An efficient quantum algorithm for simulating polynomial dynamical systems,  {\em Quantum Information Processing},  23(3), 2024, article no. 105, 22pp. 

\bibitem{WangJungersOng2023} Z.
Wang, R. M. Jungers and C. J. Ong, Computation of invariant sets via immersion for discrete-time nonlinear systems,
{\em Automatica}, 147, 2023, article no.  110686, 9pp.

\bibitem{Wu2024} H. C. Wu, J. Wang and  X. Li, Quantum algorithms for nonlinear dynamics: revisiting Carleman linearization with no dissipative conditions, 2024, {\em arXiv preprint  arXiv:2405.12714}.


	\end{thebibliography}
\end{document}